\documentclass[11pt, letterpaper]{amsart}
\pdfoutput=1
\usepackage{amssymb,latexsym, xcolor
}
\usepackage{amsthm,amssymb}
\usepackage[dvips, hcentering,
includehead,width=14.2cm,
top=0.5cm,
height=21.6cm,
paperheight=23.8cm,paperwidth=17cm
]{geometry}

\usepackage{hyperref}
\usepackage{enumitem}
\usepackage{pict2e}
\usepackage[utf8]{inputenc}
\usepackage{dsfont}

\usepackage{tikz-cd}
\usetikzlibrary{decorations.markings} 

\usepackage{mathtools}
\newcommand{\wiggleQ}{Q \hspace{-7pt} Q}

\usepackage{ccaption}
\changecaptionwidth
\captionwidth{5.6in}

\usepackage{ytableau}

\setitemize{leftmargin=.2in}
\setenumerate{leftmargin=.3in}

\definecolor{darkred}{rgb}{1,0,0}        
\definecolor{lightred}{rgb}{1,0.4,0}     
\definecolor{darkblue}{cmyk}{1,0.4,0,0.4}  
\definecolor{lightblue}{cmyk}{1,0.4,0,0}  
\definecolor{darkgreen}{cmyk}{1,0.5,1,0}  
\definecolor{lightgreen}{cmyk}{1,0,1,0}  
\definecolor{lightgray}{rgb}{0.5,0.5,0.5}  

\newcommand{\red}[1]{{\color{red} #1\color{black}}}

\newcommand{\lightblue}[1]{\color{lightblue}{#1}\color{black}}

\newcommand{\lightgray}[1]{\color{lightgray}{#1}\color{black}}

\newcommand{\cS}[1]{{\noindent\textsf{\color{purple}$\blacksquare$~#1~$\blacksquare$}}}
\newcommand{\hide}[1]{}

\newtheorem{theorem}{Theorem}[section]
\newtheorem{proposition}[theorem]{Proposition}
\newtheorem{conjecture}[theorem]{Conjecture}
\newtheorem{corollary}[theorem]{Corollary}
\newtheorem{lemma}[theorem]{Lemma}

\theoremstyle{definition}

\newtheorem{remark}[theorem]{Remark}

\newtheorem{example}[theorem]{Example}
\newtheorem{definition}[theorem]{Definition}
\newtheorem{problem}[theorem]{Problem}
\newtheorem{observation}[theorem]{Observation}
\newtheorem{algorithm}[theorem]{Algorithm}
\numberwithin{equation}{section}

\def\ZZ{{\mathbb Z}}
\def\RR{{\mathbb R}}
\def\QQ{{\mathbb Q}}

\def\CC{{\mathbb C}}

\def\O{{\mathcal{O}}}

\def\dd{{\mathbf{d}}}

\def\wind{{\operatorname{wind}}}

\def\GL{{\operatorname{GL}}}
\def\In{{\operatorname{In}}}
\def\Out{{\operatorname{Out}}}
\newcommand\LOQ{\operatorname{LOQ}}

\newcommand\Pf{\operatorname{Pf}}
\newcommand\sgn{\operatorname{sgn}}
\newcommand\GQ{K_Q}

\title{Cyclically ordered quivers}

\begin{document}

\author{Sergey Fomin}
\address{\hspace{-.3in} Department of Mathematics, University of Michigan,
Ann Arbor, MI 48109, USA}
\email{fomin@umich.edu}
\author{Scott Neville}

\email{nevilles@umich.edu}

\date{August 13, 2025. Updated May 15, 2026.}

\thanks{Partially supported by NSF grants DMS-2054231, DMS-2348501 (S.~F.) and DMS-1840234 (S.~N.),
Rackham Predoctoral Fellowship (S.~N.), BSF grant 2022157 (S.~F.), and a Simons Fellowship (S.~F.).
}



\begin{abstract}
A cyclically ordered quiver is a quiver endowed with an additional structure of a cyclic ordering of its vertices. 
This structure, which naturally arises in many important applications, 
gives rise to new powerful mutation invariants. 
\end{abstract}

\ \vspace{-10pt}

\maketitle


\thispagestyle{empty}

\begin{center}
{
\setlength{\unitlength}{.7pt}
\newcommand{\radius}{1.5cm} 
\newcommand{\vertsize}{1.5pt}
\begin{tikzpicture}
\filldraw[black] (0,0)++(180:\radius) circle (2pt) node[left] {} coordinate (a);
\filldraw[black] (0,0)++(120:\radius) circle (2pt) node[above left] {} coordinate (b);
\filldraw[black] (0,0)++(60:\radius) circle (2pt) node[above right] {} coordinate (c);
\filldraw[black] (0,0)++(0:\radius) circle (2pt) node[right] {} coordinate (d);
\filldraw[black] (0,0)++(-60:\radius) circle (2pt) node[below right] {} coordinate (e);
\filldraw[black] (0,0)++(-120:\radius) circle (2pt) node[below left] {} coordinate (f);
\draw[black, dashed, decoration={markings, mark=at position 0.1 with {\arrow{<}}}, postaction={decorate}, opacity=0.5] (0,0) circle (\radius);
\draw[black, -{stealth}, thick, shorten >=4pt, shorten <= 4pt] (a) -- (b);
\draw[black, -{stealth}, thick, shorten >=4pt, shorten <= 4pt] (b) -- (c);
\draw[black, -{stealth}, thick, shorten >=4pt, shorten <= 4pt] (c) -- (d);
\draw[black, -{stealth}, thick, shorten >=4pt, shorten <= 4pt] (d) -- (e);
\draw[black, -{stealth}, thick, shorten >=4pt, shorten <= 4pt] (e) -- (f);
\draw[black, -{stealth}, thick, shorten >=4pt, shorten <= 4pt] (f) -- (a);

\draw[black, -{stealth}, thick, shorten >=4pt, shorten <= 4pt] (c) -- (a);
\draw[black, -{stealth}, thick, shorten >=4pt, shorten <= 4pt] (e) -- (a);
\draw[black, -{stealth}, thick, shorten >=4pt, shorten <= 4pt] (b) -- (d);
\draw[black, -{stealth}, thick, shorten >=4pt, shorten <= 4pt] (f) -- (d);

\end{tikzpicture}
}
\end{center}

\setcounter{tocdepth}{1}
\tableofcontents
\vspace{-25pt}

\pagebreak[3]

Quiver mutations form the combinatorial backbone of the theory of \emph{cluster algebras}~\cite{fwz1-3,ca1}. 
Despite the elementary nature of quiver mutation, many basic questions concerning this notion remain open.
There is no known algorithm for detecting mutation equivalence of quivers. 
The dearth of known invariants of quiver mutation 
makes it difficult to determine, or even guess, whether two particular cluster structures,  
perhaps arising in different mathematical contexts, have the same mutation~type.
Identification of such coincidences is often indicative of a deeper connection, and is 
one of the key benefits that the theory of cluster algebras~supplies. 

In this paper, we develop an approach to quiver combinatorics 
that provides new powerful tools for detecting mutation inequivalence of quivers. 
This approach is~based on   
endowing a quiver~$Q$ with additional combinatorial structure:
a \emph{cyclic ordering}~$\sigma$ of its vertices, 
yielding a \emph{cyclically ordered quiver (COQ)} $(Q,\sigma)$. 

In the framework of cyclically ordered quivers, 
the usual notion of quiver mutation at a vertex~$v$ gets upgraded 
to the operation of \emph{proper mutation} of COQs 
that changes both the quiver and the cyclic ordering. 
This operation is only defined when a COQ $(Q,\sigma)$ satisfies a simple local condition: 
informally speaking, every~oriented path 
passing through~$v$ must make a right turn at~$v$.
(We then say that vertex~$v$ is \emph{proper}.) 


Tearing up a cyclic ordering~$\sigma$ into a linear ordering of the vertices of a quiver~$Q$ allows us to 
associate to a COQ $(Q, \sigma)$ a unipotent upper-triangular matrix~$U$,  
the \emph{unipotent companion} of~${(Q,\sigma)}$. 
We  show that the \emph{integral congruence class} 
${\{G U G^T \mid G\in\GL(n,\ZZ)\}}$
of this matrix (here $G^T$ denotes the transpose of~$G$) does not depend on the choice of a tearing point,
nor does it change under \emph{wiggles}, the local transformations of a cyclic ordering that
exchange consecutive vertices that are not adjacent in the quiver. 
Crucially, the integral congruence class of a unipotent companion is invariant under proper mutations.

The integral congruence class of~$U$ 
gives rise to an arguably
more useful invariant of a COQ $(Q, \sigma)$: 
the $\GL(n,\ZZ)$ conjugacy class of the \emph{cosquare} $U^{-T} U$. 
This conjugacy class, in turn, determines the \emph{Alexander lattices}
(which capture a lot of number-theoretic information about the conjugacy class) 
and the \emph{Alexander polynomial} $\Delta_Q(t)$, 
the monic characteristic polynomial of the cosquare. 
(For a quiver associated with a planar divide~\cite{FPST}, 
this polynomial agrees with the Alexander polynomial of the corresponding link.) 
All of the above are invariant under proper mutations. 

The coefficient of $t$  in the Alexander polynomial $\Delta_Q(t)$ 
gives rise to the \emph{Markov invariant}, \linebreak[3]
which generalizes a mutation invariant of 3-vertex quivers introduced in~\cite{BBH}. 
Other known mutation invariants of quivers \cite{casals-binary, fwz1-3}  
can also be derived from the Alexander polynomial. 

\enlargethispage{5pt}

Figure~\ref{fig:hierarchy} provides an overview of the invariants of proper mutations discussed
in this paper, the functional dependencies between them, and their connections to known  
invariants of quiver mutations. 

\pagebreak[3]

\begin{figure}[ht]
\begin{center}
\vspace{10pt}
\begin{tikzcd}[arrows={-stealth}, sep=1.5em]
\hline
\\[-15pt]
\text{\textbf{invariants of COQs}} & \text{\textbf{invariants of quivers}} \\[-15pt]
\hline
\\[-15pt]
\begin{array}{c}
\text{proper mutation equivalence class} \\
\text{of a cyclically ordered quiver~$(Q,\sigma)$}  
\end{array} \arrow[d]   \\[-5pt]
\text{signed braid group orbit of $U$} \arrow[r]\arrow[d] & \text{gcd's of multiplicities \cite{Seven3x3}}
\\[-5pt]
\text{integral congruence class of $U$}
\arrow[d] \arrow[r] & \text{Seven-\"Unal invariants \cite{Seven-congruence}}
\\[-5pt]
\text{$\operatorname{GL}_n(\ZZ)$ conjugacy class  of $U^{-T}U$}\arrow[d] \\[-5pt]
\text{Alexander lattices $\dd_k(Q,t)$}
\arrow[d]\arrow[r] & 
\operatorname{rank}(B)
\\[-5pt]
\text{Alexander polynomial} \arrow[d] \arrow[r] \arrow[rd, start anchor=-20, end anchor=west] & 
\det(B) 
\\[-5pt]
\text{Markov invariant} & \text{Casals' invariant \cite{casals-binary}}
\\[-10pt]
\hline
\end{tikzcd}
\end{center}
\vspace{3pt}
\caption{Hierarchy of equivalence classes and invariants. 
Here ${U}$ (resp.,~$B$) is the unipotent companion (resp., the exchange matrix) of~$Q$. 
Arrows indicate functional dependencies; 
thus, the Alexander polynomial  determines the Markov invariant,~etc. 
}
\label{fig:hierarchy}
\end{figure}

A cyclically ordered quiver is \emph{proper} if all its vertices are proper.
Proper COQs are especially nice because any single mutation leaves all our invariants intact.

Examples of quivers that possess a proper cyclic ordering 
%
include all quivers on $\le3$~vertices, all quivers of finite type, and all acyclic quivers. 
Moreover, every quiver $Q$ belonging to one of these three classes allows a 
\emph{totally proper} cyclic ordering~$\sigma$:
any sequence of (proper) mutations transforms $(Q,\sigma)$ into a proper~COQ. 
(Total properness of all acyclic quivers has been recently established by the second author~\cite{SN-acyclic-TP}.)

\pagebreak[3]

A totally proper cyclic ordering of a given quiver, if it is known to exist, is necessarily unique (up to wiggles) 
and can be efficiently computed, see Algorithm~\ref{alg:sigma-Q}. 

We show that the class of totally proper COQs is rather broad 
and includes many examples beyond the ones mentioned above. 
Within this class, 
each of the aforementioned invariants of proper mutations
becomes a fully-fledged mutation invariant. 

In a forthcoming work, we will show that quiver mutations corresponding to square moves 
in reduced plabic graphs are proper, for a suitable cyclic order; 
this statement extends to many non-reduced plabic graphs. 
More generally, many important classes of quivers arising in applications of Lie-theoretic nature
appear to always come equipped with a naturally defined proper cyclic ordering.
Each time, proper mutations seem to be sufficient to produce a generating set 
of the corresponding cluster algebra, so the latter can in fact be defined
within the framework of proper mutations. 

To be sure, there are many quivers for which no proper cyclic ordering exists, 
and consequently the tools developed in this paper do not apply. 
These quivers however do not seem to arise ``in nature,'' i.e., in important applications of cluster theory. 


To summarize, the machinery of proper mutations of cyclically ordered quivers \linebreak[3]
constitutes,
in our opinion, a useful upgrade of the traditional combinatorics of quiver mutations.  
It in particular provides new powerful mutation invariants that will hopefully  
prove effective in future applications of this theory. 

\newpage

\subsection*{Structure of the paper}
Section~\ref{sec:defs} reviews basic background on quivers and their mutations. 
Cyclically ordered quivers (COQs) and wiggles in them are introduced in Section~\ref{sec:cyclic orderings}.  
We define the \emph{winding numbers}, a family of wiggle invariants associated to cycles in a~COQ, 
characterize wiggle equivalence in terms of these numbers (Theorem~\ref{th:wiggle/winding}), 
and sketch an algorithm for constructing a cyclic ordering with given winding numbers (Theorem~\ref{thm:construct an ordering}). 
The proofs of both theorems are given in Section~\ref{sec:proof-wiggle/winding}. 
In Section~\ref{sec:unipotent companions}, we define a unipotent companion of a linearly ordered quiver 
and show that its integral congruence class is invariant under wiggles and cyclic reorderings.

The notion of a proper vertex in a COQ is introduced in Section~\ref{sec:proper verts}.
In Section~\ref{sec:proper mutations}, we define proper mutations,  
verify that they are well defined modulo wiggles,~and 
present examples of proper mutation equivalence classes of quivers of finite mutation type.
In Section~\ref{sec:invariants-of-proper-mutations}, we prove that proper mutations preserve 
the integral congruence class of a unipotent companion 
(Theorem~\ref{thm:I-N-action}). 
Section~\ref{sec:alexander-polynomials} is devoted to Alexander lattices and Alexander poly\-nomials. 
These notions are illustrated with examples of quivers on $\le 4$ vertices (Section~\ref{sec:alex-examples})
and quivers whose underlying graph is a tree (Section~\ref{sec:tree-quivers}). 

Section~\ref{sec:proper COQs} focuses on proper COQs, i.e., the ones in which every vertex is~proper. 
In Theorem~\ref{th:small winding proper}, we characterize this property in terms of subquivers supported on chordless cycles. 
One important class of proper COQs are the quivers associated with planar \emph{divides}.
Alexander polynomials of those quivers coincide with Alexander polynomials of the corresponding divide links. 

In Section~\ref{sec:opposite}, we explain that the theory of proper mutations of COQs can be developed 
modulo identification of every quiver with its \emph{opposite}.
We conjecture that for $3$-vertex quivers, mutation equivalence up to taking opposites is equivalent to 
integral conjugacy of cosquares (Conjecture~\ref{conj:3 vert mu equiv}). 

\emph{Vortices} are complete 4-vertex quivers that contain an oriented 3-cycle but not an oriented 4-cycle. 
A proper COQ must be vortex-free (Corollary~\ref{cor:vortex improper}). 
We show that for complete COQs, properness propagates under mutations as long as vortices do not emerge
(Proposition~\ref{pr:propagation-under-VF}).

In~Sections~\ref{sec:totally proper}--\ref{sec:totally proper examples}, 
we study totally proper COQs. 
We prove that a totally proper cyclic ordering is unique up to wiggles 
(Theorem~\ref{th:totally-proper-unique}),
describe algorithms for constructing this ordering (assuming it exists), 
exhibit several families of totally proper quivers,
and discuss the problem of testing mutation equivalence of such quivers. 
%
In Section~\ref{sec:admissible cartan}, we show that totally proper cyclic orderings give rise to 
admissible \emph{quasi-Cartan companions}~\cite{SevenAMatrices}, 
and discuss quivers associated with \emph{triangulated surfaces}. 

In Sections~\ref{sec:bondal}--\ref{sec:bondal-to-proper-mut}, we make a connection with 
the well-studied action of the (signed) braid group on unipotent upper-triangular integer matrices 
(equivalently, on linearly ordered quivers) \cite{Bondal, BP, CV, Dub, Rudakov}. 
We show that proper mutations and wiggles can be interpreted as instances of this action 
(Theorem~\ref{th:braid-vs-COQ-mut}) and describe each orbit as the set of quivers
related to each other by proper mutations, wiggles, and vertex reversals. 


\subsection*{Acknowledgments}

We thank Roger Casals, Michael Shapiro, and Hugh Thomas for stimulating discussions, Danielle Ensign for software assistance/consulting,
and Jineon Baek for sharing a construction that inspired Definitions~\ref{def:algorithm for wiggle equiv} and~\ref{def:collisions}. 
\hbox{Eugenii} Shustin pointed us towards Example~\ref{eg:divides-collide}.
We used \textsc{Magma}, \textsf{Maple}, and \texttt{Sage} in our computations. 


\newpage

\section{Preliminaries on quivers and their mutations}
\label{sec:defs}

This section contains basic definitions pertaining to quivers and their mutations. 
The only (potentially) non-standard notions are those of a complete (resp., abundant) quiver, 
see Definition~\ref{def:complete}, and of the unoriented simple graph of a quiver, see Definition~\ref{def:Graph-Q}.

\begin{definition} 
A \emph{quiver} is a directed graph without loops or oriented $2$-cycles. Directed edges in a quiver are called \emph{arrows}. Multiple arrows are allowed. 
We indicate multiplicities by labeling the arrows. See Figure~\ref{fig:quiver-example}. 
\end{definition}

\begin{remark} 
By default, all quivers considered in this paper have \underline{no frozen vertices}. 
\end{remark}

\begin{remark} 
\label{rem:labeled-quivers}
Unless specified otherwise, we always work with \emph{labeled} quivers. 
In particular, we distinguish between isomorphic quivers on the same set of vertices. 
\end{remark}
 
\begin{definition} 
\label{def:complete}
A quiver $Q$ is \emph{complete} (resp., \emph{abundant}) if each pair of vertices in $Q$ is connected by at least one arrow (resp., at least two arrows), in one of the two directions. 
\end{definition}

\begin{definition}
\label{def:acyclic}
Let $Q$ be a quiver on $n$ vertices. We say that $Q$ is \emph{acyclic} if $Q$ contains no oriented cycles. 
\end{definition}

\begin{figure}[ht]
\begin{tikzpicture}[scale=0.8]
\filldraw[black] (1,1) circle (2pt);
\filldraw[black] (1,2) circle (2pt);

\filldraw[black] (2,1) circle (2pt);
\filldraw[black] (2,2) circle (2pt);
\draw[black, thick, -{stealth}, shorten >=3pt, shorten <= 3pt ] (1,1) -- (1,2) node [midway, left] {$\scriptstyle 2$};
\draw[black, thick, -{stealth}, shorten >=3pt, shorten <= 3pt ] (1,1) -- (2,1);
\draw[black, thick, -{stealth}, shorten >=3pt, shorten <= 3pt ] (1,1) -- (2,2);
\draw[black, thick, -{stealth}, shorten >=3pt, shorten <= 3pt ] (1,2) -- (2,1);
\draw[black, thick, -{stealth}, shorten >=3pt, shorten <= 3pt ] (1,2) -- (2,2);
\draw[black, thick, -{stealth}, shorten >=3pt, shorten <= 3pt ] (2,2) -- (2,1) node [midway, right] {$\scriptstyle 3$};
\end{tikzpicture}
\caption{An acyclic and complete (but not abundant) quiver on 4 vertices.}
\label{fig:quiver-example}
\end{figure}

\vspace{-0.3cm}

\begin{definition}
\label{def:tree quiv}
A quiver is called a \emph{tree quiver} if its underlying undirected graph is a tree. 
A quiver is \emph{connected} if its underlying undirected graph is connected. 
\end{definition}

\begin{figure}[ht]
\begin{tikzpicture}[scale=0.7]
\filldraw[black] (1,1) circle (2pt);
\filldraw[black] (1,2) circle (2pt);
\filldraw[black] (1,3) circle (2pt);
\filldraw[black] (2,2) circle (2pt);
\filldraw[black] (2,3) circle (2pt);
\filldraw[black] (3,3) circle (2pt);
\filldraw[black] (0,2) circle (2pt);
\draw[black, thick, -{stealth}, shorten >=3pt, shorten <= 3pt ] (1,1) -- (1,2);
\draw[black, thick, -{stealth}, shorten >=3pt, shorten <= 3pt ] (1,2) -- (1,3);
\draw[black, thick, -{stealth}, shorten >=3pt, shorten <= 3pt ] (1,2) -- (0,2);
\draw[black, thick, -{stealth}, shorten >=3pt, shorten <= 3pt ] (2,2) -- (1,2);
\draw[black, thick, -{stealth}, shorten >=3pt, shorten <= 3pt ] (3,3) -- (2,3);
\draw[black, thick, -{stealth}, shorten >=3pt, shorten <= 3pt ] (2,3) -- (1,3);
\end{tikzpicture}
\hspace{2cm}
\begin{tikzpicture}[scale=0.7]
\filldraw[black] (0,2) circle (2pt);
\filldraw[black] (1,1) circle (2pt);
\filldraw[black] (1,2) circle (2pt);
\filldraw[black] (1,3) circle (2pt);

\filldraw[black] (2,1) circle (2pt);
\filldraw[black] (2,2) circle (2pt);
\filldraw[black] (2,3) circle (2pt);
\filldraw[black] (3,3) circle (2pt);
\draw[black, thick, -{stealth}, shorten >=3pt, shorten <= 3pt ] (1,1) -- (1,2);
\draw[black, thick, -{stealth}, shorten >=3pt, shorten <= 3pt ] (1,2) -- (1,3);
\draw[black, thick, -{stealth}, shorten >=3pt, shorten <= 3pt ] (1,2) -- (0,2);
\draw[black, thick, -{stealth}, shorten >=3pt, shorten <= 3pt ] (1,3) -- (0,2);

\draw[black, thick, -{stealth}, shorten >=3pt, shorten <= 3pt ] (2,1) -- (2,2);
\draw[black, thick, -{stealth}, shorten >=3pt, shorten <= 3pt ] (3,3) -- (2,3);
\draw[black, thick, -{stealth}, shorten >=3pt, shorten <= 3pt ] (2,3) -- (2,2);
\draw[black, thick, -{stealth}, shorten >=3pt, shorten <= 3pt ] (3,3) -- (2,1);
\end{tikzpicture}

\caption{Left: a tree quiver.~Right: a disconnected quiver.~Neither quiver is complete.
}
\end{figure}
\vspace{-0.3cm}

\begin{definition}
\label{def:Graph-Q}
For a quiver $Q$, we denote by $\GQ$ the underlying unoriented \emph{simple} graph of~$Q$ 
(ignoring multiplicities). 
To illustrate, if $Q$ is the quiver shown in Figure~\ref{fig:quiver-example}, 
then $\GQ$ is the complete graph~$K_4$.

In this paper, every cycle in the graph $\GQ$ will come equipped with a direction of traversal 
(one of two possible choices). 
In other words, cyclic shifts do not change the cycle, but the reversal of direction does. 

Any such cycle defines a $1$-dimensional oriented submanifold of the simplicial complex $\GQ$, 
thus an element of the first homology group $H_1(\GQ,\ZZ)$. 
\end{definition}

\begin{definition} 
To \emph{mutate} a quiver $Q$ at a vertex $j$, perform the following steps:
\begin{enumerate}
\item for each path $i \rightarrow j \rightarrow k$ in $Q$, add a new arrow $i \rightarrow k$ (thus, if we have $a$ arrows from $i$ to $j$ and $b$ arrows from $j$ to $k$, we should add $ab$ new arrows from $i$ to $k$); 
\item reverse all arrows incident to $j$; 
\item repeatedly remove oriented $2$-cycles until there are none left.
\end{enumerate}
The transformed (mutated) quiver is denoted by $\mu_j(Q)$. Mutation is an involution: $\mu_j(\mu_j(Q))=Q$. 
\end{definition}

\begin{definition} 
Two quivers are called \emph{mutation-equivalent} if they can be related to each other by a sequence of mutations. 
The \emph{mutation equivalence class} (or just \emph{mutation class}) of $Q$ is denoted by $[Q]$.
\end{definition}

\pagebreak[3]


Instead of dealing with quivers and their mutations, one can utilize the language 
of skew-symmetric matrices: 

\begin{definition}
For a given $n$-vertex quiver $Q$, the \emph{exchange matrix} 
$B=B_Q = (b_{ij})$ associated to $Q$ is an $n \times n$ skew-symmetric matrix 
defined by 
\begin{equation*}
b_{ij} = \begin{cases}
x & \text{if $Q$ contains $x\ge0$ arrows $i \rightarrow j$;} \\
-x & \text{if $Q$ contains $x\ge0$ arrows $i \leftarrow j$.}
\end{cases}
\end{equation*}
\end{definition}

\begin{example}
\label{eg:3 vert quiver and B mat}
Let $Q$ be a quiver on an ordered 3-vertex set $\{a<b<c\}$,
with $x\ge0$ arrows $a\rightarrow b$ and $y\ge0$ arrows $b\rightarrow c$.
(We can always relabel the vertices so that~the arrows point in the directions specified above.) 
The exchange matrix $B_Q$ has the~form
\begin{equation}
\label{eq:B(x,y,z)}
B_Q=\begin{bmatrix} 
0 & x & z \\
-x & 0 & y \\
-z & -y & 0
\end{bmatrix}\!. 
\end{equation}
The cases $z \ge 0$ and $z \le 0$ are shown in Figure~\ref{fig:Q(x,y,z)}.
\end{example}

\begin{figure}[ht]

\begin{tikzpicture}
\filldraw[black] (0,1) circle (2pt) node[above left=-1pt] {$a$} coordinate (a);
\filldraw[black] (1,1) circle (2pt) node[above right=-1pt] {$b$} coordinate (b);
\filldraw[black] (0,0) circle (2pt) node[below left=-1pt] {$c$} coordinate (c);
\draw[black, -{stealth}, shorten >=3pt, shorten <= 3pt] (a) -- (b) node [midway, above] {$x$};
\draw[black, -{stealth}, shorten >=3pt, shorten <= 3pt] (b) -- (c) node [midway, below right=-1pt] {$y$};
\draw[black, -{stealth}, shorten >=3pt, shorten <= 3pt] (a) -- (c) node [midway, left] {$z$};
\end{tikzpicture}
\hspace{3cm}
\begin{tikzpicture}
\filldraw[black] (0,1) circle (2pt) node[above left=-1pt] {$a$} coordinate (a);
\filldraw[black] (1,1) circle (2pt) node[above right=-1pt] {$b$} coordinate (b);
\filldraw[black] (0,0) circle (2pt) node[below left=-1pt] {$c$} coordinate (c);
\draw[black, -{stealth}, shorten >=3pt, shorten <= 3pt] (a) -- (b) node [midway, above] {$x$};
\draw[black, -{stealth}, shorten >=3pt, shorten <= 3pt] (b) -- (c) node [midway, below right=-1pt] {$y$};
\draw[black, -{stealth}, shorten >=3pt, shorten <= 3pt] (c) -- (a) node [midway, left] {$-z$};
\end{tikzpicture}
\caption{$3$-vertex quivers with the exchange matrix \eqref{eq:B(x,y,z)}. Left: an acyclic quiver ($x,y,z\ge0$). Right: a quiver with cyclically oriented arrows ($x,y\ge 0$, $z\le 0$). }
\label{fig:Q(x,y,z)}
\end{figure}

\begin{definition}
For $x\in\RR$, define the \emph{positive part} (resp., \emph{negative part}) of~$x$ by 
\begin{align*}
[x]_+ &= \max(x,0), \\
[x]_- &= \max(-x,0). 
\end{align*}
We note that both $[x]_+$ and $[x]_-$ are nonnegative.  
\end{definition}

\begin{definition}
\label{def:matrix-mut}
For a quiver $Q$, and a vertex $j$, the \emph{matrix mutation}~$\mu_j$ transforms $B_Q$ into the skew-symmetric matrix $B_{\mu_j(Q)}=(b_{ij}')=\mu_j(B)$ defined by
\begin{equation}
b_{ik}'=\begin{cases}
-b_{ik} & \text{if $i=j$ or $j=k$}; \\
b_{ik}+\frac12 (b_{ij}|b_{jk}|+|b_{ij}|b_{jk}) & \text{otherwise}
\end{cases}
\end{equation}
(see \cite[(4.3)]{ca1}). 
Alternatively, one may set 
\begin{equation}
\label{eq: def mutation signs}
b_{ik}'=\begin{cases}
-b_{ik} & \text{if $i=j$ or $j=k$}; \\
b_{ik}+ [b_{ij}]_+[b_{jk}]_+ - [b_{ij}]_-[b_{jk}]_- & \text{otherwise}
\end{cases}
\end{equation}
(see \cite[(2.2)]{ca4}). 
\end{definition}

\newpage 

\section{Cyclic orderings and wiggles}
\label{sec:cyclic orderings}

Informally speaking, a cyclic ordering of a quiver is a choice of a ``clockwise'' cyclic ordering of
its vertices, i.e., a way of placing the vertices around a circle, 
viewed modulo rotations. 
We next give a formal definition. 

\begin{definition} 
\label{def:cyclic ordering}
Let $Q$ be a quiver on an $n$-element vertex set~$V$. 
Two linear orderings $\tau=(v_1<\dots<v_n)$ and $\tau'=(v_{c(1)}<\dots<v_{c(n)})$ of~$V$ 
are \emph{cyclically equivalent} if the map $i\mapsto c(i)$ is a cyclic rearrangement, 
i.e., is given by $c(i)=(i+a)\bmod n$ for some~$a$. 
A~\emph{cyclic ordering}~$\sigma$ is an equivalence class of cyclically equivalent linear orderings. 
There are $n$ such linear orderings for a given~$\sigma$; 
we call them \emph{compatible} with~$\sigma$. 
For a linear ordering $\tau=(v_1<\cdots<v_n)$, we denote by 
$\sigma=(v_1,\dots, v_n)$ the corresponding cyclic ordering~$\sigma$. 

A~quiver on an $n$-element vertex set $V$ has $(n-1)!$ cyclic orderings. 
See Figure~\ref{fig:4-cycle orderings}. 
\end{definition}

{ 
\newcommand\cyclicLabels[5]{%
	\filldraw[black] (#4,#5)++(135:0.8cm) circle (2pt) node[above left=-1pt] {$a$} coordinate (a);
	\filldraw[black] (#4,#5)++(45:0.8cm) circle (2pt) node[above right=-1pt] {$#1$} coordinate (#1);
	\filldraw[black] (#4,#5)++(-45:0.8cm) circle (2pt) node[below right=-1pt] {$#2$} coordinate (#2);
	\filldraw[black] (#4,#5)++(-135:0.8cm) circle (2pt) node[below left=-1pt] {$#3$} coordinate (#3);
	\draw[black, dashed, decoration={markings, mark=at position 0 with {\arrow{<}}}, postaction={decorate}] (#4,#5) circle (0.8cm);
	\draw[black, -{stealth}, shorten >=3pt, shorten <= 3pt] (a) -- (b);
	\draw[black, -{stealth}, shorten >=3pt, shorten <= 3pt] (b) -- (c);
	\draw[black, -{stealth}, shorten >=3pt, shorten <= 3pt] (c) -- (d);
	\draw[black, -{stealth}, shorten >=3pt, shorten <= 3pt] (d) -- (a);
	\draw (#4,#5-1) node[below] {$(a, #1, #2, #3)$};
}

\begin{figure}[ht]
\begin{tikzpicture}
\foreach \i\j\k\xp\yp in {b/c/d/0/0, b/d/c/2.4/0, d/b/c/4.8/0, c/d/b/7.2/0, c/b/d/9.6/0, d/c/b/12/0} 
{
	\cyclicLabels{\i}{\j}{\k}{\xp}{\yp}
} 
\end{tikzpicture}
\caption{The six COQs whose underlying quiver is the 4-cycle $Q=(a \rightarrow b \rightarrow c \rightarrow d \rightarrow a)$.
The bottom row shows the corresponding cyclic orders.}
\vspace{-7pt}
\label{fig:4-cycle orderings}
\end{figure}

\begin{definition} 
A \emph{cyclically ordered quiver (COQ)} $(Q,\sigma)$
is a quiver $Q$ together with a cyclic ordering $\sigma$ of its vertices. 
When $\sigma$ is clear from the context, 
we will sometimes drop~$\sigma$ from the notation and just use $Q$ to denote a COQ.
\end{definition}

\begin{definition} 
A \emph{wiggle} is a transformation of a COQ that leaves the underlying quiver $Q$ intact while transforming the cyclic ordering via a transposition $(ij)$ that interchanges a pair of consecutive vertices $i$ and~$j$
that are not adjacent in the quiver. Note that this notion is insensitive to the orientations of the arrows. 
See Figure~\ref{fig:4-cycle-wiggles}. 
\end{definition}

\vspace{-2pt}

\begin{figure}[ht]
\newcommand{\h}[0]{1.5cm}

\begin{tikzpicture}
\cyclicLabels{b}{d}{c}{0}{0}
\end{tikzpicture}
\raisebox{\h}{$\stackrel{(bd)}{\longleftrightarrow}$}
\begin{tikzpicture}
\cyclicLabels{d}{b}{c}{3.5}{0}
\end{tikzpicture}
\raisebox{\h}{$\stackrel{(ac)}{\longleftrightarrow}$}
\begin{tikzpicture}
\cyclicLabels{c}{d}{b}{7}{0}
\end{tikzpicture}
\raisebox{\h}{$\stackrel{(bd)}{\longleftrightarrow}$}
\begin{tikzpicture}
\cyclicLabels{c}{b}{d}{10.5}{0}
\end{tikzpicture}
\caption{Four COQs related by a sequence of wiggles. 
Also, the first COQ is related to the last one by the wiggle~$(ac)$.}
\label{fig:4-cycle-wiggles}
\end{figure}
} 

\vspace{-7pt}

\begin{remark} 
If each pair of vertices of a quiver $Q$ that are consecutive in a cyclic ordering~$\sigma$ 
are connected by an arrow (going in either direction), then the COQ $(Q,\sigma)$ allows no wiggles. 
(Cf.\ the first and the last quivers in Figure~\ref{fig:4-cycle orderings}.)
In particular, no wiggles are possible if $Q$ is a complete quiver. 
\end{remark}

\begin{definition} 
\label{def:wiggle equiv}
Two cyclic orderings of a quiver are \emph{wiggle equivalent} if they can be obtained from each other by a sequence of wiggles. 
We will usually denote a wiggle equivalence class of a COQ~$Q$ (i.e., the set of all COQs wiggle equivalent to~$Q$) 
by~$\wiggleQ$.
To illustrate, Figure~\ref{fig:4-cycle-wiggles} shows a wiggle equivalence class. 
\end{definition}


In what follows, we will often consider COQs up to wiggle equivalence. 


\begin{proposition}
\label{prop:tree wiggles}
All cyclic orderings of a tree quiver are pairwise wiggle equivalent. 
\end{proposition}

\begin{proof}
Induction on the number of vertices. 
The claim for a tree~$T$ can be deduced from a similar claim for $T$ with a single leaf removed. 
Details are left to the reader. 
\end{proof}


\begin{example}
\label{eg:6-orderings-D4}
The $6$ cyclic orderings of the $4$-cycle quiver $a \rightarrow b \rightarrow c \rightarrow d \rightarrow a$ 
(see Figure~\ref{fig:4-cycle orderings}) fall into $3$ wiggle equivalence classes: 
\begin{itemize}
\item a class consisting of a single cyclic ordering $(a, b, c, d)$; 
\item a class consisting of a single cyclic ordering $(a, d, c, b)$;
\item a class consisting of the remaining four cyclic orderings of $\{a,b,c,d\}$, see Figure~\ref{fig:4-cycle-wiggles}.
\end{itemize}
\end{example}

Our next goal is to describe a solution to the following problems: 
\begin{itemize}[leftmargin=.2in]
\item 
determine whether 
two cyclic orderings of the same labeled quiver (cf.\ Remark~\ref{rem:labeled-quivers}) 
yield wiggle equivalent COQs; 
\item
if two COQs are wiggle equivalent,
construct a sequence of wiggles relating them to each other.
\end{itemize}
%
%

\begin{definition}
\label{def:dist and winding number}
Let $V$ be a finite set.
Let $\sigma=(v_1,\dots,v_n)$ be a (``clockwise'') cyclic ordering of~$V$, cf.\ Definition~\ref{def:cyclic ordering}.
(Thus $V\!=\!\{v_1,\dots,v_n\}$.)  
Let~$a,b\in V$; say, $a\!=\!v_i$ and $b\!=\!v_j$. 
The (clock\-wise) \emph{distance} $\theta(\sigma, a, b)$ between $a$ and~$b$,  
with respect to the cyclic ordering~$\sigma$, is defined by 
\begin{equation*}
\theta(\sigma, a, b) = 
\begin{cases}
j-i & \text{if $i\le j$;} \\
n+j-i & \text{if $i>j$.}
\end{cases}
\end{equation*} 
In other words, for distinct $a$ and~$b$, the distance $\theta(\sigma, a, b)$ is equal to 
1 plus the number of elements of~$V$ that we pass while moving clockwise from~$a$ to~$b$.
Notice that this notion only depends on the cyclic ordering; no quivers are involved. 
\end{definition}

\begin{example}
For the cyclic ordering $\sigma = (a,b,c)$ on a 3-element set $V=\{a,b,c\}$, we have:
$\theta(\sigma,a,b)=\theta(\sigma,b,c)=\theta(\sigma,c,a)=1$ and
$\theta(\sigma,a,c)=\theta(\sigma,b,a)=\theta(\sigma,c,b)=2$. 
\end{example}

\begin{definition}
\label{def:winding number}
Let $(Q,\sigma)$ be a COQ. 
Let 
\begin{equation}
\label{eq:u0...un}
C=(u_0 - u_1 - \cdots - u_{k-1} - u_k=u_0)
\end{equation}
be a $k$-cycle in the undirected simple graph~$\GQ$.
We consider the cycle $C$ to be endowed with a preferred direction of traversal, 
wherein $u_{i+1}$ follows~$u_i$ for $i=0,\dots,k-1$; 
cf.~Definition~\ref{def:Graph-Q}.

For each $i\in\{0,\dots,k-1\}$, precisely one of the two cases takes place:
\begin{itemize}[leftmargin=.27in]
\item[{\rm (a)}]
the quiver~$Q$ contains at least one arrow $u_i\rightarrow u_{i+1}$; 
\item[{\rm (b)}]
the quiver~$Q$ contains at least one arrow $u_i\leftarrow u_{i+1}$.
\end{itemize}
\pagebreak[3]
Let $\ell$ be the number of locations~$i$ that fall in category~(b) above. The number 
\begin{equation}
\label{eq:wind-n-cycle}
\wind(C)=\wind(C,\sigma) = \tfrac{1}{n} \Bigl ( \sum_{0\le i\le k-1} \theta(\sigma, u_i , u_{i+1}) \Bigr ) - \ell
\end{equation}
is called the \emph{winding number} of $C$. 
\end{definition}

\pagebreak[3]

\begin{example}
In the six COQs shown in Figure~\ref{fig:4-cycle orderings}, 
the cycle $(a\rightarrow b\rightarrow c\rightarrow d\rightarrow a)$
has winding numbers $1$, $2$, $2$, $2$, $2$,~$3$, respectively. 
\end{example}

\begin{remark}
The winding number of a cycle can be informally described as follows. 
We start by placing the vertices of the given quiver~$Q$ on the circle $\RR/n\ZZ$ 
in accordance with the given cyclic ordering~$\sigma$. 
We then traverse the given cycle~$C$ (cf.~\eqref{eq:u0...un}), each time moving from 
a vertex~$u_i$ to the next vertex~$u_{i+1}$ in one of the two directions:
\begin{itemize}[leftmargin=.2in]
\item 
if we see an arrow $u_i \rightarrow u_{i+1}$ (cf.\ case~(a) above), 
we move clockwise from~$u_i$ to~$u_{i+1}$;
\item
if we see an arrow $u_i \leftarrow u_{i+1}$ (cf.\ case~(b)), 
we move counterclockwise from~$u_i$~to~$u_{i+1}$. 
\end{itemize}
The winding number $\wind(C,\sigma)$ is the signed number of clockwise revolutions around  
the circle $\RR/n\ZZ$ that will occur as we complete the traversal of the cycle~$C$ in the way described above. 
(To see that, verify that each instance of case~(b) contributes $ \frac{1}{n}\theta(\sigma, u_i , u_{i+1}) - 1=-\frac{1}{n}\theta(\sigma, u_{i+1}, u_i) $ to the sum~\eqref{eq:wind-n-cycle}.) 
\end{remark}

In particular, the winding number is always an integer.
Moreover:

\begin{observation}
\label{ob:H1-to-Z}
For any COQ $(Q,\sigma)$, the map $C\mapsto \wind(C,\sigma)$ extends to a homomorphism
$H_1(\GQ,\ZZ)\to\ZZ$ from the first homology group of the undirected graph~$\GQ$. 
\end{observation}

The following result shows that the winding numbers determine, up to wiggle equivalence,
the choice of a cyclic ordering of a given quiver~$Q$: 

\begin{theorem}
\label{th:wiggle/winding}
Let $\sigma$ and $\sigma'$ be two cyclic orderings of the vertices of a quiver~$Q$. 
The following are equivalent:
\begin{itemize}[leftmargin=.2in]
\item 
the COQs $(Q,\sigma)$ and $(Q,\sigma')$ are wiggle equivalent; 
\item
for any cycle~$C$ in~$\GQ$ (equipped with a distinguished direction of traversal),~we~have 
\begin{equation*}
\wind(C,\sigma)=\wind(C,\sigma'). 
\end{equation*}
\end{itemize}
\end{theorem}


The proof of Theorem~\ref{th:wiggle/winding} is given in Section~\ref{sec:proof-wiggle/winding}. 

\begin{remark}
\label{rem:homotopy gen invariant}
Observation~\ref{ob:H1-to-Z} implies that 
in Theorem~\ref{th:wiggle/winding}, 
it suffices to check~the equality $\wind(C,\sigma)=\wind(C,\sigma')$ 
on any set of cycles~$C$ that generate the group $H_1(\GQ,\ZZ)$.
To rephrase, the collection of winding numbers of such cycles 
uniquely determines the wiggle equivalence class of a cyclic ordering of~$Q$. 
\end{remark}

The problem of constructing a cyclic ordering of a given quiver that has prescribed winding numbers
for a given basis of cycles has an efficient algorithmic solution that utilizes linear programming: 

\newcommand\LPT{P}

\begin{theorem}
\label{thm:construct an ordering}
Let $Q$ be an $n$-vertex quiver.
Let $C_1, \ldots, C_m$ be cycles in~$\GQ$, each equipped with a direction of traversal.
Assume that these cycles span~$H_1(\GQ)$. 
Then for any integers $w_1, \ldots, w_m$, 
there is an explicitly constructed polyhedral set $\LPT$ 
defined by $O(n^2+m)$ linear inequalities in a real vector space of dimension~$O(n^2)$, such that
\begin{itemize}[leftmargin=.2in]
\item 
if the set $\LPT$ is nonempty, then any point in $\LPT$ directly yields 
a cyclic ordering $\sigma$ of~$Q$ with the winding numbers $\wind(C_i, \sigma) = w_i$ for all~$i$; 
\item
if the set $\LPT$ is empty, then no such cyclic ordering exists.
\end{itemize}
\end{theorem}

The proof of Theorem~\ref{thm:construct an ordering} will also appear in Section~\ref{sec:proof-wiggle/winding}.

\newpage

\section{From winding numbers to cyclic orderings
}
\label{sec:proof-wiggle/winding}

In this section we discuss several relationships between winding numbers and wiggle equivalence classes of COQs, including proofs of Theorems~\ref{th:wiggle/winding} and~\ref{thm:construct an ordering}.
We begin with a detailed treatment of the case when $\GQ$ is an $n$-cycle.

\begin{proposition}
\label{prop:n-cycle with no wiggles}
Let $Q$ be a quiver such that the undirected simple graph~$\GQ$ is a chordless $n$-cycle 
\begin{equation}
\label{eq:n-cycle-C}
C = (v_0 - v_1 - \cdots - v_{n-1} - v_n=v_0 )\,.  
\end{equation}
(There are no arrows in~$Q$ between non-consecutive vertices in~$C$.)  
We fix one of the two directions of traversal of the cycle, namely the direction in which $v_{i+1}$ follows~$v_i$ for $i=0,\dots,n-1$. 
Let $r = \# \{ i | v_i \rightarrow v_{i+1}\}$ (resp., $\ell = \# \{ i | v_i \leftarrow v_{i+1}\}$) 
be the number of locations~$i$ for which the orientation of the arrows connecting $v_i$ with~$v_{i+1}$ 
agrees (resp., disagrees) with the chosen direction, cf.\ Definition~\ref{def:winding number}. 
(Thus $r + \ell = n$.)
Then, for any cyclic ordering~$\sigma$ on~$Q$, we have 
\begin{equation}
\label{eq:bounds-on-wind}
1-\ell\le \wind(Q,\sigma)\le r-1 = n-\ell -1.   
\end{equation}
Moreover, every winding number between $1-\ell$ and $r-1$ is achieved for some~$\sigma$. 

Furthermore, for a cyclic ordering~$\sigma$ on~$Q$, the following are equivalent:
\begin{itemize}[leftmargin=.25in]
\item[{\rm (a)}]
the COQ $(Q,\sigma)$ allows no wiggles; 
\item[{\rm (b)}]
the winding number $\wind(Q,\sigma)$ is equal to $1-\ell$ or $r-1$ (cf.~\eqref{eq:bounds-on-wind}); 
\item[{\rm (c)}]
either $\sigma=(v_n, v_{n-1}, \dots, v_1)$ or $\sigma=(v_1, v_2, \dots, v_n)$. 
\end{itemize}
\end{proposition}

\begin{proof}
The contribution of each  arrow $v_i \rightarrow v_{i+1}$ (resp., $v_i \leftarrow v_{i+1}$) 
to $\wind(Q,\sigma)$ lies in the interval $[1,n-1]$ (resp., $[1-n,-1]$). 
It follows that 
\begin{equation*}
\wind(Q,\sigma)\cdot n \in [r+\ell(1-n), r(n-1)-\ell]=[n(1-\ell),n(r-1)], 
\end{equation*}
implying~\eqref{eq:bounds-on-wind}. 

To show that every winding number between $1-\ell$ and $r-1$ is achieved, 
take $k$ between $1$ and $n-1$ and consider the cyclic ordering 
\begin{equation*}
\sigma=(v_k, v_{k-1}, \dots, v_1, v_{k+1}, \ldots, v_n). 
\end{equation*}
Straightforward calculations verify that $\wind(Q,\sigma)=k-\ell$. 

We next show that (a)$\Rightarrow$(c)$\Rightarrow$(b). 
If $(Q,\sigma)$ has no wiggles, then each pair of vertices $\{v_i, v_{i+1}\}$ must be adjacent in the cyclic ordering.
As each vertex has two neighbors in $Q$ and two neighbors in the cyclic ordering, 
the cyclic ordering must either be $\sigma_\rightarrow=(v_1, v_2, \ldots, v_n)$ 
or $\sigma_\leftarrow=(v_n, v_{n-1}, \dots, v_1)$.
A quick computation verifies that the corresponding winding numbers are $1-\ell$ and $r-1$, respectively.

It remains to prove that (b)$\Rightarrow$(a). 
Suppose that $\wind(Q,\sigma)=1-\ell$. 
By~\eqref{eq:bounds-on-wind}, this is the smallest possible value of $\wind(Q,\sigma)$. 
It follows that each arrow $v_i \rightarrow v_{i+1}$ (resp., $v_i \leftarrow v_{i+1}$) contributes its minimal possible value to the summation~\eqref{eq:wind-n-cycle}. 
Thus we must have the cyclic ordering $(v_n, v_1, v_2 \ldots
v_{n-1}) = \sigma_\rightarrow$, as claimed.
%
The case $\wind(Q,\sigma)=r-1$ is treated in the same way. 
\hide{Each edge $v_i \rightarrow v_{i+1}$ contributes between $\frac 1 n$ and $\frac{n-1}{n}$ to the winding number sum, while each edge $v_i \leftarrow v_{i+1}$ contributes between $-\frac{n-1}{n}$ and$ -\frac 1 n$.
Thus the winding number is bounded above by $\frac{n-1}{n} r - \frac l n = r-1,$ and the minimum winding number is similarly bounded below by $1-l$.
So we must achieve exactly these values for $w(\sigma, v_i, v_{i+1})$}
\end{proof}

\pagebreak[3]

\centerline{\sc Proof of Theorem~\ref{thm:construct an ordering}}

\medskip


We begin by defining the polyhedral set $\LPT$ appearing in Theorem~\ref{thm:construct an ordering}.

\begin{definition}
\label{def:polyhedra for winding}
Let $E$ denote the set of edges of the simple graph~$\GQ$. 
Let $V$ be a real vector space of dimension $|E|$ with coordinates~$\theta_{uv}$,
two for each edge $u-\!\!-v$ in~$\GQ$, satisfying $\theta_{vu} = n-\theta_{uv}$.
The set $P\subset V$ consists of all points $\theta\in V$ whose coordinates $\theta_{uv}$ satisfy the following constraints: 
\begin{itemize}[leftmargin=.2in]
\item 
for each edge $u-\!\!-v$ in~$\GQ$, we have 
\begin{equation}
\label{eq:lp thetas}
1 \leq \theta_{uv} \leq n-1;
\end{equation}
\item
for each cycle $C_i = (u_0 - u_1 - \cdots - u_k = u_0)$, we have 
\begin{equation}
\label{eq:lp winding}
\sum_{0\le j\le k-1} \theta_{u_j u_{j+1}} = n( \ell_i + w_i),
\end{equation}
where $\ell_i=\#\{j\mid u_j \leftarrow u_{j+1}\}$, cf.\ Definition~\ref{def:winding number}. 
\end{itemize}
\end{definition}


\begin{lemma}
\label{lem:lp feasible}
Suppose $\sigma$ is a cyclic ordering of $Q$ such that $\wind(C_i, \sigma) = w_i$ for all~$i$.
Then the point $\theta\in V$ defined by $\theta_{uv} = \theta(\sigma, u, v)$ lies in $\LPT$.
\end{lemma}

\begin{proof}
The inequalities \eqref{eq:lp thetas} follow from Definition~\ref{def:dist and winding number}.
The equalities \eqref{eq:lp winding} follow from the assumptions and Definition~\ref{def:winding number}.
\end{proof}

Lemma~\ref{lem:lp feasible} implies that if $\LPT=\varnothing$, then there are no cyclic orderings with the desired winding numbers.

We next describe how a point $\theta' = \{ \theta'_{uv}\}\in\LPT$ yields a cyclic ordering with the desired winding numbers.

\begin{definition}
Fix a spanning tree $T$ of $\GQ$ and a root vertex~$v_\circ$. 
For a vertex~$v$, we denote by
\begin{equation*}
T(v) = (v_\circ, v_1, \dots, v_k=v)
\end{equation*}
the unique (undirected) path in $T$ that connects the root $v_\circ$ to~$v$. 
We then define
\begin{equation*}
\mathcal{R}(v) = \sum_{(v_i, v_{i+1}) \in T(v)}  \theta'_{v_i v_{i+1}}\bmod n \in \RR/n \ZZ. 
\end{equation*}
The circle $\RR/n \ZZ$ has a natural cyclic ordering coming from the linear ordering on~$[0, n)$.
Restricting to the values $\mathcal{R}(v)$, we obtain a cyclic ordering $\sigma_{\theta'}$ on the vertices of~$Q$ 
(after breaking ties if necessary, by perturbing the $\mathcal{R}(v)$ slightly).
\end{definition}

Equation~\eqref{eq:lp winding}, together with the fact that the cycles~$C_i$ span $H_1(\GQ)$, 
implies the following statement. 

\begin{lemma}
\label{lem:respect all distances}
Let $u,v$ be two vertices in $Q$. 
Then 
$$\mathcal{R}(v) - \mathcal{R}(u) \equiv \theta'_{uv} \bmod n \in \mathbb R / n \ZZ$$
 for all adjacent vertices $u,v$ in $Q$.
\end{lemma}

To complete the proof of Theorem~\ref{thm:construct an ordering},  
fix a cycle $C_i = (u_0 - u_1 - \cdots - u_k = u_0)$. 
The cyclic ordering~$\sigma_{\theta'}$ and the cyclic ordering induced by $\mathcal{R}(v)$ induce homotopic maps from $C_i$ to $S^1$.
Thus they have the same winding numbers.
By Lemma~\ref{lem:respect all distances}, the winding number induced by $\mathcal{R}(v)$
is given by 
\begin{equation*}
\wind(C_i) = \tfrac{1}{n} \Big ( \sum_{(u_j, u_{j+1}) \in C_i} \theta'_{u_j  u_{j+1}} \Big ) - \ell_i = w_i,
\end{equation*}
cf.\ Definition~\ref{def:polyhedra for winding} and specifically equation~\eqref{eq:lp winding}. 
\qed

\begin{example}
\label{eg:grid2x6}
Let $Q$ be the $2 \times 6$ grid quiver $Q$ shown in Figure~\ref{fig:grid2x6}.
Consider the following cycles $C_1,\dots,C_5$ and associated winding numbers $w_1,\dots,w_5$: 
\begin{equation*}
\begin{array}{ll} 
C_1 = ( a - b -  h - g-a) \qquad & w_1=1 \\[2pt]
C_2 = (b - c - i - h-b) & w_2=-1 \\[2pt]
C_3 = (c - d - j - i - c) & w_3= 1 \\[2pt]
C_4 = (d - e - k - j - d) & w_4=-1 \\[2pt]
C_5 = (f - e - k - l - f) & w_5=-3. 
\end{array}
\end{equation*}
The cycles $C_i$ span $H_1(K_Q)$. 
The set $P$ is defined by the inequalities 
and equations 
\begin{equation*}
\begin{array}{lll}
& \theta_{ab} + \theta_{bh} + \theta_{hg} + \theta_{ga} = 12(0+1) \\[2pt]
& \theta_{bc} + \theta_{ci} + \theta_{ih} + \theta_{hb} = 12(4-1) &  \theta_{bh} +\theta_{hb} = 12 \\[2pt]
1 \leq \theta_{u,v} \leq n-1\qquad
& \theta_{cd} + \theta_{dj} + \theta_{ji} + \theta_{ic} = 12(0+1) \quad\qquad & \theta_{ci} +\theta_{ic} = 12  \\[2pt]
& \theta_{de} + \theta_{ek} + \theta_{kj} + \theta_{jd} = 12(4-1) & \theta_{dj} +\theta_{jd} = 12\\[2pt]
& \theta_{fe} + \theta_{ek} + \theta_{kl} + \theta_{lf} = 12(4-3)
\end{array}
\end{equation*}
(cf.\ Definition~\ref{def:polyhedra for winding}). 
Starting with the solution 
\begin{equation*}
\begin{array}{c}
\theta_{bh} = \theta_{hg} = \theta_{ga}= \theta_{dj}= \theta_{ji}= \theta_{ic}= \theta_{fe} = \theta_{kl} = \theta_{lf} =1, \\[4pt]
\theta_{ih} = 3, \quad \theta_{kj} = 5, \quad \theta_{ab} = \theta_{cd} = \theta_{ek} = 9, \quad \theta_{bc} = \theta_{de} = 11,
\end{array}
\end{equation*}
we construct a cyclic ordering $\sigma$ with the desired winding numbers. 
Let the spanning tree $T$ be $K_Q$ with the path $(g - h - i - j  - k - l)$ removed, with root vertex~$c$.
Then 
\begin{equation*}
\begin{array}{ll}
\mathcal{R}(a) \equiv -\theta_{bc} - \theta_{ab} \equiv  4, &\mathcal{R}(g) \equiv 3, \\[2pt]
\mathcal{R}(b) \equiv -\theta_{bc} \equiv 1, &\mathcal{R}(h) \equiv 2, \\[2pt]
\mathcal{R}(c) \equiv 0, &\mathcal{R}(i) \equiv 11, \\[2pt]
\mathcal{R}(d) \equiv \theta_{cd}\equiv 9, &\mathcal{R}(j) \equiv 10,\\[2pt]
\mathcal{R}(e)\equiv \theta_{cd} + \theta_{de} \equiv 8, &\mathcal{R}(k) \equiv 5,\\[2pt]
\mathcal{R}(f) \equiv \theta_{cd} + \theta_{de} - \theta_{fe} \equiv 7, \qquad &\mathcal{R}(l) \equiv 6.
\end{array}
\end{equation*}
Ordering the vertices according to the values $\mathcal{R}(v)$, we obtain the cyclic ordering
$$\sigma = (c, b, h, g, a, k, l, f, e, d, j, i).$$
\end{example}

\enlargethispage{5pt}

\begin{figure}[h]
\vspace{-5pt}
\begin{equation*}
\begin{tikzcd}[arrows={-stealth}, sep=1.6em]
  a  \arrow[r] & b \arrow[d] & c \arrow[r] \arrow[l] & d \arrow[d] & e \arrow[r] \arrow[l] & f \arrow[d]   \\[-4pt]
  g  \arrow[u] & h \arrow[r] \arrow[l] & i \arrow[u] & j \arrow[r] \arrow[l] & k \arrow[u] & l \arrow[l]  
\end{tikzcd}
\end{equation*}
\vspace{-8pt}
\caption{A $2 \times 6$ grid quiver.} 
\vspace{-8pt}
\label{fig:grid2x6}
\end{figure}

\pagebreak[3]

\centerline{\sc Proof of Theorem~\ref{th:wiggle/winding}}

\medskip

One direction of Theorem~\ref{th:wiggle/winding} is rather straightforward (so we omit its proof):   

\begin{lemma}
\label{lem:winding number invariant}
The winding number of any cycle in a COQ is invariant under wiggles.

More precisely, let $Q$ be a quiver and let $C$ be a cycle in the undirected graph~$K_Q$.
We fix a direction of traversal of~$C$, as in Definition~\ref{def:winding number}. 
If two cyclic orderings $\sigma$ and $\sigma'$ are related by a wiggle (or more generally, are wiggle equivalent), 
then the correspoding winding numbers coincide: $\wind(C,\sigma)=\wind(C,\sigma')$. 
\end{lemma}

It remains to show that if two COQs $(Q, \sigma)$ and $(Q, \sigma')$ with the same underlying quiver~$Q$ 
have the same winding numbers, then they are wiggle equivalent. 

Without loss of generality, we assume that (the underlying unoriented graph~of) the quiver~$Q$ is connected. 


We fix a spanning tree $T$ of the underlying undirected graph of~$Q$.
We also fix a root vertex~$v_\circ$. 
For any vertex~$v$, we denote by
\begin{equation*}
T(v) = (v_\circ, v_1, \dots, v_k=v)
\end{equation*}
the unique (undirected) path in $T$ that connects the root $v_\circ$ to~$v$. 

\begin{definition}
\label{def:algorithm for wiggle equiv}
Let $u$ and $v$ be two adjacent vertices in the tree~$T$.
We assume that $Q$ contains an edge $u \rightarrow v$. 
For any value of the \emph{time parameter} $t\in [0,1]$, we set
\begin{equation*}
\theta(t, u,v) = (1-t) \theta(\sigma,u,v) + t \theta(\sigma',u,v).
\end{equation*}
Thus, $\theta(t, u,v)$ linearly interpolates between $\theta(\sigma, u,v)$ and $\theta(\sigma', u,v)$.  
 

We then define, for any vertex $v$ and time $t\in [0,1]$, 
\begin{align*}
\O(t,v) &= \sum_{(v_i, v_{i+1}) \in T(v)} \theta(t,v_i, v_{i+1}) \in \RR. 
\end{align*}
Alternatively, the real numbers $\O(t,v)$ can be defined as follows. 
We define the numbers $\O(0,v)$ by the initial condition $\O(0,v_\circ)=0$
together with the recurrence
\begin{equation*}
\O(0,v)-\O(0,u) = \theta(\sigma,u,v), 
\end{equation*}
for every arrow $u\rightarrow v$ as above (i.e., $u$ and $v$ are adjacent in the tree~$T$).
We define the numbers $\O(1,v)$ in the same way using the cyclic ordering~$\sigma'$. 
Finally, we interpolate linearly for $0<t<1$:
\begin{equation*}
\O(t,v)=(1-t) \O(0,v)+t \O(1,v). 
\end{equation*}

We also set
\begin{align*}
\mathcal{R}(t,v) &= \O(t,v) \bmod n \in \RR / n \ZZ. 
\end{align*}
\end{definition}

\begin{example}
\label{eg:tree wiggle equiv-1}
Let $Q$ be the 3-vertex quiver $(a\rightarrow b \rightarrow c)$ of type~$A_3$.
Consider two wiggle equivalent cyclic orderings 
$\sigma = (a,b,c)$ and $\sigma' = (a,c,b)$. 
The underlying undirected graph of $Q$ is a tree. 
Selecting the root $v_\circ=a$ gives 
\begin{align*}
&\O(t,a) = 0, \quad \O(t,b)= 1+t, \quad \O(t,c)= 2+2t, \\
&\mathcal{R}(t,a) = 0, \quad \mathcal{R}(t,b)= 1+t, \quad \mathcal{R}(t,c)= \begin{cases}
2+2t & \text{if $t<\frac12$;} \\ -1+2t & \text{if $t\ge\frac12$.}
\end{cases}
\end{align*}
\end{example}

\begin{example}
\label{eg:4cycle wiggle equiv-1}
Consider the 4-cycle quiver 
\begin{equation*}
\begin{tikzcd}[arrows={-stealth}, sep=2em]
  a  \arrow[r] & b \arrow[d]  \\
 d\arrow[u]  & c   \arrow[l] 
\end{tikzcd}
\end{equation*}
of type~$D_4$, 
with two cyclic orderings $\sigma = (a,b,d,c)$ and $\sigma' = (a,c,d,b)$. 
Remove the arrow $c\rightarrow d$ from $Q$ to get the tree~$T$. 
Select the root $v_\circ=a$. 
Then 
\begin{equation*}
\begin{array}{llll}
\O(t,a) = 0, \quad & \O(t,b) = 1+2t, \quad & \O(t,c) = 3+2t, \quad & \O(t,d) = -2; \\[5pt]
\mathcal{R}(t,a) = 0, \quad & \mathcal{R}(t,b)=1+2t, \quad & \mathcal{R}(t,c) =\begin{cases}
3+2t & \text{if $t<\frac12$;} \\ -1+2t & \text{if $t\ge\frac12$,}
\end{cases}
\quad &\mathcal{R}(t,d)=2. 
\end{array} 
\end{equation*}
\end{example}

\begin{definition}
\label{def: Z/nZ circle ordering}
The circle $\RR/n\ZZ$ is naturally endowed with the cyclic ordering associated to the linear order $([0,n), <)$. 
Restricting this cyclic ordering to the locations $\mathcal{R}(t,v)\in\RR / n \ZZ$, we obtain, for a generic time parameter
$t\in [0,1]$, a well-defined cyclic ordering $\sigma_t$ on the set of vertices of~$Q$. 
This cyclic ordering ``interpolates'' between the cyclic orderings $\sigma$ (at $t=0$) and $\sigma'$ (at $t=1$). 
\end{definition}

We next focus on the instances of ``collisions'' where the cyclic orderings $\sigma_t$ are ill-defined. 

\begin{definition}
\label{def:collisions}
For $t \in (0,1)$ and $x \in \RR/n \ZZ$, we say that $(t,x)$ is a \emph{collision point} 
if there exist distinct vertices $u \neq v$ such that $x=\mathcal{R}(t,u) = \mathcal{R}(t,v)$. 
It is easy to see that the number of collision points is finite. 

For a collision point $(t,x)$, we refer to the set 
\begin{equation*}
\mathcal{C}(t,x) = \{v | \mathcal{R}(t,v) = x\}
\end{equation*}
as the \emph{set of colliding vertices} (at $(t,x)$). 
The vertices in $\mathcal{C}(t,x)$ are permuted at time~$t$ according to some permutation~$w(t,x)$.
More precisely, $w(t,x)$ is the permutation of the vertices of~$Q$
that intertwines the orderings of $\mathcal{C}(t,x)$ 
induced by $\sigma_{t-\varepsilon}$ and $\sigma_{t+\varepsilon}$, 
respectively, keeping the remaining vertices fixed. 
\end{definition}

\begin{lemma}
\label{lem:colliding set contiguous}
Each set of colliding vertices $\mathcal{C}(t, x)$ is a contiguous interval 
in the cyclic ordering $\sigma_{t-\varepsilon}$
(resp., $\sigma_{t+\varepsilon}$), for $\varepsilon > 0$ sufficiently small.
The permutation $w(t,x)$ reverses the order of the elements of $\mathcal{C}(t,x)$,
keeping the remaining vertices fixed. 
\end{lemma}

\begin{proof}
The first statement is clear. 
To prove the second, recall that 
the ``location'' $\mathcal{R}(t,v)$ of each vertex~$v\in\mathcal{C}(t,x)$ is moving at constant speed. 
\end{proof}

\begin{example}
\label{eg:tree wiggle equiv-2}
In Example~\ref{eg:tree wiggle equiv-1}, the only collision point is $(\frac{1}{2}, 0)$. 
Its set of colliding vertices is $\{a,c\}$. 
The two cyclic orderings are related by the wiggle $(a c)=w(\frac12,0)$.
\end{example}

\begin{example}
\label{eg:4cycle wiggle equiv-2}
In Example~\ref{eg:4cycle wiggle equiv-1}, the collision points are $(\frac{1}{2},0)$ and $(\frac{1}{2},2)$. 
The sets of colliding vertices are $\mathcal{C}(\frac{1}{2}, 0) = \{a,c\}$ and $\mathcal{C}(\frac{1}{2},2)=\{b,d\}$.
The two orders are related by the composition of two commuting wiggles $w(\frac{1}{2}, 0) = (a c)$ 
and $w(\frac{1}{2}, 2) = (b d)$. 
\end{example}

\pagebreak[3]

\begin{lemma}
\label{lem:w works for all edges}
If two COQs $(Q, \sigma)$ and $(Q, \sigma')$ have the same winding numbers,
then for every arrow $u \rightarrow v$ in the quiver~$Q$ 
and every $t\in [0,1]$, we have 
\begin{equation*}
\O(t,v)-\O(t,u)
\equiv (1-t) \theta(\sigma,u,v) + t \theta(\sigma',u,v) \bmod n.
\end{equation*}
\end{lemma}

\begin{proof}
If $u$ and $v$ are adjacent in $T$, then the claim follows from Definition~\ref{def:collisions}:
\begin{align*}
\O(t,v)-\O(t,u) &= (1-t) \O(0,v)+t\O(1,v) - (1-t)\O(0,u) - t\O(1,u) \\
&= (1-t) \theta(\sigma,u,v) + t \theta(\sigma',u,v). 
\end{align*}
Now suppose that $u$ and $v$ are not adjacent in~$T$. 
Then adding the edge $u-v$ to~$T$ produces exactly one cycle,
say with the edges $v = u_1 - u_2 - \cdots - u_k=u - v$. 
Let $m\in\ZZ$ be the winding number of this cycle with respect to the cyclic orderings
$\sigma$ and~$\sigma'$. (We know that the two winding numbers agree.) 
We then have:
\begin{align*}
\O(t,v)-\O(t,u) 
&= -\sum_{1\le i\le k-1} (\O(t,u_{i+1})\!-\! \O(t,u_{i})) \\
&= -\sum_{1\le i\le k-1} \sgn(b_{u_i, u_{i+1}}) ((1-t)  \theta(\sigma , u_i, u_{i+1}) + t \theta(\sigma',u_i,u_{i+1})) \\
&= - ((1-t)(mn-\theta(\sigma,u,v)) + t (mn - \theta(\sigma',u,v))) \\
&\equiv (1-t) \theta(\sigma,u,v) + t \theta(\sigma',u,v) \bmod n. \qedhere
\end{align*}
\end{proof}

\begin{example}
\label{eg:4cycle wiggle equiv-3}
Continuing with Examples~\ref{eg:4cycle wiggle equiv-1} and~\ref{eg:4cycle wiggle equiv-2}, we get
\begin{align*}
\O(t,d) - \O(t,c) &= -2 - 3 - 2t \\
&= -5-2t \\
&\equiv (1-t) \theta(\sigma,c,d)+ t \theta(\sigma',c,d) \bmod 4,
\end{align*}
consistent with Lemma~\ref{lem:w works for all edges}.
\end{example}

\begin{lemma}
\label{lem:colliding set disjoint}
If $(Q, \sigma)$ and $(Q, \sigma')$ have the same winding numbers, 
then each set of colliding vertices $\mathcal{C}(t,x)$ consists of vertices that are pairwise non-adjacent in~$Q$.
\end{lemma}

\begin{proof}
Let $u,v\in\mathcal{C}(t,x)$. 
Then $\mathcal{R}(t,v) = \mathcal{R}(t,u)$ by Definition~\ref{def:collisions}. 
Suppose that $u \rightarrow v$ is an arrow in~$Q$. 
By Lemma~\ref{lem:w works for all edges}, we have
\begin{equation}
0 = \mathcal{R}(t,v) - \mathcal{R}(t,u) \equiv (1-t) \theta(\sigma,u,v) + t \theta(\sigma',u,v) \bmod n .
\label{eq:R-R}
\end{equation} 
On the other hand, both $\theta(\sigma,u,v)$ and $\theta(\sigma',u,v)$ lie in the interval $(1,n-1)$.  
Therefore the same is true for $(1-t) \theta(\sigma,u,v) + t \theta(\sigma',u,v)$, in contradiction with~\eqref{eq:R-R}. 
\end{proof}

We are now ready to complete the proof of Theorem~\ref{th:wiggle/winding}.

\begin{proof}[Proof of Theorem~\ref{th:wiggle/winding}]
As $t$ changes from $t=0$ to $t=1$, the 
cyclic ordering $\sigma_t$ is transformed from~$\sigma$ to~$\sigma'$ 
via a sequence of vertex permutations $w(t,x)$ corresponding to the various collision points~$(t,x)$.  
(We apply these permutations in the order~dictated by~$t$, breaking ties arbitrarily.) 
Lemma~\ref{lem:colliding set disjoint} ensures that
each permutation $w(t,x)$ permutes pairwise non-adjacent vertices---so 
this permutation can be implemented
as a sequence of wiggles. We conclude that $(Q,\sigma)$ and $(Q,\sigma')$ are wiggle equivalent. 
\end{proof}

\newpage


\section{Unipotent companions and their cosquares}
\label{sec:unipotent companions}
We will need to recall some basic linear algebra. 
For a matrix $M$, we will denote by $M^T$ the transpose of $M$. 
We denote by $I$ the $n \times n$ identity matrix. 

\begin{definition} 
Two $n\times n$ integer matrices $L$ and $M$ are called \emph{congruent} 
(over~$\ZZ$) 
if there exists a matrix $G\in \operatorname{GL}_n(\ZZ)$ 
(i.e., an integer matrix of determinant~$\pm1$) 
such that $M=GLG^T$. 
The congruence relation is symmetric. 
The integral \emph{congruence class} of an $n\times n$ integer matrix $M$ 
consists of all matrices congruent to $M$ over $\ZZ$. 
\end{definition}

The following definition is fundamental for all subsequent developments. 

\begin{definition} 
\label{def:unipotent companions}
Let $Q$ be a quiver on a linearly ordered vertex set $\{1<\cdots<n\}$. 
Let $B=B_Q=(b_{ij})$ be the corresponding exchange matrix. 
The \emph{unipotent companion} of~$Q$ (or of~$B$) is 
the unique unipotent upper-triangular matrix $U=U_Q$ satisfying 
\begin{equation}
\label{eq:-B=U-UT}
-B = U-U^T.
\end{equation}
In other words, $U$ is obtained by taking the strictly upper-triangular part of $B$, changing its sign, and placing 1's on the diagonal:
\begin{equation*}
U=\begin{bmatrix}
1 & -b_{12} & -b_{13} &\cdots  & -b_{1n} \\
0 & 1 & -b_{23} & \cdots  & -b_{2n} \\
0 & 0 & 1 & \cdots & -b_{3n}\\
\vdots & \vdots & \vdots & \ddots & \vdots  \\
0 & \cdots & 0 & \cdots & 1
\end{bmatrix}.
\end{equation*}

\end{definition}

We note that the unipotent companion depends on the choice of a linear ordering of the vertices of a quiver. 

\begin{remark}
\label{rem:U-vs-A}
The notion of a unipotent companion is reminiscent of (but distinct from) the notion of a 
\emph{quasi-Cartan companion} introduced by M.~Barot, C.~Geiss, and A.~Zelevinsky~\cite{BGZ}. 
Given a linear ordering of the vertices of a quiver~$Q$,
the corresponding quasi-Cartan companion 
is the symmetric matrix $A=A_Q$ defined by
\begin{equation}
\label{eq:A(Q)-via-U}
A=U+U^T. 
\end{equation}
Both the exchange matrix $B$ and the quasi-Cartan matrix~$A$ are determined by the unipotent companion~$U$,
cf.\ \eqref{eq:-B=U-UT} and~\eqref{eq:A(Q)-via-U}. 
See also Remark~\ref{rem:congruence-AB}. 
\end{remark}

\begin{proposition}
\label{prop:UQ cyclic shift congruent} 
The integral congruence class of a unipotent companion 
is invariant under cyclic rearrangements of the vertices of a quiver. 
\end{proposition}

\begin{proof}
Let $Q$ be an $n$-vertex quiver on the vertex set $\{1<\cdots<n\}$.
Let $c\in\GL_n$ be the permutation matrix 
\begin{equation}
\label{eq:cyclic-shift}
c=
\begin{bmatrix}
0 & 1 & 0 & \cdots & 0 \\
0 & 0 & 1 & \cdots & 0 \\
\vdots & \vdots & \vdots & \ddots & \vdots  \\
0 & 0 & 0 & \cdots & 1 \\
1 & 0 & 0 & \cdots & 0
\end{bmatrix}
\end{equation}
associated with the $n$-cycle (a Coxeter element)
\begin{equation}
\label{eq:long-cycle}
1 \rightarrow 2 \rightarrow 3 \rightarrow \cdots \rightarrow n \rightarrow 1
\end{equation}
in the symmetric group~$\mathcal{S}_n$.
Let $cQ$ be the quiver with the exchange matrix~$c B_Q c^T$,
or equivalently the quiver obtained by changing the vertex ordering in~$Q$ according~to~$c$.
Let $B_1$ denote the $n\times n$ matrix whose top row is the same as in $B$ and whose other entries are equal to~$0$: 
\begin{equation*}
B_1 =
\begin{bmatrix}
0 & b_{12} & \cdots & b_{1n} \\
0 & 0 & \cdots & 0 \\
\vdots &  \vdots & \ddots & \vdots  \\
0 & 0 & \cdots & 0 
\end{bmatrix}
.
\end{equation*}
Set 
$G=c (I + B_1^T)$. 
We will prove the proposition by showing that $U_{cQ} = G U_Q G^T$. 

We begin by expressing the matrix~$U_{cQ}$ in terms of the original unipotent companion $U=U_Q$, 
the permutation matrix~$c$, and the matrix~$B_1$: 
\begin{align*}
U_{cQ}&=
\begin{bmatrix}
1 & \!-b_{23}\! &  \cdots & \!-b_{2n}\! & \!-b_{21} \\
0 & 1 &  \cdots & \!-b_{3n}\! & \!-b_{31}\\
\vdots &  \vdots & \ddots & \vdots & \vdots  \\
0 & 0 &  \cdots & 1 & \!-b_{n1} \\
0 & 0 &  \cdots & 0 & 1 
\end{bmatrix}
\\
&=
c
\begin{bmatrix}
1 & 0 &  \cdots & 0 & 0 \\
-b_{21} & 1 &  \cdots & \!-b_{2,n-1}\! & \!-b_{2n}\\
\vdots &  \vdots & \ddots & \vdots & \vdots  \\
-b_{n-1,1} & 0 &  \cdots & 1 & \!-b_{n-1,n}\! \\
-b_{n,1} & 0 &  \cdots & 0 & 1 
\end{bmatrix}
c^{-1} \\
&=c(U+B_1+B_1^T) c^T.
\end{align*}
Since the matrix $I-U$ is strictly upper-triangular, we have $(I-U)B_1=0$, so that 
\begin{equation*}
B_1 = UB_1. 
\end{equation*}
Since the top row of $I-U-B_1$ consists entirely of zeroes, we have
$B_1^T (I-U-B_1)=0$, or equivalently
\begin{equation*}
B_1^T = B_1^T U + B_1^T B_1= B_1^T U + B_1^T UB_1. 
\end{equation*}
It follows that 
\begin{align*}
U+B_1+B_1^T 
&= U + UB_1 + B_1^T U + B_1^T U B_1 \\
&= (I+B_1^T)U(I+B_1). 
\end{align*}
We conclude that 
\begin{equation*}
U_{cQ} = c(U+B_1+B_1^T) c^T = c (I+B_1^T)U(I+B_1) c^T =G U G^T. 
\qedhere
\end{equation*}
\end{proof}

Proposition~\ref{prop:UQ cyclic shift congruent} shows that 
the unipotent companion of a cyclically ordered quiver (COQ) is well defined up to integral congruence.
Thus, we can associate to any COQ the integral congruence class of a unipotent companion. 

\begin{proposition} 
\label{prop:UQ wiggle congruent}
The integral congruence class of a unipotent companion is invariant under wiggles.  
 \end{proposition}

\begin{proof}
Let $(Q,\sigma)$ be a COQ on the vertex set $\{v_1,\dots,v_n\}$,
with the cyclic ordering $\sigma=(v_1,\dots,v_n)$. 
Assume that the transposition $s_1=(v_1 v_2)$ is a wiggle in $(Q,\sigma)$, 
i.e., the vertices $v_1$~and~$v_2$ are not adjacent in~$Q$.
We identify $s_1$ with the corresponding $n\times n$ permutation matrix. 

Let $\sigma'=(v_2,v_1,v_3,\dots,v_n)$ be the ordering obtained by swapping the vertices $v_1$ and~$v_2$.
Let $U'$ denote the corresponding unipotent companion matrix. 
This matrix can be related to the original unipotent companion~$U$ as follows: 
\begin{align*}
s_1 U s_1 &=
\begin{bmatrix}
0 & 1 & 0 &\cdots  & 0 \\
1 & 0 & 0 & \cdots  & 0 \\
0 & 0 & 1 & \cdots & 0\\
\vdots & \vdots & \vdots & \ddots & \vdots  \\
0 & \cdots & 0 & \cdots & 1
\end{bmatrix}
\begin{bmatrix}
1 & 0 & -b_{13} &\cdots  & -b_{1n} \\
0 & 1 & -b_{23} & \cdots  & -b_{2n} \\
0 & 0 & 1 & \cdots & -b_{3n}\\
\vdots & \vdots & \vdots & \ddots & \vdots  \\
0 & \cdots & 0 & \cdots & 1
\end{bmatrix}
\begin{bmatrix}
0 & 1 & 0 &\cdots  & 0 \\
1 & 0 & 0 & \cdots  & 0 \\
0 & 0 & 1 & \cdots & 0\\
\vdots & \vdots & \vdots & \ddots & \vdots  \\
0 & \cdots & 0 & \cdots & 1
\end{bmatrix} 
\\
&= 
\begin{bmatrix}
1 & 0 & -b_{23} &\cdots  & -b_{2n} \\
0 & 1 & -b_{13} & \cdots  & -b_{1n} \\
0 & 0 & 1 & \cdots & -b_{3n}\\
\vdots & \vdots & \vdots & \ddots & \vdots  \\
0 & \cdots & 0 & \cdots & 1
\end{bmatrix}
\\
&=U'. 
\end{align*}
We conclude that the matrices $U$ and $U'=s_1Us_1=s_1Us_1^T$ are congruent over $\ZZ$.
\end{proof}

Proposition~\ref{prop:UQ wiggle congruent} shows that the integral congruence class of a unipotent companion is uniquely determined by the wiggle equivalence class of a COQ.

\begin{remark}
\label{rem:congruence-AB}
Various authors (see, e.g., \cite{BGZ,casals-binary, fwz1-3,Seven-congruence}) considered
using the congruence class of either the exchange matrix~$B_Q$ or 
a particular quasi-Cartan companion $A_Q$ to study the properties of a quiver~$Q$. 
Unfortunately, these congruence classes appear to retain a lot less useful information about~$Q$
compared to the congruence class of the unipotent companion~$U_Q$.  
\end{remark}

We next recall some well-known results relating congruence classes of square matrices to similarity/conjugacy classes. 
We will use the notation $M^{-T}\!=\!(M^T)^{-1}\!=\!(M^{-1})^T$. 
For invertible matrices $A$ and~$B$,
we have 
$(AB)^{-T} \!=\! ((AB)^T)^{-1} \!=\! (B^T A^T)^{-1} \!=\!A^{-T} B^{-T}$. 

\begin{definition}
\label{def:cosquare}
The \emph{cosquare} of an invertible matrix $M$ is the matrix $M^{-T}M$. 
\end{definition}

\begin{lemma} 
\label{lem:C-sim=Ucong}
If two  matrices in $\operatorname{GL}_n(\ZZ)$ are congruent over $\ZZ$, 
then their respective cosquares are similar over~$\ZZ$ (i.e., conjugate in $\operatorname{GL}_n(\ZZ)$). 
 \end{lemma}

\begin{proof}
Let $L$ and $M=GLG^T$ be two congruent matrices. 
Then 
\[
 M^{-T} M =(GLG^T)^{-T}GLG^T=G^{-T}L^{-T}G^{-1}GLG^T=G^{-T}L^{-T}LG^T. \qedhere
 \]
\end{proof}

\begin{remark}
\label{rem:horn-sergeichuk}
As shown by R.~Horn and V.~Sergeichuk \cite[Lemma 2.1]{horn-sergeichuk}, 
the converse to Lemma~\ref{lem:C-sim=Ucong} holds over~$\CC$.
However, it fails over the integers (and over the reals). 
For example, take two symmetric matrices $A,B\in\GL(n,\ZZ)$ (or ${A, B\in\GL(n,\RR)}$).  
The~cosquares of $A$ and $B$ are equal, as they are both equal to the identity matrix. 
Yet it can happen that $A$ and $B$ are not congruent over~$\RR$, let alone over~$\ZZ$. 
Indeed, Sylvester's Law of Inertia asserts that  
two real symmetric matrices are congruent over $\RR$ 
if and only if they have the same number of positive, negative, and zero eigenvalues.
\end{remark}

\pagebreak[3]

\begin{problem}
It is conceivable (although unlikely) that the integer 
version of the Horn-Sergeichuk theorem referenced in Remark~\ref{rem:horn-sergeichuk} 
might hold for upper-triangular unipotent matrices. 
In other words, if $U$ and $U'$ are two upper-triangular unipotent integer matrices whose cosquares
are conjugate in $\GL(n,\ZZ)$, does it follow that $U$ and~$U'$ are congruent over~$\ZZ$?
(If not, provide a counterexample.) 
\end{problem}

\begin{corollary}
\label{cor:similarity-U}
Let $Q$ be a quiver on a linearly ordered set of vertices ${\{v_1\!<\cdots<v_n\}}$. 
The $\GL(n,\ZZ)$ conjugacy class of the cosquare of the unipotent companion~$U_Q$ 
is uniquely determined by the cyclic ordering $\sigma=(v_1,\dots,v_n)$, and indeed by 
the wiggle equivalence class of the COQ~$(Q,\sigma)$. 
\end{corollary}

\begin{proof}
This follows from Lemma~\ref{lem:C-sim=Ucong}, 
together with Propositions~\ref{prop:UQ cyclic shift congruent} 
and~\ref{prop:UQ wiggle congruent}. 
\end{proof}

\begin{remark}
The construction of the cosquare of a unipotent upper-triangular matrix has appeared, 
under the name of a \emph{Coxeter matrix} or \emph{monodromy matrix}, 
in a number of contexts ranging from algebraic geometry to singularity theory and mathematical physics. 
See Section~\ref{sec:bondal} and references therein. 
\end{remark}

\begin{example}
\label{eg:U-n=2}
Let $Q$ be a 2-vertex quiver with 
\begin{align*}
B=B_Q&=\begin{bmatrix} 
0 & x  \\
-x & 0 
\end{bmatrix}. 
\end{align*}
The unipotent companion is 
\begin{align*}
U=U_Q&=\begin{bmatrix} 
1 &  \! \!-x  \\
0 &  \! \!1 
\end{bmatrix}, 
\end{align*}
and its cosquare is 
\begin{align*}
U^{-T} U
&=
\begin{bmatrix} 
1 &  0  \\
x &  1 
\end{bmatrix}
\begin{bmatrix} 
1 &  -x  \\
0 &  1
\end{bmatrix}
=
\begin{bmatrix} 
1 &  -x  \\
x & 1-x^2 
\end{bmatrix}
\!\!.
\end{align*}
\end{example}

\begin{example}
\label{eg:U-n=3}
Let $Q$ be a quiver on three linearly ordered vertices, with 
\begin{align*}
B=B_Q&=\begin{bmatrix} 
0 & x & z \\
-x & 0 & y \\
-z & -y & 0
\end{bmatrix},
\end{align*}
cf.\ Example~\ref{eg:3 vert quiver and B mat}. 
The unipotent companion $U=U_Q$ is  the matrix
\begin{align*}
U&=\begin{bmatrix} 
1 &  \! \!-x \! \! & -z \\
0 &  \! \!1 \! \! & -y \\
0 &  \! \!0 \! \! & 1
\end{bmatrix}. 
\end{align*}
The cosquare of $U$ is then computed as follows: 
\begin{align*}
U^{-T} U
&=
\begin{bmatrix} 
1 &  0 & 0 \\
x &  1 & 0 \\
z+xy & y & 1
\end{bmatrix}
\begin{bmatrix} 
1 &  \! \!-x \! \! & -z \\
0 &  \! \!1 \! \! & -y \\
0 &  \! \!0 \! \! & 1
\end{bmatrix}
=
\begin{bmatrix} 
1 &  -x & -z \\
x & 1-x^2 & -y-xz \\
z+xy & y- xz-x^2y   & 1 - y^2 - z^2 -xyz
\end{bmatrix}
\!\!.
\end{align*}
\end{example}

\begin{example}
\label{ex:C(Q)-4vert}
Let $Q$ be a quiver on four linearly ordered vertices $a<b<c<d$, with %
\begin{equation*}
B_Q=\begin{bmatrix} 
0 & x & z & w \\
-x & 0 & y & v \\
-z & -y & 0 & u \\
-w & -v & -u & 0
\end{bmatrix} .
\end{equation*}
The case $x,y,z,u,v,w\ge0$ (an acyclic 4-vertex quiver) is illustrated in 
Figure~\ref{fig:Q(a,b,c,d,e,f)}. 
We emphasize that we do not require these inequalities to hold: 
the computations provided below apply for general 4-vertex quivers. 

The unipotent companion of $Q$ is 
\begin{equation*}
U\!=\!U_Q\!=\begin{bmatrix} 
1 & -x & -z & -w \\
0 & 1 & -y & -v\\
0 & 0 & 1 & -u \\
0 & 0 & 0 & 1
\end{bmatrix} . 
\end{equation*}
The cosquare of $U$ is then computed as follows: 
\begin{align*}
U^{-T}&=
\begin{bmatrix} 
1 & 0 & 0 & 0 \\
x & 1 & 0 & 0 \\
xy+z & y& 1 & 0 \\
xyu+xv+uz+w & yu+v & u & 1
\end{bmatrix}, 
\\[5pt]
U^{-T} U 
&=
\begin{bmatrix} 
1 	& 		-x	&		-z		&-w  \\[3pt]
x 	& 		1-x^2	&		-y-xz		&-v-xw\\[3pt]	
z+xy 	& 		y-xz-x^2y &		1-y^2-z^2-xyz	&-u-yv-zw-xyw\\[3pt]	
c_{41} 	& c_{42} & c_{43}	& c_{44} 
\end{bmatrix}\!,
\end{align*}
where
\begin{align*}
c_{41}&=xyu + xv + uz - w, \\ 
c_{42}&=v-xw+yu-xuz-x^2v-x^2yu, \\
c_{43}&=u-yv-zw-xzv-uz^2-uy^2-xyuz, \\
c_{44}&=1-w^2-v^2-u^2-uzw-yuv-xvw-xyuw. 
\end{align*}
\end{example}

\begin{figure}[ht]
$Q =$ \begin{tikzcd}[arrows={-stealth}, sep=4em]
  a  \arrow[r,  "x"] \arrow[d, "w", -stealth]  \arrow[dr, "z", near end, swap, outer sep=-1, -stealth]
  & b  \arrow[d, swap, "y"] \arrow[dl, "v", swap, near start, outer sep=-1]
  \\[-5pt]
   d  
   & c   \arrow[l, "u", -stealth] 
\end{tikzcd}

\caption{A $4$-vertex acyclic quiver. }
\label{fig:Q(a,b,c,d,e,f)}
\end{figure}

\begin{remark}
\label{rem:GLZ-conjugacy}
There are known algorithms for testing whether two given matrices in $\GL_n(\ZZ)$
are conjugate to each other (over~$\ZZ$). 
We will discuss this topic, and provide references, in Section~\ref{sec:invariants-of-proper-mutations},
see Remark~\ref{rem:algorithms-conjugacy}. 
\end{remark}

\newpage 

\section{Proper vertices in cyclically ordered quivers}
\label{sec:proper verts}

\begin{definition}
Let $(Q,\sigma)$ be a cyclically ordered quiver. 
We say that an oriented two-arrow path $i \rightarrow j \rightarrow k$ in $Q$ makes a \emph{right turn} at~$j$ 
if the cyclic ordering~$\sigma$ can be represented as 
\begin{equation*}
\sigma=(\dots,i,\dots,j,\dots,k,\dots). 
\end{equation*}
Otherwise, we say that the path $i \rightarrow j \rightarrow k$ makes a left turn at~$j$. 
\end{definition}

\begin{definition}
\label{def:proper}
A vertex $j$ in a COQ $Q$ is \emph{proper} (alternatively, $Q$ is proper at~$j$) 
if the following ``no-left-turn rule'' is satisfied: every oriented path $\cdots \rightarrow j \rightarrow \cdots$ makes a right turn at $j$. See Figure~\ref{fig:proper illustration}. 
\end{definition}

\begin{figure}[ht]
\begin{tikzpicture}
	\filldraw[black] (0,0)++(15:2cm) circle (2pt) node[above right] {$i$} coordinate (i);
	\filldraw[black] (0,0)++(-45:2cm) circle (2pt) node[below right] {$j$} coordinate (j);
	\filldraw[black] (0,0)++(-135:2cm) circle (2pt) node[below left] {$k$} coordinate (k);
	\filldraw[black] (0,0)++(165:2cm) circle (2pt) node[above left] {$\ell$} coordinate (l);
	\filldraw[black] (0,0)++(65:2cm) circle (2pt) node[above right] {$h$} coordinate (h);
	\filldraw[black] (0,0)++(115:2cm) circle (2pt) node[above left] {$g$} coordinate (g);
	\draw[black, dashed, decoration={markings, mark=at position 0 with {\arrow{<}}}, postaction={decorate}] (0,0) circle (2cm);
	\draw[black, -{stealth}, shorten >=3pt, shorten <= 3pt] (i) -- (j);
	\draw[black, -{stealth}, shorten >=3pt, shorten <= 3pt] (j) -- (k);
	\draw[black, -{stealth}, shorten >=3pt, shorten <= 3pt] (l) -- (k);
	\draw[black, -{stealth}, shorten >=3pt, shorten <= 3pt] (i) -- (l);
	\draw[black, -{stealth}, shorten >=3pt, shorten <= 3pt] (g) -- (l);
	\draw[black, -{stealth}, shorten >=3pt, shorten <= 3pt] (g) -- (j);
	\draw[black, -{stealth}, shorten >=3pt, shorten <= 3pt] (k) -- (h);
	\draw[black, -{stealth}, shorten >=3pt, shorten <= 3pt] (h) -- (i);
\end{tikzpicture}
\vspace{5pt}
\caption{In this COQ, the vertices $g,h,i,$ and $j$ are proper but $k$ and~$\ell$ are not. 
}
\label{fig:proper illustration}
\end{figure}


\begin{remark}
\label{rem:wiggles dont respect}
Properness of an individual vertex in a COQ is \underbar{not} preserved under wiggles:  
 a wiggle may transform a proper vertex into a non-proper one. 
 To see it, suppose that two vertices $i$ and $k$ 
 are adjacent in the cyclic ordering but not connected by an arrow~in~$Q$. 
 If an oriented path $i \rightarrow j \rightarrow k$ makes a right turn at $j$, then the same path would make a left turn at $j$ after the wiggle $(i k)$ has been performed. 
\end{remark}

\begin{definition}
\label{def:wiggle-class-proper-vertex}
We say that a vertex $j$ is \emph{proper} in a wiggle equivalence class~$\wiggleQ$
if $j$ is a proper vertex in some COQ $Q\in\wiggleQ$.
\end{definition}

\begin{example}
\label{eg:tree-proper}
Let $Q$ be a tree quiver. 
By Proposition~\ref{prop:tree wiggles}, all cyclic orderings of~$Q$ are wiggle equivalent.
It follows that every vertex of~$Q$ is proper in the corresponding wiggle equivalence class~$\wiggleQ$. 
\end{example}

Another series of examples involves the quivers that appeared in 
Proposition~\ref{prop:n-cycle with no wiggles}: 

\begin{proposition}
\label{prop:n-cycle proper vertices}
Let $Q$ be a COQ of the kind described in Proposition~\ref{prop:n-cycle with no wiggles}, 
with notation $n, r, \ell, v_i$ as specified there. 
Recall that $1-\ell\le \wind(Q,\sigma)\le r-1$, see~\eqref{eq:bounds-on-wind}. 
If $\wind(Q) < r-1$ (resp., $\wind(Q) > 1-\ell$), then every vertex~$v_i$ with $v_{i-1} \leftarrow v_i$ (resp. $v_{i-1} \rightarrow v_i$) is proper in $\wiggleQ$.

Further, if $\wind(Q) = r-1$ (resp., $\wind(Q) = 1-\ell$), then every vertex~$v_i$ with $v_{i-1} \leftarrow v_i \leftarrow v_{i+1}$ (resp. $v_{i-1} \rightarrow v_i \rightarrow v_{i+1}$) is not proper in $\wiggleQ$.
\end{proposition}

\begin{proof}
If $v_i$ is a sink or source vertex, then $v_i$ is proper in any cyclic order.
Thus we may restrict ourselves to vertices in the middle of directed paths, 
i.e., $v_{i-1} \rightarrow v_i \rightarrow v_{i+1}$ or $v_{i-1} \leftarrow v_i \leftarrow v_{i+1}$.

Suppose that $\wind(Q) = w- \ell < r-1$ for a positive integer~$w$ (the case $\wind(Q)>1-\ell$ is similar).
Fix a vertex $v_i$ with $v_{i-1} \leftarrow v_i \leftarrow v_{i+1}$.
Then $v_i$ is proper in the COQ 
\begin{equation*}
Q' = (Q, (v_{i+1}, v_i, v_{i-1}, \ldots , v_{i+1-w}, v_{i+2}, \ldots, v_{i-w})).
\end{equation*}
Since $\wind(Q') = w-\ell = \wind(Q)$, by Theorem~\ref{th:wiggle/winding} $Q'$ and $Q$ are wiggle equivalent.
Thus $v_i$ is proper in~$\wiggleQ$.

Now suppose that $\wind(Q) = r-1$ (resp. $\wind(Q) = 1-\ell$). 
Proposition~\ref{prop:n-cycle with no wiggles} implies that $Q$ is unique in its wiggle equivalence class 
and has cyclic ordering that is opposite (resp., agrees with) the order of traversal of $C$. 
The claim follows. 
\end{proof}


We will make use of Proposition~\ref{prop:n-cycle proper vertices} in Section~\ref{sec:proper COQs},
cf.\ Proposition~\ref{prop:n-cycle proper vertices 2}. 

\medskip

The material below in this section will be used in Section~\ref{sec:proper mutations},
cf.\ Proposition~\ref{prop:proper mutation for wiggle equivalent}. 

\begin{definition}
\label{def:IOsets}
 We denote by $\In(j)=\In_Q(j)$ (resp., $\Out(j)=\Out_Q(j)$) the set of vertices~$i$ for which $Q$ contains an arrow $i \rightarrow j$ 
 (resp., an arrow $j \rightarrow i$). 
\end{definition}

\begin{remark}
\label{rem:In-Out-proper}
If $j$ is a proper vertex in a COQ $(Q,\sigma)$, then the cyclic ordering~$\sigma$ 
can be obtained from a compatible linear ordering (cf.\ Definition~\ref{def:cyclic ordering})
in which all elements of $\In(j)$ precede~$j$, 
while $j$ precedes all elements of~$\Out(j)$.
\end{remark}

\begin{lemma}
\label{lem:preserve j proper wiggles}
 Let $Q$ and $Q'$ be two wiggle-equivalent COQs on n vertices. 
 Suppose that a vertex $j$ is proper in both $Q$ and $Q'$. 
 Then there exists a sequence of wiggles $w_1,\ldots , w_k$ such that $w_k \cdots w_1(Q) = Q'$ and $j$ is proper in every intermediate COQ $w_\ell w_{\ell-1} \cdots w_1 (Q)$ (for $1 \leq \ell \leq k$).
\end{lemma}

\begin{proof}
We will use the notation and construction from Definitions~\ref{def:algorithm for wiggle equiv}~\ref{def: Z/nZ circle ordering} and~\ref{def:collisions}, assuming:
\begin{itemize}
\item the root vertex is $v_0=j$ (the vertex we wish to keep proper); 
\item and the tree $T$ contains all edges adjacent to $j$.
\end{itemize} 
We claim that the sequence of wiggles found in Theorem~\ref{th:wiggle/winding} are $w_1, \ldots, w_k$. 

Wiggles do not change the relative order of any pair of vertices except the two vertices being wiggled. 
Thus a wiggle can only create a left turn $u \rightarrow j \rightarrow v$ if the wiggle involves both $u$ and $v$.
So it suffices to check that each set of colliding vertices $C$ never includes both $u$ and $v$. 

Assume for contradiction that $u,v \in \mathcal{C}(t,x)$.
Then
\begin{equation*}
\mathcal{R}(t, v) = \theta(t, j, v) = x =  -\theta(t, u, j) = \mathcal{R}(t, u) \bmod n,
\end{equation*}
which (since $\theta(t,a,b) \in (0, n)$ for $a \rightarrow b$ in $T$) implies $$\theta(t,u,j) + \theta(t,j,v) = n.$$ 
But, as $j$ is proper in $Q$ and $Q'$, we have: 
\begin{align*}
\theta(\sigma,u,j) + \theta(\sigma,j,v) &< n;\\
\theta(\sigma',u,j) + \theta(\sigma',j,v) &< n.
\end{align*}
Thus 
\begin{align*}
n &= \theta(t,u,j) + \theta(t,j,v) \\
&= (1-t) (\theta(\sigma,u,j) + \theta(\sigma,j,v)) + t (\theta(\sigma',u,j) + \theta(\sigma',j,v)) < n,
\end{align*}
 a contradiction.
\end{proof}

\newpage

\section{Proper mutations}
\label{sec:proper mutations}

The following definition can be viewed as a ``proper upgrade'' of the notion of quiver mutation to COQs. 

\begin{definition}
\label{def:proper mutation}
 A \emph{proper mutation} in a COQ $Q$ is a mutation at a proper vertex $j$, accompanied by the following modification of the cyclic ordering. 
 Let $\In(j)=\In_Q(j)$ and $\Out(j)=\Out_Q(j)$ be the sets from Definition~\ref{def:IOsets} for the original COQ $Q$. 
 (Note that after the mutation, the roles of $\In(j)$ and $\Out(j)$ will get interchanged, since all arrows incident to $j$ are reversed by the mutation.) 
 In the new cyclic ordering, the placement of all vertices besides $j$ remains the same, whereupon $j$ moves clockwise past all the vertices in $\Out(j)$ without passing any vertices in $\In(j)$. 
 In other words, we place $j$ so that the no-left-turn rule is satisfied at~$j$ in the mutated COQ~$\mu_j(Q)$, 
 keeping $j$ a proper vertex. 
 This placement of $j$ is defined canonically up to wiggle equivalence. 
 In~what follows, when dealing with a proper mutation $\mu_j$ of a COQ~$Q$,
 we denote by $\mu_j(Q)$ the COQ described above, i.e., the quiver $\mu_j(Q)$ whose cyclic ordering is determined,
 up to wiggle equivalence, by the above rule. 
\end{definition}

\begin{remark}
\label{rem:proper mutation involution}
As mentioned above, if $j$ is a proper vertex in a COQ~$Q$, then $j$ is also proper in the mutated COQ~$\mu_j(Q)$. 
Furthermore, a proper mutation at $j$ in the COQ $\mu_j(Q)$ recovers the original COQ~$Q$,
up to wiggle equivalence.
\end{remark}

\begin{example}
\label{eg:3-vert-always-proper}
Let $Q$ be a 3-vertex quiver. 
We can always assign the labels $a, b, c$ to the vertices of~$Q$ so that 
$Q$ would contain $x\ge0$ arrows $a \rightarrow b$
 and $y\ge0$ arrows $b \rightarrow c$; cf.\ Example~\ref{eg:3 vert quiver and B mat}.  
 For the cyclic ordering $(a,b,c)$, all three vertices in $Q$ are proper. 
 A (proper) mutation at each of these vertices produces a new COQ with the reversed cyclic ordering $(a,c,b)$. 
Cf.\ Proposition~\ref{pr:3vertex proper}. 
\end{example}

Our next goal is to show that the notion of proper mutation is well defined 
at the level of wiggle equivalence classes of cyclically ordered quivers.

\begin{proposition}
\label{prop:proper mutation for wiggle equivalent}
 Let $(Q,\sigma)$ and $(Q,\sigma')$ be wiggle equivalent COQs. 
Suppose that a vertex~$j$ is proper in both $(Q,\sigma)$ and~$(Q,\sigma')$ 
(cf. Remark~\ref{rem:wiggles dont respect}).
Then the COQs $\mu_j(Q,\sigma)$ and~$\mu_j(Q,\sigma')$ are wiggle equivalent. 
\end{proposition}

\begin{proof}
By Lemma~\ref{lem:preserve j proper wiggles}, we may assume that $(Q,\sigma)$ and~$(Q,\sigma')$ 
are related by a single wiggle. 
If the wiggle involves the vertex~$j$, then we may choose the placements of~$j$ 
within the cyclic orderings of mutated quivers so that $\mu_j(Q,\sigma) = \mu_j(Q,\sigma')$. 
Now suppose that $(Q,\sigma)$ and~$(Q,\sigma')$ are related by a single wiggle $(uv)$ not involving~$j$. 
Since $j$ remains proper after the wiggle, the vertices $u,v$ 
are not connected by an oriented path passing through~$j$. 
Therefore $u$ and $v$ remain non-adjacent in~$\mu_j(Q)$, and the wiggle $(u v)$ commutes with~$\mu_j$.
\end{proof}

\begin{definition}
\label{def:wiggle class mutation}
We say that $\mu_j$ is a \emph{proper mutation} for a wiggle equivalence class~$\wiggleQ$
if the vertex~$j$ is proper in~$\wiggleQ$, 
i.e., if $j$ is a proper vertex in some COQ $Q\in\wiggleQ$. \linebreak[3] 
We~then define $\mu_j(\wiggleQ)$ to be the wiggle equivalence class of~$\mu_j(Q)$.
By Proposition~\ref{prop:proper mutation for wiggle equivalent}, 
the wiggle equivalence class $\mu_j(\wiggleQ)$ is well defined, i.e., it does not depend on the choice of 
a COQ $Q\in\wiggleQ$ in which $j$ is a proper vertex.
\end{definition}

\begin{definition}
 A \emph{proper mutation equivalence class} of a COQ $Q$ consists of all COQs 
 that can be obtained from $Q$ by repeated proper mutations and wiggles.
\end{definition}

In Examples ~\ref{eg:A3}--\ref{eg:D4} below,
we begin with a quiver~$Q$ of type $A_3$, $A_4$,~or~$D_4$. 
Since $Q$ is a tree quiver, all its cyclic orderings are wiggle equivalent. 
Furthermore, each vertex is proper in the wiggle equivalence class~$\wiggleQ$ of~$Q$, cf.\ Example~\ref{eg:tree-proper}.

\begin{example}
\label{eg:A3}
Consider the COQ $Q=(a \rightarrow b \rightarrow c)$ of type~$A_3$ with the cyclic ordering $\sigma=(a,b,c)$. 
Its proper mutation class is shown in Figure~\ref{fig:mu class A3}. 
Cf.\ also Figure~\ref{fig:mu class A3 abstracted}, which shows (on the left) 
the same class with COQs identified up to relabeling.

The COQ $\mu_b(Q,\sigma)$ (see the leftmost quiver in Figure~\ref{fig:mu class A3}) 
is the oriented $3$-cycle $Q'=(c \rightarrow b \rightarrow a\rightarrow c)$ with the cyclic ordering $\sigma'=(a,c,b)$. 
Taking instead the same quiver~$Q'$ with the cyclic ordering $\sigma=(a,b,c)$, 
we get a COQ $(Q',\sigma)$ that does not lie in the proper mutation class of~$(Q,\sigma)$. 
Indeed, no vertex~is proper in $(Q',\sigma)$ and no wiggles are possible, so it is the only COQ in its proper mutation class.
\end{example}

\enlargethispage{15pt}

\begin{figure}[ht]
\vspace{3pt}

{
\newcommand{\radius}{0.4cm} 
\newcommand{\locsize}{1cm} 
\newcommand{\dbllocsize}{2.5cm} 
\newcommand{\mkwigglebox}[6]{
\draw (#2,#3) node [rectangle, fill=white, text opacity=0.0] (#1) {\vrule width \locsize height \locsize} ;
\draw (#1) node [rectangle, draw, text opacity=0.0] {\vrule width \locsize height \locsize} ;
\filldraw[black] (#1)++(135:\radius) circle (1.5pt) node[left] {$\scriptstyle #4$} coordinate (#4);
\filldraw[black] (#1)++(45:\radius) circle (1.5pt) node[right] {$\scriptstyle #5$} coordinate (#5) ;
\filldraw[black] (#1)++(-135:\radius) circle (1.5pt) node[ left] {$\scriptstyle #6$} coordinate (#6) ;
\draw[black, dashed, decoration={markings, mark=at position 0.875 with {\arrow{<}}}, opacity=0.5, postaction={decorate}] (#1) circle (\radius);
}

\newcommand{\mkwigglerect}[3]{
\draw (#2,#3) node  [rectangle, fill=white, text opacity=0.0] (#1) {\vrule width \dbllocsize height \locsize} ;
\draw (#1) node [rectangle, draw, text opacity=0.0] {\vrule width \dbllocsize height \locsize} ;
\draw (#2-0.7*\locsize,#3) node (#1r) {};
\draw (#2 + 0.7*\locsize,#3) node (#1l) {};
\filldraw[black] (#1r)++(135:\radius) circle (1.5pt) node[left] {$\scriptstyle a$} coordinate (a);
\filldraw[black] (#1r)++(45:\radius) circle (1.5pt) node[right] {$\scriptstyle b$} coordinate (b);
\filldraw[black] (#1r)++(-135:\radius) circle (1.5pt) node[left]  {$\scriptstyle c$} coordinate (c);
\draw[black, dashed, decoration={markings, mark=at position 0.875 with {\arrow{<}}}, opacity=0.5, postaction={decorate}] (#1r) circle (\radius);
}

\newcommand{\mkwiggleverts}[4]{
\filldraw[black] (#1)++(135:\radius) circle (1.5pt) node[ left] {$\scriptstyle #2$} coordinate (#2);
\filldraw[black] (#1)++(45:\radius) circle (1.5pt) node[ right] {$\scriptstyle #3$} coordinate (#3);
\filldraw[black] (#1)++(-135:\radius) circle (1.5pt) node[ left] {$\scriptstyle #4$} coordinate (#4);
\draw[black, dashed, decoration={markings, mark=at position 0.875 with {\arrow{<}}}, opacity=0.5, postaction={decorate}] (#1) circle (\radius);
}

\begin{tikzpicture}
\mkwigglebox{loc1}{-5*\locsize}{0cm}{a}{c}{b}

\draw[black, thick, -{stealth}, shorten >=3pt, shorten <= 3pt ] (a) -- (c);
\draw[black, thick, -{stealth}, shorten >=3pt, shorten <= 3pt ] (c) -- (b);
\draw[black, thick, -{stealth}, shorten >=3pt, shorten <= 3pt ] (b) -- (a);

\mkwigglerect{loc2}{-2*\locsize}{0cm} 
\draw[black, thick, -{stealth}, shorten >=3pt, shorten <= 3pt ] (a) -- (b);
\draw[black, thick, -{stealth}, shorten >=3pt, shorten <= 3pt ] (b) -- (c);
\mkwiggleverts{loc2l}{c}{b}{a}
\draw[black, thick, -{stealth}, shorten >=3pt, shorten <= 3pt ] (a) -- (b);
\draw[black, thick, -{stealth}, shorten >=3pt, shorten <= 3pt ] (b) -- (c);

\mkwigglerect{loc3}{0cm}{2.5*\locsize} 
\draw[black, thick, -{stealth}, shorten >=3pt, shorten <= 3pt ] (b) -- (a);
\draw[black, thick, -{stealth}, shorten >=3pt, shorten <= 3pt ] (b) -- (c);
\mkwiggleverts{loc3l}{c}{b}{a}
\draw[black, thick, -{stealth}, shorten >=3pt, shorten <= 3pt ] (b) -- (a);
\draw[black, thick, -{stealth}, shorten >=3pt, shorten <= 3pt ] (b) -- (c);

\mkwigglerect{loc4}{0cm}{-2.5*\locsize}
\draw[black, thick, -{stealth}, shorten >=3pt, shorten <= 3pt ] (a) -- (b);
\draw[black, thick, -{stealth}, shorten >=3pt, shorten <= 3pt ] (c) -- (b);
\mkwiggleverts{loc4l}{c}{b}{a}
\draw[black, thick, -{stealth}, shorten >=3pt, shorten <= 3pt ] (a) -- (b);
\draw[black, thick, -{stealth}, shorten >=3pt, shorten <= 3pt ] (c) -- (b);

\mkwigglerect{loc5}{2*\locsize}{0cm}
\draw[black, thick, -{stealth}, shorten >=3pt, shorten <= 3pt ] (b) -- (a);
\draw[black, thick, -{stealth}, shorten >=3pt, shorten <= 3pt ] (c) -- (b);
\mkwiggleverts{loc5l}{c}{b}{a}
\draw[black, thick, -{stealth}, shorten >=3pt, shorten <= 3pt ] (b) -- (a);
\draw[black, thick, -{stealth}, shorten >=3pt, shorten <= 3pt ] (c) -- (b);

\mkwigglebox{loc6}{5*\locsize}{0cm}{a}{b}{c}
\draw[black, thick, -{stealth}, shorten >=3pt, shorten <= 3pt ] (a) -- (b);
\draw[black, thick, -{stealth}, shorten >=3pt, shorten <= 3pt ] (b) -- (c);
\draw[black, thick, -{stealth}, shorten >=3pt, shorten <= 3pt ] (c) -- (a);

\draw[black, thick, {stealth}-{stealth}, shorten >=3pt, shorten <= 3pt ] (loc1) -- node[above]{$\mu_b$} (loc2);
\draw[black, thick, {stealth}-{stealth}, shorten >=3pt, shorten <= 3pt ] (loc2) -- node[above left]{$\mu_a$} (loc3);
\draw[black, thick, {stealth}-{stealth}, shorten >=3pt, shorten <= 3pt ] (loc2) -- node[below left]{$\mu_c$} (loc4);
\draw[black, thick, {stealth}-{stealth}, shorten >=3pt, shorten <= 3pt ] (loc3) -- node[above right]{$\mu_c$} (loc5);
\draw[black, thick, {stealth}-{stealth}, shorten >=3pt, shorten <= 3pt ] (loc3) -- node[right]{$\mu_b$} (loc4);
\draw[black, thick, {stealth}-{stealth}, shorten >=3pt, shorten <= 3pt ] (loc4) -- node[below right]{$\mu_a$} (loc5);
\draw[black, thick, {stealth}-{stealth}, shorten >=3pt, shorten <= 3pt ] (loc6) -- node[above]{$\mu_b$} (loc5);

\draw[black, thick, {stealth}-{stealth}, shorten >=3pt, shorten <= 3pt ] (loc1) -- node[above left]{$\mu_c$} (-6*\locsize, -1.5*\locsize) node[left] {$\cdots$};
\draw[black, thick, {stealth}-{stealth}, shorten >=3pt, shorten <= 3pt ] (loc1) -- node[above right]{$\mu_a$} (-6*\locsize, 1.5*\locsize) node[left] {$\cdots$};

\draw[black, thick, {stealth}-{stealth}, shorten >=3pt, shorten <= 3pt ] (loc6) -- node[above left]{$\mu_a$} (6*\locsize, 1.5*\locsize) node[right] {$\cdots$};
\draw[black, thick, {stealth}-{stealth}, shorten >=3pt, shorten <= 3pt ] (loc6) -- node[above right]{$\mu_c$} (6*\locsize, -1.5*\locsize) node[right] {$\cdots$};

\end{tikzpicture}
}
\vspace{3pt}
\caption{The proper mutation class of a COQ of type~$A_3$ discussed in Example~\ref{eg:A3}. 
The branches marked ``$\cdots$'' lead to isomorphic copies of the middle ``diamond''.}
\label{fig:mu class A3}
\end{figure}

In the rest of this section, we do not distinguish between isomorphic COQs, 
i.e., COQs that are related by an isomorphism of quivers that identifies the respective cyclic orderings.  
Put differently, we consider COQs up to relabeling of their vertices.   

\begin{figure}[ht]
\vspace{7pt}
{
\newcommand{\radius}{0.4cm} 
\newcommand{\locsize}{0.8cm} 
\newcommand{\dbllocsize}{2cm} 
\newcommand{\mkwigglebox}[6]{
\draw (#2,#3) node  [rectangle, fill=white, text opacity=0.0] (#1) {\vrule width \locsize height \locsize} ;
\draw (#1) node [rectangle, draw, text opacity=0.0] {\vrule width \locsize height \locsize} ;
\filldraw[black] (#1)++(135:\radius) circle (1.5pt) node[left] {} coordinate (#4);
\filldraw[black] (#1)++(45:\radius) circle (1.5pt) node[right] {} coordinate (#5) ;
\filldraw[black] (#1)++(-135:\radius) circle (1.5pt) node[ left] {} coordinate (#6) ;
\draw[black, dashed, decoration={markings, mark=at position 0.875 with {\arrow{<}}}, opacity=0.5, postaction={decorate}] (#1) circle (\radius);
}

\newcommand{\mkredwigglebox}[6]{
\draw (#2,#3) node  [rectangle, fill=white, text opacity=0.0] (#1) {\vrule width \locsize height \locsize} ;
\draw (#1) node [rectangle, draw, text opacity=0.0] {\vrule width \locsize height \locsize} ;
\filldraw[red] (#1)++(135:\radius) circle (1.5pt) node[left] {} coordinate (#4);
\filldraw[red] (#1)++(45:\radius) circle (1.5pt) node[right] {} coordinate (#5) ;
\filldraw[red] (#1)++(-135:\radius) circle (1.5pt) node[ left] {} coordinate (#6) ;
\draw[black, dashed, decoration={markings, mark=at position 0.875 with {\arrow{<}}}, opacity=0.5, postaction={decorate}] (#1) circle (\radius);
}

\newcommand{\mkwigglerect}[3]{
\draw (#2,#3) node  [rectangle, fill=white, text opacity=0.0] (#1) {\vrule width \dbllocsize height \locsize} ;
\draw (#1) node [rectangle, draw, text opacity=0.0] {\vrule width \dbllocsize height \locsize} ;
\draw (#2-0.7*\locsize,#3) node (#1r) {};
\draw (#2 + 0.7*\locsize,#3) node (#1l) {};
\filldraw[black] (#1r)++(135:\radius) circle (1.5pt) node[left] {} coordinate (a);
\filldraw[black] (#1r)++(45:\radius) circle (1.5pt) node[right] {} coordinate (b);
\filldraw[black] (#1r)++(-135:\radius) circle (1.5pt) node[left]  {} coordinate (c);
\draw[black, dashed, decoration={markings, mark=at position 0.875 with {\arrow{<}}}, opacity=0.5, postaction={decorate}] (#1r) circle (\radius);
}

\newcommand{\mkwiggleverts}[4]{
\filldraw[black] (#1)++(135:\radius) circle (1.5pt) node[ left] {} coordinate (#2);
\filldraw[black] (#1)++(45:\radius) circle (1.5pt) node[ right] {} coordinate (#3);
\filldraw[black] (#1)++(-135:\radius) circle (1.5pt) node[ left] {} coordinate (#4);
\draw[black, dashed, decoration={markings, mark=at position 0.875 with {\arrow{<}}}, opacity=0.5, postaction={decorate}] (#1) circle (\radius);
}

\begin{tikzpicture}
\mkwigglebox{loc1}{-5*\locsize}{0cm}{a}{c}{b}

\draw[black, thick, -{stealth}, shorten >=3pt, shorten <= 3pt ] (a) -- (c);
\draw[black, thick, -{stealth}, shorten >=3pt, shorten <= 3pt ] (c) -- (b);
\draw[black, thick, -{stealth}, shorten >=3pt, shorten <= 3pt ] (b) -- (a);

\mkwigglerect{loc2}{-2*\locsize}{0cm} 
\draw[black, thick, -{stealth}, shorten >=3pt, shorten <= 3pt ] (a) -- (b);
\draw[black, thick, -{stealth}, shorten >=3pt, shorten <= 3pt ] (b) -- (c);
\mkwiggleverts{loc2l}{c}{b}{a}
\draw[black, thick, -{stealth}, shorten >=3pt, shorten <= 3pt ] (a) -- (b);
\draw[black, thick, -{stealth}, shorten >=3pt, shorten <= 3pt ] (b) -- (c);

\mkwigglebox{loc3}{1cm}{1.5*\locsize}{a}{b}{c} 
\draw[black, thick, -{stealth}, shorten >=3pt, shorten <= 3pt ] (b) -- (a);
\draw[black, thick, -{stealth}, shorten >=3pt, shorten <= 3pt ] (b) -- (c);

\mkwigglebox{loc4}{1cm}{-1.5*\locsize}{a}{b}{c}
\draw[black, thick, -{stealth}, shorten >=3pt, shorten <= 3pt ] (a) -- (b);
\draw[black, thick, -{stealth}, shorten >=3pt, shorten <= 3pt ] (c) -- (b);

\mkredwigglebox{loc0}{6*\locsize}{0*\locsize}{a}{b}{c}
\draw[black, thick, -{stealth}, shorten >=3pt, shorten <= 3pt ] (a) -- (c);
\draw[black, thick, -{stealth}, shorten >=3pt, shorten <= 3pt ] (c) -- (b);
\draw[black, thick, -{stealth}, shorten >=3pt, shorten <= 3pt ] (b) -- (a);

\draw[black, thick, {stealth}-{stealth}, shorten >=3pt, shorten <= 3pt ] (loc1) -- node[above]{} (loc2);
\draw[black, thick, {stealth}-{stealth}, shorten >=3pt, shorten <= 3pt] (loc2) -- node[above left]{} (loc3);
\draw[black, thick, {stealth}-{stealth}, shorten >=3pt, shorten <= 3pt ] (loc2) -- node[below left]{} (loc4);
\draw[black, thick, {stealth}-{stealth}, shorten >=3pt, shorten <= 3pt ] (loc3) -- node[right]{} (loc4);

\end{tikzpicture}
}
\vspace{3pt}
\caption{Proper mutation classes of COQs of type $A_3$. 
Each box contains a wiggle equivalence class. 
Double-sided arrows represent proper mutations.
The~red vertices are not proper.}
\label{fig:mu class A3 abstracted}
\end{figure}

\vspace{-3pt}

\begin{example}
\label{eg:A4}
Consider the quiver $Q = (a \rightarrow b \rightarrow c \rightarrow d)$ of type~$A_4$, 
with the cyclic ordering $\sigma=(a,b,c,d)$.
The COQs in the proper mutation class of $(Q,\sigma)$, viewed up to relabeling  
and wiggle equivalence, are depicted in Figure~\ref{fig:mu class A4} on the left.

The quiver $\mu_b(Q)$ has just one other wiggle equivalence class, 
with a representative cyclic ordering $(a,b,c,d)$. In this COQ, the only proper vertex is the sink~$d$. 
Mutating at $d$ gives a similar COQ which again has only $d$, now the source, as a proper vertex. 
\end{example}

\begin{figure}[ht]{
\vspace{5pt}

\newcommand{\radius}{0.4cm} 
\newcommand{\locsize}{0.8cm} 
\newcommand{\mkwigglebox}[3]{
\draw (#2,#3) node [rectangle, fill=black, opacity=0] (#1) {\vrule width \locsize height \locsize} ;
\draw (#1) node [rectangle, draw, text opacity=0.0] {\vrule width \locsize height \locsize} ;
\filldraw[black] (#1)++(135:\radius) circle (1.5pt) node[left] {} coordinate (a);
\filldraw[black] (#1)++(45:\radius) circle (1.5pt) node[right] {} coordinate (b) ;
\filldraw[black] (#1)++(-45:\radius) circle (1.5pt) node[ left] {} coordinate (c) ;
\filldraw[black] (#1)++(-135:\radius) circle (1.5pt) node[ left] {} coordinate (d) ;
\draw[black, dashed, decoration={markings, mark=at position 0 with {\arrow{<}}}, opacity=0.5, postaction={decorate}] (#1) circle (\radius);
}

\newcommand{\mkredwigglebox}[3]{
\draw (#2,#3) node  [rectangle, fill=white, text opacity=0.0] (#1) {\vrule width \locsize height \locsize} ;
\draw (#1) node [rectangle, draw, text opacity=0.0] {\vrule width \locsize height \locsize} ;
\filldraw[red] (#1)++(135:\radius) circle (1.5pt) node[left] {} coordinate (a);
\filldraw[red] (#1)++(45:\radius) circle (1.5pt) node[right] {} coordinate (b) ;
\filldraw[red] (#1)++(-45:\radius) circle (1.5pt) node[ left] {} coordinate (c) ;
\filldraw[black] (#1)++(-135:\radius) circle (1.5pt) node[ left] {} coordinate (d) ;
\draw[black, dashed, decoration={markings, mark=at position 0 with {\arrow{<}}}, opacity=0.5, postaction={decorate}] (#1) circle (\radius);
}

\begin{tikzpicture}
\mkwigglebox{loc1}{-2*\locsize}{2*\locsize}

\draw[black, thick, -{stealth}, shorten >=3pt, shorten <= 3pt ] (b) -- (a);
\draw[black, thick, -{stealth}, shorten >=3pt, shorten <= 3pt ] (b) -- (c);
\draw[black, thick, -{stealth}, shorten >=3pt, shorten <= 3pt ] (d) -- (a);

\mkwigglebox{loc2}{0*\locsize}{2*\locsize}

\draw[black, thick, -{stealth}, shorten >=3pt, shorten <= 3pt ] (a) -- (b);
\draw[black, thick, -{stealth}, shorten >=3pt, shorten <= 3pt ] (c) -- (b);
\draw[black, thick, -{stealth}, shorten >=3pt, shorten <= 3pt ] (d) -- (a);

\mkwigglebox{loc3}{-2*\locsize}{0*\locsize}

\draw[black, thick, -{stealth}, shorten >=3pt, shorten <= 3pt ] (a) -- (b);
\draw[black, thick, -{stealth}, shorten >=3pt, shorten <= 3pt ] (b) -- (c);
\draw[black, thick, -{stealth}, shorten >=3pt, shorten <= 3pt ] (a) -- (d);

\mkwigglebox{loc4}{0*\locsize}{0*\locsize}

\draw[black, thick, -{stealth}, shorten >=3pt, shorten <= 3pt ] (a) -- (b);
\draw[black, thick, -{stealth}, shorten >=3pt, shorten <= 3pt ] (b) -- (c);
\draw[black, thick, -{stealth}, shorten >=3pt, shorten <= 3pt ] (d) -- (a);

\mkwigglebox{loc5}{2*\locsize}{0*\locsize}

\draw[black, thick, -{stealth}, shorten >=3pt, shorten <= 3pt ] (a) -- (b);
\draw[black, thick, -{stealth}, shorten >=3pt, shorten <= 3pt ] (b) -- (c);
\draw[black, thick, -{stealth}, shorten >=3pt, shorten <= 3pt ] (c) -- (a);
\draw[black, thick, -{stealth}, shorten >=3pt, shorten <= 3pt ] (d) -- (a);

\mkwigglebox{loc6}{0*\locsize}{-2*\locsize}

\draw[black, thick, -{stealth}, shorten >=3pt, shorten <= 3pt ] (a) -- (b);
\draw[black, thick, -{stealth}, shorten >=3pt, shorten <= 3pt ] (b) -- (c);
\draw[black, thick, -{stealth}, shorten >=3pt, shorten <= 3pt ] (b) -- (d);
\draw[black, thick, -{stealth}, shorten >=3pt, shorten <= 3pt ] (d) -- (a);

\mkredwigglebox{loc7}{6*\locsize}{2*\locsize}

\draw[black, thick, -{stealth}, shorten >=3pt, shorten <= 3pt ] (c) -- (b);
\draw[black, thick, -{stealth}, shorten >=3pt, shorten <= 3pt ] (a) -- (c);
\draw[black, thick, -{stealth}, shorten >=3pt, shorten <= 3pt ] (b) -- (a);
\draw[black, thick, -{stealth}, shorten >=3pt, shorten <= 3pt ] (d) -- (a);

\mkredwigglebox{loc8}{6*\locsize}{-0*\locsize}

\draw[black, thick, -{stealth}, shorten >=3pt, shorten <= 3pt ] (c) -- (b);
\draw[black, thick, -{stealth}, shorten >=3pt, shorten <= 3pt ] (a) -- (c);
\draw[black, thick, -{stealth}, shorten >=3pt, shorten <= 3pt ] (b) -- (a);
\draw[black, thick, -{stealth}, shorten >=3pt, shorten <= 3pt ] (a) -- (d);

\draw[black, thick, {stealth}-{stealth}, shorten >=3pt, shorten <= 3pt ] (loc1) -- node[above]{} (loc2);
\draw[black, thick, {stealth}-{stealth}, shorten >=3pt, shorten <= 3pt ] (loc1) -- node[above]{} (loc3);
\draw[black, thick, {stealth}-{stealth}, shorten >=3pt, shorten <= 3pt ] (loc2) -- node[above]{} (loc4);
\draw[black, thick, {stealth}-{stealth}, shorten >=3pt, shorten <= 3pt ] (loc2) -- node[above]{} (loc5);
\draw[black, thick, {stealth}-{stealth}, shorten >=3pt, shorten <= 3pt ] (loc3) -- node[above]{} (loc4);
\draw[black, thick, {stealth}-{stealth}, shorten >=3pt, shorten <= 3pt ] (loc3) -- node[above]{} (loc6);
\draw[black, thick, {stealth}-{stealth}, shorten >=3pt, shorten <= 3pt ] (loc4) -- node[above]{} (loc5);
\draw[black, thick, {stealth}-{stealth}, shorten >=3pt, shorten <= 3pt ] (loc4) -- node[above]{} (loc6);
\draw[black, thick, {stealth}-{stealth}, shorten >=3pt, shorten <= 3pt ] (loc5) -- node[above]{} (loc6);

\draw[black, thick, {stealth}-{stealth}, shorten >=3pt, shorten <= 3pt ] (loc7) -- node[above]{} (loc8);

\end{tikzpicture}
}
\vspace{10pt}
\caption{Proper mutation classes of COQs of type $A_4$, treated up to wiggle equivalence. 
Red vertices are not proper.}
\label{fig:mu class A4}
\end{figure}

\vspace{-10pt}

\begin{example}
\label{eg:D4}
Consider the oriented 4-cycle quiver $Q$ of type $D_4$, with arrows 
\begin{equation*}
a \rightarrow b \rightarrow c \rightarrow d\rightarrow a. 
\end{equation*}
This quiver has three wiggle equivalence classes of cyclic orderings, cf.\ Example~\ref{eg:6-orderings-D4},
with representatives 
$\sigma_1=(a, b, c, d)$, $\sigma_2=(a,b,d,c)$, $\sigma_3=(a, d, c, b)$.
(These have winding numbers 1, 2, and~3, respectively.) 
The  COQs in the proper mutation class of $(Q,\sigma_1)$, 
viewed up to relabeling and wiggle equivalence, are shown on the left of Figure~\ref{fig:mu class D4}.
Every vertex in each of these quivers is proper.

The COQ $(Q,\sigma_3)$ has no proper vertices and cannot be wiggled.
The COQ $(Q,\sigma_2)$ can be wiggled so that any given vertex is proper. 
Any single proper mutation applied to $(Q,\sigma_2)$ gives a COQ isomorphic to $(\mu_a(Q),\sigma_3)$. 
The only proper vertex in it is~$a$.

There is one additional proper mutation class of COQs of type~$D_4$ (up to wiggles and relabeling), 
represented by $(\mu_a(Q),\sigma_2)$. 
Every vertex of this COQ is not proper.
\end{example}

\begin{figure}[ht]
\vspace{5pt}
{
\newcommand{\radius}{0.4cm} 
\newcommand{\locsize}{0.8cm} 
\newcommand{\mkwigglebox}[3]{
\draw (#2,#3) node [rectangle, fill=black, opacity=0] (#1) {\vrule width \locsize height \locsize} ;
\draw (#1) node [rectangle, draw, text opacity=0.0] {\vrule width \locsize height \locsize} ;
\filldraw[black] (#1)++(135:\radius) circle (1.5pt) node[left] {} coordinate (a);
\filldraw[black] (#1)++(45:\radius) circle (1.5pt) node[right] {} coordinate (b) ;
\filldraw[black] (#1)++(-45:\radius) circle (1.5pt) node[ left] {} coordinate (c) ;
\filldraw[black] (#1)++(-135:\radius) circle (1.5pt) node[ left] {} coordinate (d) ;
\draw[black, dashed, decoration={markings, mark=at position 0 with {\arrow{<}}}, opacity=0.5, postaction={decorate}] (#1) circle (\radius);
}

\newcommand{\mkcolorwigglebox}[7]{
\draw (#2,#3) node  [rectangle, fill=white, text opacity=0.0] (#1) {\vrule width \locsize height \locsize} ;
\draw (#1) node [rectangle, draw, text opacity=0.0] {\vrule width \locsize height \locsize} ;
\filldraw[#4] (#1)++(135:\radius) circle (1.5pt) node[left] {} coordinate (a);
\filldraw[#5] (#1)++(45:\radius) circle (1.5pt) node[right] {} coordinate (b) ;
\filldraw[#6] (#1)++(-45:\radius) circle (1.5pt) node[ left] {} coordinate (c) ;
\filldraw[#7] (#1)++(-135:\radius) circle (1.5pt) node[ left] {} coordinate (d) ;
\draw[black, dashed, decoration={markings, mark=at position 0 with {\arrow{<}}}, opacity=0.5, postaction={decorate}] (#1) circle (\radius);
}

\begin{tikzpicture}
\mkwigglebox{loc1}{-3*\locsize}{3*\locsize}

\draw[black, thick, -{stealth}, shorten >=3pt, shorten <= 3pt ] (a) -- (b);
\draw[black, thick, -{stealth}, shorten >=3pt, shorten <= 3pt ] (a) -- (c);
\draw[black, thick, -{stealth}, shorten >=3pt, shorten <= 3pt ] (a) -- (d);

\mkwigglebox{loc2}{0*\locsize}{3*\locsize}

\draw[black, thick, -{stealth}, shorten >=3pt, shorten <= 3pt ] (a) -- (b);
\draw[black, thick, -{stealth}, shorten >=3pt, shorten <= 3pt ] (a) -- (c);
\draw[black, thick, -{stealth}, shorten >=3pt, shorten <= 3pt ] (d) -- (a);

\mkwigglebox{loc3}{-3*\locsize}{0*\locsize}

\draw[black, thick, -{stealth}, shorten >=3pt, shorten <= 3pt ] (b) -- (a);
\draw[black, thick, -{stealth}, shorten >=3pt, shorten <= 3pt ] (c) -- (a);
\draw[black, thick, -{stealth}, shorten >=3pt, shorten <= 3pt ] (d) -- (a);

\mkwigglebox{loc4}{0*\locsize}{0*\locsize}

\draw[black, thick, -{stealth}, shorten >=3pt, shorten <= 3pt ] (a) -- (b);
\draw[black, thick, -{stealth}, shorten >=3pt, shorten <= 3pt ] (c) -- (a);
\draw[black, thick, -{stealth}, shorten >=3pt, shorten <= 3pt ] (d) -- (a);

\mkwigglebox{loc5}{2.5*\locsize}{1.5*\locsize}

\draw[black, thick, -{stealth}, shorten >=3pt, shorten <= 3pt ] (a) -- (c);
\draw[black, thick, -{stealth}, shorten >=3pt, shorten <= 3pt ] (b) -- (c);
\draw[black, thick, -{stealth}, shorten >=3pt, shorten <= 3pt ] (d) -- (a);
\draw[black, thick, -{stealth}, shorten >=3pt, shorten <= 3pt ] (d) -- (b);
\draw[black, thick, -{stealth}, shorten >=3pt, shorten <= 3pt ] (c) -- (d);

\mkwigglebox{loc6}{4.5*\locsize}{1.5*\locsize}

\draw[black, thick, -{stealth}, shorten >=3pt, shorten <= 3pt ] (a) -- (b);
\draw[black, thick, -{stealth}, shorten >=3pt, shorten <= 3pt ] (b) -- (c);
\draw[black, thick, -{stealth}, shorten >=3pt, shorten <= 3pt ] (c) -- (d);
\draw[black, thick, -{stealth}, shorten >=3pt, shorten <= 3pt ] (d) -- (a);

\mkcolorwigglebox{loc7}{11*\locsize}{3*\locsize}{red}{red}{red}{red}

\draw[black, thick, -{stealth}, shorten >=3pt, shorten <= 3pt ] (a) -- (d);
\draw[black, thick, -{stealth}, shorten >=3pt, shorten <= 3pt ] (b) -- (d);
\draw[black, thick, -{stealth}, shorten >=3pt, shorten <= 3pt ] (c) -- (a);
\draw[black, thick, -{stealth}, shorten >=3pt, shorten <= 3pt ] (c) -- (b);
\draw[black, thick, -{stealth}, shorten >=3pt, shorten <= 3pt ] (d) -- (c);

\mkcolorwigglebox{loc8}{7.75*\locsize}{3*\locsize}{red}{black}{red}{red}

\draw[black, thick, -{stealth}, shorten >=3pt, shorten <= 3pt ] (d) -- (c);
\draw[black, thick, -{stealth}, shorten >=3pt, shorten <= 3pt ] (b) -- (c);
\draw[black, thick, -{stealth}, shorten >=3pt, shorten <= 3pt ] (a) -- (d);
\draw[black, thick, -{stealth}, shorten >=3pt, shorten <= 3pt ] (a) -- (b);
\draw[black, thick, -{stealth}, shorten >=3pt, shorten <= 3pt ] (c) -- (a);

\mkwigglebox{loc9}{7.75*\locsize}{0*\locsize}

\draw[black, thick, -{stealth}, shorten >=3pt, shorten <= 3pt ] (a) -- (b);
\draw[black, thick, -{stealth}, shorten >=3pt, shorten <= 3pt ] (b) -- (d);
\draw[black, thick, -{stealth}, shorten >=3pt, shorten <= 3pt ] (d) -- (c);
\draw[black, thick, -{stealth}, shorten >=3pt, shorten <= 3pt ] (c) -- (a);

\mkcolorwigglebox{loc10}{11*\locsize}{-0*\locsize}{red}{red}{red}{red}

\draw[black, thick, -{stealth}, shorten >=3pt, shorten <= 3pt ] (b) -- (a);
\draw[black, thick, -{stealth}, shorten >=3pt, shorten <= 3pt ] (c) -- (b);
\draw[black, thick, -{stealth}, shorten >=3pt, shorten <= 3pt ] (d) -- (c);
\draw[black, thick, -{stealth}, shorten >=3pt, shorten <= 3pt ] (a) -- (d);

\draw[black, thick, {stealth}-{stealth}, shorten >=3pt, shorten <= 3pt ] (loc1) -- node[above]{} (loc2);
\draw[black, thick, {stealth}-{stealth}, shorten >=3pt, shorten <= 3pt ] (loc1) -- node[above]{} (loc3);
\draw[black, thick, {stealth}-{stealth}, shorten >=3pt, shorten <= 3pt ] (loc2) -- node[above]{} (loc4);
\draw[black, thick, {stealth}-{stealth}, shorten >=3pt, shorten <= 3pt ] (loc2) -- node[above]{} (loc5);
\draw[black, thick, {stealth}-{stealth}, shorten >=3pt, shorten <= 3pt ] (loc3) -- node[above]{} (loc4);
\draw[black, thick, {stealth}-{stealth}, shorten >=3pt, shorten <= 3pt ] (loc4) -- node[above]{} (loc5);
\draw[black, thick, {stealth}-{stealth}, shorten >=3pt, shorten <= 3pt ] (loc5) -- node[above]{} (loc6);

\draw[black, thick, {stealth}-{stealth}, shorten >=3pt, shorten <= 3pt ] (loc8) -- node[above]{} (loc9);
\end{tikzpicture}
}
\vspace{10pt}
\caption{Proper mutation classes of COQs of type $D_4$,
considered up to wiggle equivalence. Red vertices are not proper.}
\label{fig:mu class D4}
\end{figure}

We next discuss a couple of examples of quivers of affine types $\tilde A(n_1,n_2)$, cf.~\cite{cats1}. 

\begin{example}
\label{eg: proper mu classes of A(2,1)}
Let $Q$ be a quiver of type $\tilde A(2,1)$ with arrows $a \rightarrow b \rightarrow c$ and $a \rightarrow c$. 
Up to relabeling, there are only two quivers mutation-equivalent to~$Q$.
Each has two cyclic orderings, which fall into $3$ proper mutation classes shown in Figure~\ref{fig:proper mu classes of A(2,1)}.
Fix the cyclic ordering $\sigma=(a,b,c)$.
Then every vertex in $(Q, \sigma)$ is proper. Mutating at $b$ results in the COQ $(\mu_b(Q), (a, c, b))$, where again every vertex is proper.

The other cyclic ordering of $Q$ is $(a,c,b)$. 
In this COQ, only the sink $a$ and the source $c$ are proper.

The other cyclic ordering of $\mu_b(Q)$ is $(a,b,c)$. 
No vertex in this COQ is proper.
\end{example}

\begin{figure}[ht]
\vspace{5pt}

{
\newcommand{\radius}{0.4cm} 
\newcommand{\locsize}{0.8cm} 
\newcommand{\mkwigglebox}[3]{
\draw (#2,#3) node [rectangle, fill=black, opacity=0] (#1) {\vrule width \locsize height \locsize} ;
\draw (#1) node [rectangle, draw, text opacity=0.0] {\vrule width \locsize height \locsize} ;
\filldraw[black] (#1)++(135:\radius) circle (1.5pt) node[left] {} coordinate (a);
\filldraw[black] (#1)++(45:\radius) circle (1.5pt) node[right] {} coordinate (b) ;
\filldraw[black] (#1)++(-135:\radius) circle (1.5pt) node[ left] {} coordinate (c) ;
\draw[black, dashed, decoration={markings, mark=at position 0 with {\arrow{<}}}, opacity=0.5, postaction={decorate}] (#1) circle (\radius);
}

\newcommand{\mkcolorwigglebox}[6]{
\draw (#2,#3) node  [rectangle, fill=white, text opacity=0.0] (#1) {\vrule width \locsize height \locsize} ;
\draw (#1) node [rectangle, draw, text opacity=0.0] {\vrule width \locsize height \locsize} ;
\filldraw[#4] (#1)++(135:\radius) circle (1.5pt) node[left] {} coordinate (a);
\filldraw[#5] (#1)++(45:\radius) circle (1.5pt) node[right] {} coordinate (b) ;
\filldraw[#6] (#1)++(-135:\radius) circle (1.5pt) node[ left] {} coordinate (c) ;
\draw[black, dashed, decoration={markings, mark=at position 0 with {\arrow{<}}}, opacity=0.5, postaction={decorate}] (#1) circle (\radius);
}

\begin{tikzpicture}
\mkwigglebox{loc1}{-3*\locsize}{0cm}
\draw[black, thick, -{stealth}, shorten >=3pt, shorten <= 3pt ] (a) -- (b);
\draw[black, thick, -{stealth}, shorten >=3pt, shorten <= 3pt ] (b) -- (c);
\draw[black, thick, -{stealth}, shorten >=3pt, shorten <= 3pt ] (a) -- (c);

\mkwigglebox{loc2}{0*\locsize}{0cm}
\draw[black, thick, -{stealth}, shorten >=3pt, shorten <= 3pt ] (a) -- (b);
\draw[black, thick, -{stealth}, shorten >=3pt, shorten <= 3pt ] (b) -- (c);
\draw[black, thick, double, -{stealth}, shorten >=3pt, shorten <= 3pt ] (c) -- (a);

\mkcolorwigglebox{loc3}{4*\locsize}{-0cm}{black}{black}{red}
\draw[black, thick, -{stealth}, shorten >=3pt, shorten <= 3pt ] (a) -- (c);
\draw[black, thick, -{stealth}, shorten >=3pt, shorten <= 3pt ] (c) -- (b);
\draw[black, thick, -{stealth}, shorten >=3pt, shorten <= 3pt ] (a) -- (b);

\mkcolorwigglebox{loc4}{8*\locsize}{-0cm}{red}{red}{red}
\draw[black, thick, -{stealth}, shorten >=3pt, shorten <= 3pt ] (b) -- (a);
\draw[black, thick, -{stealth}, shorten >=3pt, shorten <= 3pt ] (c) -- (b);
\draw[black, thick, double, -{stealth}, shorten >=3pt, shorten <= 3pt ] (a) -- (c);

\draw[black, thick, {stealth}-{stealth}, shorten >=3pt, shorten <= 3pt ] (loc1) -- (loc2);
\end{tikzpicture}
}
\vspace{10pt}
\caption{Proper mutation classes of COQs of type $\tilde A(2,1)$, considered up to wiggle equivalence. 
The red vertices are not proper.}
\label{fig:proper mu classes of A(2,1)}
\end{figure}

\vspace{-10pt}

\begin{example}[cf.\ Figure~\ref{fig:mu class C}]
\label{eg:acyclic 4 cycle}
Let $Q$ be a quiver of type $\tilde A(3,1)$ (cf.\ \cite[Figure~16]{cats1})
with the vertices and arrows $a \rightarrow b \rightarrow c \rightarrow d$, $a \rightarrow d$. 
This quiver has $3$ cyclic orderings up to wiggle equivalence, 
represented by $(a,d,c, b), (a,b,c,d)$  and $(a,b,d,c)$.  
In every cyclic ordering, the mutations at the sink $d$ and the source $a$ are proper, and yield relabelings of the same COQs.
In the COQ $(Q, (a,d,c,b))$, there are no other proper vertices.
Every vertex of every COQ in the proper mutation class of the COQ $(C, (a,b,c,d))$ is proper.
In the COQ $(Q, (a,b,d,c))$, every vertex is proper, but mutating at $b$ (resp.,~$c$) results in a COQ where $c$ (resp. $b$) is not proper.

The quiver $Q'=\mu_b(Q)$ has $4$ distinct wiggle equivalence classes of cyclic orderings.
Both the COQs $(Q', (a, b,c,d))$ and $(Q', (a,d,b, c))$ have no proper vertices besides the sink $d$.
The COQs $(Q',(a,c,b,d))$ and $(Q',(a, d, c, b))$ are in the proper mutation classes of  $(Q, (a,b,c,d))$ and $(Q, (a,b,d,c))$ respectively.

The quiver $Q''=\mu_a (\mu_b(Q))$ has two wiggle equivalence classes of cyclic orderings.
The COQ $(Q'', (a,b, d,c))$ is in the proper mutation class of$(Q, (a,b,c,d))$, so every vertex is proper.
By contrast, only the sink vertex $b$ is proper in $(Q'', (a,b,c,d))$.
\end{example}

\enlargethispage{10pt}

\begin{figure}[ht]
\vspace{2pt}
{
\newcommand{\radius}{0.4cm} 
\newcommand{\locsize}{0.8cm} 

\newcommand{\mkwigglebox}[7]{
\draw (#2,#3) node  [rectangle, fill=white, text opacity=0.0] (#1) {\vrule width \locsize height \locsize} ;
\draw (#1) node [rectangle, draw, text opacity=0.0] {\vrule width \locsize height \locsize} ;
\draw[black, dashed, decoration={markings, mark=at position 0 with {\arrow{<}}}, opacity=0.5, postaction={decorate}] (#1) circle (\radius);
\filldraw[#4] (#1)++(135:\radius) circle (1.5pt) node[left] {} coordinate (a);
\filldraw[#5] (#1)++(45:\radius) circle (1.5pt) node[right] {} coordinate (b) ;
\filldraw[#6] (#1)++(-45:\radius) circle (1.5pt) node[ left] {} coordinate (c) ;
\filldraw[#7] (#1)++(-135:\radius) circle (1.5pt) node[ left] {} coordinate (d) ;
}

\begin{tikzpicture}
\mkwigglebox{loc1}{1*\locsize}{-2*\locsize}{black}{black}{black}{black}
\draw[black, thick, -{stealth}, shorten >=3pt, shorten <= 3pt] (a) -- (b);
\draw[black, thick, -{stealth}, shorten >=3pt, shorten <= 3pt] (b) -- (c);
\draw[black, thick, -{stealth}, shorten >=3pt, shorten <= 3pt] (c) -- (d);
\draw[black, thick, -{stealth}, shorten >=3pt, shorten <= 3pt] (a) -- (d);

\mkwigglebox{loc2}{3*\locsize}{-3*\locsize}{black}{black}{black}{black}
\draw[black, thick, -{stealth}, shorten >=3pt, shorten <= 3pt] (a) -- (b);
\draw[black, thick, -{stealth}, shorten >=3pt, shorten <= 3pt] (b) -- (c);
\draw[black, thick, -{stealth}, shorten >=3pt, shorten <= 3pt] (c) -- (a);
\draw[black, thick, -{stealth}, shorten >=3pt, shorten <= 3pt] (d) -- (b);
\draw[black, thick, -{stealth}, shorten >=3pt, shorten <= 3pt] (d) -- (a);

\mkwigglebox{loc2o}{3*\locsize}{-1*\locsize}{black}{black}{black}{black}
\draw[black, thick, -{stealth}, shorten >=3pt, shorten <= 3pt] (a) -- (b);
\draw[black, thick, -{stealth}, shorten >=3pt, shorten <= 3pt] (b) -- (c);
\draw[black, thick, -{stealth}, shorten >=3pt, shorten <= 3pt] (c) -- (a);
\draw[black, thick, -{stealth}, shorten >=3pt, shorten <= 3pt] (b) -- (d);
\draw[black, thick, -{stealth}, shorten >=3pt, shorten <= 3pt] (a) -- (d);

\mkwigglebox{loc3}{5*\locsize}{-3*\locsize}{black}{black}{black}{black}
\draw[black, thick, double, -{stealth}, shorten >=3pt, shorten <= 3pt] (c) -- (d);
\draw[black, thick, -{stealth}, shorten >=3pt, shorten <= 3pt] (d) -- (a);
\draw[black, thick, -{stealth}, shorten >=3pt, shorten <= 3pt] (a) -- (c);
\draw[black, thick, -{stealth}, shorten >=3pt, shorten <= 3pt] (a) -- (b);

\mkwigglebox{loc3o}{5*\locsize}{-1*\locsize}{black}{black}{black}{black}
\draw[black, thick, double, -{stealth}, shorten >=3pt, shorten <= 3pt] (c) -- (d);
\draw[black, thick, -{stealth}, shorten >=3pt, shorten <= 3pt] (d) -- (a);
\draw[black, thick, -{stealth}, shorten >=3pt, shorten <= 3pt] (a) -- (c);
\draw[black, thick, -{stealth}, shorten >=3pt, shorten <= 3pt] (b) -- (a);

\mkwigglebox{loc4}{7*\locsize}{2*\locsize}{black}{black}{black}{black}
\draw[black, thick, -{stealth}, shorten >=3pt, shorten <= 3pt] (a) -- (b);
\draw[black, thick, -{stealth}, shorten >=3pt, shorten <= 3pt] (b) -- (d);
\draw[black, thick, -{stealth}, shorten >=3pt, shorten <= 3pt] (d) -- (c);
\draw[black, thick, -{stealth}, shorten >=3pt, shorten <= 3pt] (a) -- (c);

\mkwigglebox{loc5}{9*\locsize}{2*\locsize}{black}{black}{red}{black}
\draw[black, thick, -{stealth}, shorten >=3pt, shorten <= 3pt] (a) -- (b);
\draw[black, thick, -{stealth}, shorten >=3pt, shorten <= 3pt] (c) -- (b);
\draw[black, thick, -{stealth}, shorten >=3pt, shorten <= 3pt] (a) -- (c);
\draw[black, thick, -{stealth}, shorten >=3pt, shorten <= 3pt] (c) -- (d);
\draw[black, thick, -{stealth}, shorten >=3pt, shorten <= 3pt] (d) -- (a);

\mkwigglebox{loc5o}{8*\locsize}{0*\locsize}{red}{black}{black}{black}
\draw[black, thick, -{stealth}, shorten >=3pt, shorten <= 3pt] (b) -- (a);
\draw[black, thick, -{stealth}, shorten >=3pt, shorten <= 3pt] (b) -- (c);
\draw[black, thick, -{stealth}, shorten >=3pt, shorten <= 3pt] (a) -- (c);
\draw[black, thick, -{stealth}, shorten >=3pt, shorten <= 3pt] (c) -- (d);
\draw[black, thick, -{stealth}, shorten >=3pt, shorten <= 3pt] (d) -- (a);

\mkwigglebox{loc6}{-3*\locsize}{2*\locsize}{black}{red}{red}{black}
\draw[black, thick, -{stealth}, shorten >=3pt, shorten <= 3pt] (b) -- (a);
\draw[black, thick, -{stealth}, shorten >=3pt, shorten <= 3pt] (c) -- (b);
\draw[black, thick, -{stealth}, shorten >=3pt, shorten <= 3pt] (d) -- (a);
\draw[black, thick, -{stealth}, shorten >=3pt, shorten <= 3pt] (d) -- (c);

\mkwigglebox{loc7}{-1*\locsize}{2*\locsize}{red}{black}{red}{red}
\draw[black, thick, -{stealth}, shorten >=3pt, shorten <= 3pt] (a) -- (b);
\draw[black, thick, -{stealth}, shorten >=3pt, shorten <= 3pt] (c) -- (b);
\draw[black, thick, -{stealth}, shorten >=3pt, shorten <= 3pt] (d) -- (a);
\draw[black, thick, -{stealth}, shorten >=3pt, shorten <= 3pt] (c) -- (d);
\draw[black, thick, -{stealth}, shorten >=3pt, shorten <= 3pt] (a) -- (c);

\mkwigglebox{loc7o}{-1*\locsize}{0*\locsize}{red}{black}{red}{red}
\draw[black, thick, -{stealth}, shorten >=3pt, shorten <= 3pt] (b) -- (a);
\draw[black, thick, -{stealth}, shorten >=3pt, shorten <= 3pt] (b) -- (c);
\draw[black, thick, -{stealth}, shorten >=3pt, shorten <= 3pt] (d) -- (a);
\draw[black, thick, -{stealth}, shorten >=3pt, shorten <= 3pt] (c) -- (d);
\draw[black, thick, -{stealth}, shorten >=3pt, shorten <= 3pt] (a) -- (c);

\mkwigglebox{loc8}{1*\locsize}{2*\locsize}{red}{black}{red}{red}
\draw[black, thick, -{stealth}, shorten >=3pt, shorten <= 3pt] (a) -- (b);
\draw[black, thick, -{stealth}, shorten >=3pt, shorten <= 3pt] (d) -- (b);
\draw[black, thick, -{stealth}, shorten >=3pt, shorten <= 3pt] (d) -- (a);
\draw[black, thick, -{stealth}, shorten >=3pt, shorten <= 3pt] (c) -- (d);
\draw[black, thick, -{stealth}, shorten >=3pt, shorten <= 3pt] (a) -- (c);

\mkwigglebox{loc8o}{1*\locsize}{0*\locsize}{red}{black}{red}{red}
\draw[black, thick, -{stealth}, shorten >=3pt, shorten <= 3pt] (b) -- (a);
\draw[black, thick, -{stealth}, shorten >=3pt, shorten <= 3pt] (b) -- (d);
\draw[black, thick, -{stealth}, shorten >=3pt, shorten <= 3pt] (d) -- (a);
\draw[black, thick, -{stealth}, shorten >=3pt, shorten <= 3pt] (c) -- (d);
\draw[black, thick, -{stealth}, shorten >=3pt, shorten <= 3pt] (a) -- (c);

\mkwigglebox{loc9}{3*\locsize}{2*\locsize}{red}{black}{red}{red}
\draw[black, thick, double, -{stealth}, shorten >=3pt, shorten <= 3pt] (d) -- (c);
\draw[black, thick, -{stealth}, shorten >=3pt, shorten <= 3pt] (a) -- (d);
\draw[black, thick, -{stealth}, shorten >=3pt, shorten <= 3pt] (c) -- (a);
\draw[black, thick, -{stealth}, shorten >=3pt, shorten <= 3pt] (a) -- (b);

\mkwigglebox{loc9o}{5*\locsize}{2*\locsize}{red}{black}{red}{red}
\draw[black, thick, double, -{stealth}, shorten >=3pt, shorten <= 3pt] (d) -- (c);
\draw[black, thick, -{stealth}, shorten >=3pt, shorten <= 3pt] (a) -- (d);
\draw[black, thick, -{stealth}, shorten >=3pt, shorten <= 3pt] (c) -- (a);
\draw[black, thick, -{stealth}, shorten >=3pt, shorten <= 3pt] (b) -- (a);

\draw[black, thick, {stealth}-{stealth}, shorten >=3pt, shorten <= 3pt ] (loc1) -- node[above]{} (loc2);
\draw[black, thick, {stealth}-{stealth}, shorten >=3pt, shorten <= 3pt ] (loc2) -- node[above]{} (loc3);
\draw[black, thick, {stealth}-{stealth}, shorten >=3pt, shorten <= 3pt ] (loc2o) -- node[above]{} (loc2);
\draw[black, thick, {stealth}-{stealth}, shorten >=3pt, shorten <= 3pt ] (loc3o) -- node[above]{} (loc3);
\draw[black, thick, {stealth}-{stealth}, shorten >=3pt, shorten <= 3pt ] (loc1) -- node[above]{} (loc2o);
\draw[black, thick, {stealth}-{stealth}, shorten >=3pt, shorten <= 3pt ] (loc2o) -- node[above]{} (loc3o);

\draw[black, thick, {stealth}-{stealth}, shorten >=3pt, shorten <= 3pt ] (loc4) -- node[above]{} (loc5);
\draw[black, thick, {stealth}-{stealth}, shorten >=3pt, shorten <= 3pt ] (loc5o) -- node[above]{} (loc5);
\draw[black, thick, {stealth}-{stealth}, shorten >=3pt, shorten <= 3pt ] (loc5o) -- node[above]{} (loc4);

\draw[black, thick, {stealth}-{stealth}, shorten >=3pt, shorten <= 3pt ] (loc7o) -- node[above]{} (loc7);

\draw[black, thick, {stealth}-{stealth}, shorten >=3pt, shorten <= 3pt ] (loc8o) -- node[above]{} (loc8);

\draw[black, thick, {stealth}-{stealth}, shorten >=3pt, shorten <= 3pt ] (loc9o) -- node[above]{} (loc9);
\end{tikzpicture}
}
\vspace{5pt}
\caption{Proper mutation classes of COQs of type~$\tilde A(3,1)$, 
considered up to relabeling and wiggle equivalence.
The red vertecies are not proper.}
\label{fig:mu class C}
\end{figure}

\newpage
\section{Congruence classes and conjugacy classes}
\label{sec:invariants-of-proper-mutations}


\begin{theorem}
\label{thm:I-N-action}
 Proper mutations and wiggles preserve the integral congruence class of a unipotent companion of a cyclically ordered quiver.
\end{theorem}

Put slightly differently, proper mutations of wiggle equivalence classes of COQs
preserve the integral congruence class of associated unipotent companions. 

\begin{proof} 
Let $k$ be a proper vertex in a COQ $(Q,\sigma)$. 
By Remark~\ref{rem:In-Out-proper}, we can choose a linear ordering (denoted~$<$) on the vertices of~$Q$ 
that is compatible with the cyclic ordering~$\sigma$ and 
satisfies $i<k$ for $i\in\In(k)$
and  $k <j$ for $j\in\Out(k)$.
For the mutated COQ $Q'=\mu_k(Q)$, we choose a linear ordering $<'$ such that 
$k <' i$ for all vertices $i\ne k$, and otherwise $<'$ agrees with $<$. 

We will use the notation $U = U_Q = (u_{ij})$ and $U' = U_{Q'} = (u'_{ij})$. 
Here and below, the rows and columns of matrices associated with $Q$ and~$Q'$ 
are ordered using $<$ and $<'$ respectively.
Our goal is to show that $U$ and $U'$ are congruent over~$\ZZ$. 
We note that for the purposes of establishing congruence, the ordering of the rows and columns of $U$ and~$U'$
does not matter, as long as the rows and the columns are permuted in the same way. 

We denote $B=B_Q=(b_{ij})$ and let $N=U-I=(n_{ij})$, the strictly upper-triangular part of $U$ (or of $-B$). 
We also denote
\begin{equation*}
\varepsilon_i = 
\begin{cases} -1 & \text{if } i=k; \\ 
1 & \text{else.}
\end{cases}
\end{equation*}

\begin{lemma}
\label{lem:U mutation matrix}
We have 
$u'_{ij} = \varepsilon_i \varepsilon_j u_{ij} - n_{i k} u_{k j} \varepsilon_j - \varepsilon_i u_{i k} n_{j k} + n_{i k} n_{j k}$.
\end{lemma}

\begin{proof}
In light of Definitions~~\ref{def:matrix-mut} and~\ref{def:unipotent companions} , we have
(recall that vertex $k$ is minimal with respect to~$<'$): 
\begin{equation}
\label{eq:muk U def}
u'_{ij} = \begin{cases}
1 & \text{if } i=j; \\
b_{kj} & \text{if } i = k \neq j; \\ 
-b_{ij} - [b_{ik}]_+[b_{kj}]_+ + [b_{ik}]_-[b_{kj}]_- & \text{if } k \ne i <' j; 
\\
0 & \text{if } i>' j.  
\end{cases}
\end{equation}

Let $u_{ij}''= \varepsilon_i \varepsilon_j u_{ij} - n_{i k} u_{k j} \varepsilon_j - \varepsilon_i u_{i k} n_{j k} + n_{i k} n_{j k}$. 
To establish the equality $u''_{ij} =  u'_{ij} $,
we check each case of equation~\eqref{eq:muk U def} separately:
\begin{itemize}[leftmargin=.2in]
\item 
If $i=j$, then  $u''_{ij} =  1 - 0 - u_{ik} n_{ik} + n_{ik}^2 = 1 = u'_{ij} $. 
\item
If $i=k \neq j$, then $ u''_{ij} = - u_{kj} - 0 + n_{jk} + 0 = b_{kj} = u'_{ij} $.
\item 
If $k \neq i <' j $, then 
$u''_{ij} = u_{ij} - n_{ik} u_{kj} - u_{ik} n_{jk} + n_{ik} n_{jk} = u_{ij} - n_{ik} u_{kj}$. 
For these vertices $i$ and~$j$, we have $[b_{ik}]_+[b_{kj}]_+= n_{ik} u_{kj}$ and $[b_{ik}]_-[b_{kj}]_- = 0$, both by construction of~$<$. 
So $u'_{ij} =-b_{ij} -n_{ik} u_{kj} +0 = u_{ij}-n_{ik}u_{kj}=u_{ij}''$.
\item
If $i\ne k$ and $j=k$, then $u_{ij}''=-u_{ik}+n_{ik}-0+0=0=u_{ij}'$. 
\\
If  $i,j \neq k$ and $j <i$, then 
$u''_{ij} = 0 - 0 - u_{ik} n_{jk} + n_{ik} n_{jk} = 0=u'_{ij}$. \qedhere
\end{itemize}
\end{proof}

\pagebreak[3]

To complete the proof of Theorem~\ref{thm:I-N-action}, we observe that
Lemma~\ref{lem:U mutation matrix} can be restated as follows: 
\begin{align*}
\pi^T U' \pi &= J U J - N E_{kk} U J - J U E_{kk} N^T + N E_{kk} N^T = (J - N E_{kk}) U (J - E_{kk} N^T)
\\[-20pt]
\end{align*}
where 
\begin{itemize}
\item 
$\pi$ is a permutation matrix such that 
$\pi^T U' \pi$ is obtained from $U'$ by reordering of its rows and columns according to the linear ordering~$<$
(as opposed to~$<'$); 
\item $J$ is the $n \times n$ diagonal matrix with diagonal entries $\varepsilon_1,\dots,\varepsilon_n$, and 
\item $E_{kk}$ is the $n \times n$ diagonal matrix whose sole nonzero entry is $1$ in row and column~$k$. 
\end{itemize}
(Here we used that $E_{kk} U E_{kk} = E_{kk}$ because $u_{kk}=1$.)
\end{proof}

\begin{remark}
\label{rem:U mutates like B}
The above proof is similar to the argument in \cite[p.~34]{fwz1-3}, which uses 
the matrix $E_j$ defined (for $\varepsilon=-1$) by setting $(E_k)_{ik} = \max(0, b_{ik})$ 
for all~$i$ and letting all other entries of $E_k$ be equal to~$0$.
This matrix is then used in \cite{fwz1-3} in the identity $(J + E_k) B_Q (J + E_k^T) = B_{\mu_k(Q)}$. 
Under our choice of linear ordering, $E_k = N E_{kk}$.
\end{remark}

\begin{remark}
Theorem~\ref{thm:I-N-action} asserts that proper mutation equivalence of COQs
implies integral congruence of their unipotent companions. 
The converse is false: 
it is easy to find pairs of COQs whose unipotent companions are integrally congruent 
while the quivers (ignoring the ordering) are not mutation equivalent. 
It is harder, but possible, to find pairs of this kind where all the vertices in both COQs are proper,
cf.\ Example~\ref{eg:proper-COQs-congr-but-mut-inequiv}. 
\end{remark}

\medskip


We are not aware of algorithms for detecting integral congruence,~i.e., deciding whether two given matrices in $\GL(n,\ZZ)$ are congruent to each other over~$\ZZ$. 
This makes it impractical to directly use Theorem~\ref{thm:I-N-action} 
to establish mutation (in)equivalence for specific pairs of quivers. 

We will instead replace integral congruence by some necessary conditions that can be readily checked.
The first and most powerful of these conditions utilizes the notion of a cosquare introduced in Definition~\ref{def:cosquare}: 

\begin{corollary}
\label{cor:proper-mutations-preserve-cosquare}
Proper mutations and wiggles preserve the $\GL(n,\ZZ)$ conjugacy class 
of the cosquare of the unipotent companion of a COQ.
\end{corollary}

\begin{proof}
Combine Theorem~\ref{thm:I-N-action} with Lemma~\ref{lem:C-sim=Ucong}. 
\end{proof}

\begin{remark}
\label{rem:algorithms-conjugacy}
The conjugacy problem in $\GL(n,\mathbb{Z})$ 
has an algorithmic solution whose idea goes back to F.~Grunewald~\cite{Grunewald}
(cf.~also R.~A.~Sarkisyan \cite{Sarkisjan} and F.~Grunewald--D.~Segal~\cite{Grunewald-Segal}).
It reduces the problem of deciding whether two matrices in $\GL(n,\ZZ)$ 
are conjugate to each other (over~$\ZZ$) 
to the isomorphism problem for (integral) modules over truncated polynomial rings $\mathcal{O}_K[t]/(t^\ell)$,
where $\mathcal{O}_K$ is the ring of algebraic integers in a number field~$K$. 
An~algorithm based on this approach was fully developed and implemented in \textsc{Magma}
by B.~Eick, T.~Hofmann, and E.~A.~O'Brien~\cite{EHO}. 
(For another, open source, software, see \cite[Section~9.5]{BHJ}.) 
We used the implementation of~\cite{EHO} to perform computational experiments
for various families of quivers. 
%
\end{remark}


\begin{remark}
Apparently, there is no settled ``canonical form,'' i.e., a distinguished choice of a representative, 
in a given conjugacy class in $\GL(n, \ZZ)$, see \cite[Problem~7.3]{EHO}. 
\end{remark}

\begin{remark}
\label{rem:conjugacy-over-Q-or-C}
The $\GL(n,\mathbb{Z})$ conjugacy class of an $n\times n$ matrix 
is contained in (hence determines) its $\GL(n,\mathbb{Q})$ conjugacy class,
which in turn determines the $\GL(n,\mathbb{C})$ conjugacy class. 
As we move from $\ZZ$ to $\QQ$ and then to~$\CC$, 
the conjugacy class of a given matrix
(in our applications, of the cosquare of a unipotent companion)
becomes much easier to compute---but the corresponding (proper) mutation invariants of quivers 
become substantially less powerful. 

Recall that the $\GL(n,\mathbb{Q})$ (resp., $\GL(n,\mathbb{C})$) conjugacy class
of a matrix is captured by its Frobenius normal form (resp., Jordan canonical form). 
A~$\GL(n,\mathbb{Q})$ conjugacy class is a disjoint union of $\GL(n,\mathbb{Z})$ conjugacy classes. 
This union may be infinite, in which case a lot of information is lost when passing from a $\GL(n,\mathbb{Z})$ class
to~a $\GL(n,\mathbb{Q})$ class. 
As pointed out in \cite[p.~755]{EHO}, the Jordan--Zassen\-haus theorem \cite{Zassenhaus} implies that 
this happens if and only if the matrices involved are not semisimple, 
i.e., when their minimal polynomial has repeated irreducible~factors. 
\end{remark}

\begin{remark}
\label{rem:rank(B)}
It is well known that the rank of the exchange matrix~$B_Q$ is a mutation invariant, see 
\cite[Theorem 2.8.3]{fwz1-3} or \cite[Lemma~3.2]{ca3}. 
This invariant 
can be easily recovered
from the Jordan normal form of the cosquare~$U_Q^{-T}U_Q$.
Specifically, the corank of~$B_Q$ is equal to the number of Jordan blocks of the cosquare
that correspond to the eigenvalue~$1$. 
\end{remark}

\begin{example}
\label{eg:2mm}
For a positive integer $m$, let $Q_m$ be the following quiver on a 3-vertex linearly ordered set $\{a<b<c\}$:
\begin{equation}
\label{eq:Qm-2mm}
 \begin{tikzcd}[arrows={-stealth}, sep=2em]
  a  \arrow[r,  "m"]   & b  \arrow[r, "m"] & c \arrow[ll, bend left=45, swap, "2"] 
\end{tikzcd}.
\end{equation}
The unipotent companion of $Q_m$ and its cosquare $C_m$ are given by
\begin{align*}
U_m &=U_{Q_m}= \begin{bmatrix} 1 & -m & 2 \\ 0 & 1 & -m \\ 0 &  0 & 1 \end{bmatrix}, 
\\
C_m&=U_m^{-T}U_m= \begin{bmatrix} 1 & -m & 2 \\ m & -m^2+1 & m \\ m^2-2 & -m^3+3m & m^2-3\end{bmatrix}.
\end{align*}
The cosquare $C_m$ has the same characteristic polynomial for all~$m$: 
\begin{equation*}
\det(tI-C_m) = (t - 1)(t + 1)^2. 
\end{equation*}
The Jordan normal form of $C_m$ carries a little bit more information: it is given by
\begin{equation*}
\begin{bmatrix}
1 & 0 & 0 \\
0 & -1 & 1 \\
0 & 0 & -1
\end{bmatrix}
(m\neq 2), 
\quad
\begin{bmatrix}
1 & 0 & 0 \\
0 & -1 & 0 \\
0 & 0 & -1
\end{bmatrix}
(m= 2). 
\end{equation*}
Thus, the complex conjugacy class distinguishes the Markov quiver $Q_2$ from all other~$Q_m$'s.
(Indeed, $Q_2$ is only mutation-equivalent to itself.)


The $\GL(n,\mathbb{Q})$ conjugacy classes do not provide any additional refinement: 
the Frobenius normal form of~$C_m$ (also known as the rational canonical form) is given~by
\begin{equation*}
\begin{bmatrix}
0 & 0 & 1 \\
1 & 0 & 1 \\
0 & 1 & -1
\end{bmatrix}
(m\neq 2), 
\quad
\begin{bmatrix}
0 & 1 & 0 \\
1 & 0 & 0 \\
0 & 0 & -1
\end{bmatrix}
(m= 2). 
\end{equation*}

On the other hand, the $\GL(n,\mathbb{Z})$ conjugacy classes of the matrices~$C_m$ are all distinct. 
To see this, substitute $C_m$ into the polynomial $t^2-1=(t-1)(t+1)$: 
\begin{equation*}
C_m^2-I
= \begin{bmatrix}
m^2 - 4 & -m^3 + 4m & m^2 - 4 \\
0 & 0 & 0 \\
-m^2 + 4 & m^3 - 4m & -m^2 + 4
\end{bmatrix}
=(m^2-4)
\begin{bmatrix}
1 & -m & 1 \\
0 & 0 & 0 \\
-1 & m & -1
\end{bmatrix}
.
\end{equation*}
This implies that $C_m^2-I\equiv 0 \bmod (m^2-4)$ but $C_{m'}^2-I\not\equiv 0 \bmod (m^2-4)$ for $0<m'<m$. 
Therefore $C_m$ and $C_{m'}$ are not conjugate in $\GL(n,\mathbb{Z})$. 


It follows that no two distinct quivers $Q_m$ are related by proper mutations. 
This conclusion can also be derived from Observation~\ref{obs:gcd-seven}. 

\end{example}

\begin{remark}
The modular arithmetic argument used in Example~\ref{eg:2mm}
is not guaranteed to always work to establish non-conjugacy over the integers. 
As shown~by P.~F.~Stebe~\cite{Stebe},
for $n\ge3$, there exist matrices $M,M'\in\GL(n,\ZZ)$ such that 
(a) $M$ and~$M'$ are not conjugate in $\GL(n,\ZZ)$ and 
(b) this fact cannot be detected by passing to $\bmod N$ arithmetic for some~$N$  
(or~by applying some other homomorphism from $\GL(n,\ZZ)$ to a finite group). 
\end{remark}


\section{Alexander lattices and Alexander polynomials  
}
\label{sec:alexander-polynomials}

By Corollary~\ref{cor:proper-mutations-preserve-cosquare},
the $\GL(n,\ZZ)$ conjugacy class of the cosquare of a unipotent companion
is invariant under proper mutations and wiggles. 
Given two such cosquares, determining whether they are conjugate to each other in $\GL(n,\ZZ)$ 
is a nontrivial (but solvable) problem, see Remark~\ref{rem:algorithms-conjugacy}. 
It turns out that in many cases, this problem can be solved using fairly elementary tools introduced below. 

\begin{definition}
\label{def:Ut}
Let $Q$ be a quiver on a linearly ordered set of $n$ vertices.
Let $U_Q$ be the corresponding unipotent companion, as in Definition~\ref{def:unipotent companions}. \linebreak[3]
The \emph{parametrized companion} of $Q$ is the $n \times n$ matrix $P_Q(t)$ 
with entries in $\ZZ[t]$, defined by 
\begin{equation*}
P_Q(t) = tU_Q - U_Q^T.
\end{equation*}

For $1\le k\le n$, we define the \emph{Alexander lattice} $\dd_k(Q,t)\subset\ZZ[t]$ as
the $\ZZ$-span of all $k \times k$ minors (i.e., determinants of $k \times k$ submatrices) of 
the parametrized companion~$P_Q(t)$.  

The \emph{Alexander polynomial}  $\Delta_Q(t)\in\ZZ[t]$ is the determinant of~$P_Q(t)$,
or equivalently the monic characteristic polynomial of the cosquare $U_Q^{-T}U_Q$:
\begin{equation}
\label{eq:delta-Q-char-poly}
\Delta_Q(t) = \det(tU_Q-U_Q^{T}) =\det(tU_Q^{T}-U_Q)=\det(tI-U_Q^{-T}U_Q).
\end{equation}

The \emph{Markov invariant} $M_Q$ 
is defined by
\begin{equation}
\label{eq:Markov-n}
M_Q=n+(\text{coefficient of $t^{n-1}$ in $\Delta_Q(t)$})=n-\operatorname{Trace}(U^{-T}U). 
\end{equation}
\end{definition}

\begin{remark}
The unipotent companion $U_Q$ and the exchange matrix $B_Q$ can both be recovered from the parametrized companion~$P_Q(t)$.  
Indeed, 
$P_Q(1) = - B_Q$ and $P_Q(0) = - U_Q^T$.
\end{remark}

\begin{remark}
\label{rem:d_n}
The $n$th Alexander lattice $\dd_n(Q,t)$ is the $\ZZ$-span of the Alexander polynomial:
\begin{equation}
\label{eq:d_n}
\dd_n(Q,t) = \Delta_Q(t)\, \ZZ . 
\end{equation}
Hence $\dd_n(P_Q(t))$ contains the same information as $\Delta_Q(t)$. 
\end{remark}

\begin{remark}
\label{rem:d_1}
The Alexander lattice $\dd_1(Q,t)$  is the $\ZZ$-span of the entries of~$P_Q(t)$. 
Denoting by $u_{ij}$ ($i<j$) the upper-triangular entries of~$U_Q$, 
we see that 
$\dd_1(Q,t)$ is the $\ZZ$-span of the polynomials $t-1$, $u_{ij}$, and $tu_{ij}$ 
(or just $t-1$ and the scalars~$u_{ij}$).
Hence $\dd_1(Q,t)$ contains the same information as the gcd of all entries of~$B_Q$. 
\end{remark}

In Sections~\ref{sec:alex-examples}--\ref{sec:tree-quivers}, we compute Alexander polynomials and Alexander lattices for 
various classes of quivers, including quivers on 2, 3, or~4 vertices as well as tree quivers. 

\hide{
\begin{proposition}
If the unipotent companions of two quivers are integrally congruent, 
then their parametrized companions are integrally congruent as well. 
\end{proposition}

\begin{proof}
Indeed, if $U' = G U G^T$, then 
\begin{equation*}
t U' - (U')^T = t G U G^T - G U^T G^T = G (t U - U^T) G^T. \qedhere
\end{equation*}
\end{proof}
} 

\begin{corollary}
\label{cor:alex-poly-inv}
The Alexander polynomial $\Delta_Q(t)$ is invariant under cyclic shifts, wiggles, and proper mutations. 
\end{corollary}

\begin{proof}
Recall from~\eqref{eq:delta-Q-char-poly} that $\Delta_Q(t)$ 
is the characteristic polynomial of the cosquare of a unipotent companion.
It follows that $\Delta_Q(t)$ is uniquely determined by the $\GL(n,\CC)$ conjugacy class of this cosquare, 
hence by its $\GL(n,\ZZ)$ conjugacy class. 
The claim follows by Corollaries \ref{cor:similarity-U} and~\ref{cor:proper-mutations-preserve-cosquare}. 
\end{proof}

Our next goal is to extend Corollary~\ref{cor:alex-poly-inv} to Alexander lattices.

\begin{proposition}
\label{prop:Ut}
Let $Q$ and $Q'$ be two quivers on linearly ordered sets of $n$ vertices. 
If the cosquares of their unipotent companions are conjugate in $\GL(n,\ZZ)$, 
then the corresponding Alexander lattices $\dd_k$ coincide: for all~$k$, we have $\dd_k(Q,t)=\dd_k(Q',t)$.
\end{proposition}

\begin{proof}
We begin by showing that the parametrized companions $P_Q(t)$ and $P_{Q'}(t)$ lie in the same double $\GL(n,\ZZ)$-coset: 
\begin{equation}
\label{eq:PQ-double-coset}
P_{Q'}(t) \in \GL(n, \ZZ) \cdot P_Q(t) \cdot \GL(n, \ZZ).
\end{equation}
Assume that $ U_{Q'} U_{Q'}^{-T}  = G U_Q U_Q^{-T} G^{-1}.$ 
Then 
\begin{align*}
P_{Q'}(t) &= t U_{Q'} - U_{Q'}^{T} \\ 
&= ( t U_{Q'} U_{Q'}^{-T} - I) U_{Q'}^T\\
&=  ( t G U_Q U_Q^{-T} G^{-1} - I) U_{Q'}^T\\
&=  ( t G U_Q - G U_Q^T) U_Q^{-T} G^{-1} U_{Q'}^T\\
&= G P_Q(t) U_Q^{-T} G^{-1} U_{Q'}^T, 
\end{align*}
proving~\eqref{eq:PQ-double-coset}. 
\hide{
Then one can verify that $P_{Q'}(t) = U_{Q'} G^{-T} U_{Q}^{-1}P_Q(t) G^T$. \cS{computation hidden here.}
\begin{align*}
P_{Q'}(t) &= t U_{Q'} - U_{Q'}^{T} \\ 
&= U_{Q'} ( tI - U_{Q'}^{-1} U_{Q'}^{T}) \\
&=  U_{Q'} ( tI - (U_{Q'} U_{Q'}^{-T} )^{T})\\
&=U_{Q'} ( tI - (G U_{Q} U_{Q}^{-T} G^{-1})^{T}) \\
&= U_{Q'} G^{-T}( tI - (U_{Q} U_{Q}^{-T})^{T}) G^T\\
&= U_{Q'} G^{-T}( tI - U_{Q}^{-1} U_{Q}^{T}) G^T \\
&= U_{Q'} G^{-T} U_{Q}^{-1}P_Q(t)G^T . \qedhere
\end{align*}
}
To complete the proof, it suffices to show that if $A$ and $A'$ are $n\times n$ matrices with entries in $\ZZ[t]$ 
that lie in the same double $\GL(n,\ZZ)$-coset, 
then the lattices spanned by their $k\times k$ minors are equal. 
To prove this claim, it  is enough to consider the cases $A'=GA$ and $A'=AG$, for $G\in\GL(n,\ZZ)$.
In each case, the Binet-Cauchy formula implies that
every $k\times k$ minor of $A'$ is an integer linear combination of $k\times k$ minors of~$A$. 
\end{proof}

\begin{corollary}
\label{cor:alex-lattice-coq}
Let $(Q,\sigma)$ be a COQ on $n$ vertices.
Let $\tau$ be a linear ordering of the vertices of~$Q$ that is compatible with~$\sigma$.
Then the associated Alexander lattices $\dd_k(Q,t)$, for $1\le k\le n$, 
do not depend on the choice of~$\tau$. 
\end{corollary}

\begin{proof}
By Corollary~\ref{cor:similarity-U}, the $\GL(n,\ZZ)$ conjugacy class 
of the cosquare~$U_Q^{-T}U_Q$ is invariant under cyclic shifts.
The claim follows by Proposition~\ref{prop:Ut}.
\end{proof}

By Corollary~\ref{cor:alex-lattice-coq}, we can 
associate well-defined Alexander lattices $\dd_k(Q,\sigma,t)$ to any COQ~$(Q,\sigma)$. 
Furthermore, these lattices are invariant under proper mutations:

\begin{corollary}
\label{cor:alex-lattices-inv}
Alexander lattices of COQs are invariant under wiggles and proper mutations. 
\end{corollary}

\begin{proof}
The argument essentially repeats the proof of Corollary~\ref{cor:alex-lattice-coq},
except that this time we combine Proposition~\ref{prop:Ut} 
with Corollary~\ref{cor:proper-mutations-preserve-cosquare}. 
\end{proof}

\begin{remark}
Already in the case of 3-vertex quivers,
there are many examples of quivers that are not mutation equivalent but 
have the same Alexander polynomial (equivalently, the same Markov invariant). 
See, for instance, Example~\ref{eg:2mm}. 

It is harder---but possible---to find examples of 3-vertex COQs whose Alexander lattices all agree 
while the corresponding cosquares are not conjugate in $\GL(n,Z)$
(and the COQs are not mutation equivalent). 
See, in particular, Example~\ref{eg:2m-delta}. 
\end{remark}

We conclude this section with a couple of observations on Alexander polynomials. 

As in the case of links/knots, 
the Alexander polynomials of quivers are \emph{palindromic}, up to a change of signs: 

\begin{proposition} 
\label{pr:markov-palyndromic}
For any COQ~$Q$, the Alexander polynomial $\Delta(t)=\Delta_Q(t)$ satisfies
\begin{equation*}
\Delta(t) = (-t)^n \Delta(t^{-1}). 
\end{equation*}
In particular, $M_Q=n+(-1)^n (\text{\rm coefficient of $t$ in $\Delta_Q(t)$})=n+(-1)^n \frac{d}{dt}\Delta_Q(0)$. 
\end{proposition}

\begin{proof}
Let $C = U^{-T} U$.
We first note that $C^{-T}$ is conjugate to~$C$:
\begin{equation*}
C^{-T}=U U^{-T} = U^T U^{-T} U U^{-T} = U^T C U^{-T}. 
\end{equation*}
It follows that $C^{-1}=(C^{-T})^T$ has the same characteristic polynomial as~$C$.  
Hence
\begin{align*}
\Delta(t) &= \det(tI-C) = \det(tI-C^{-1}) \\
&= \det(tC-I) = (-t)^n \det(t^{-1}I-C) = (-t)^n \Delta(t^{-1}). \qedhere
\end{align*}
\end{proof}

It is well known \cite[Theorem 2.8.4]{fwz1-3} that the determinant of the exchange matrix $B_Q$
is preserved by mutations. 
This mutation invariant can be recovered from the Alexander polynomial of $Q$, as follows: 

\begin{proposition} 
 $\det(B_Q)=(-1)^n\Delta_Q(1)$.  
\end{proposition}

\begin{proof}
 The formula $-B_Q = U_Q-U_Q^T$ implies that 
\begin{equation*}
 (-1)^n\det(B_Q) =\det(-B_Q) =\det(U_Q-U_Q^{T}) = \Delta_Q(1).\qedhere
\end{equation*}
 \end{proof}
 
\begin{remark}
\label{rem:casals}
R.~Casals~\cite{casals-binary} introduced a binary invariant of quiver mutations that 
can be derived from the specialization $\Delta_Q(-1)$ of the Alexander polynomial. 
(See also the follow-up work by A.~Seven and \.I.~\"Unal~\cite{Seven-congruence}.) 
Importantly, those invariants do not depend on the choice of cyclic ordering.  
\end{remark}


\newpage

\section{Quivers with few vertices}
\label{sec:alex-examples}

In this section, we compute Alexander polynomials and Alexander lattices of various quivers on 2, 3, or 4 vertices. 

\begin{example}[$n=2$]
Let $Q$ be a 2-vertex quiver with 
\begin{align*}
U_Q&=\begin{bmatrix} 
1 &  \! \!-x  \\
0 &  \! \!1 
\end{bmatrix}, 
\end{align*}
cf.\ Example~\ref{eg:U-n=2}.
The parametrized companion of $Q$ is given by 
\begin{equation*}
P_Q(t)=\begin{bmatrix} 
t-1 &  -tx  \\
x &  t-1 
\end{bmatrix}. 
\end{equation*}
The Alexander polynomial of $Q$ is
\begin{equation}
\label{eq:alex-n=2}
\Delta_Q(t)=\det(P_Q(t))=t^2+t(-2+x^2)+1. 
\end{equation}
The Markov invariant is given by 
\begin{equation}
\label{eq:markov-n=2}
M_Q=2+(-2+x^2)=x^2. 
\end{equation}
The Alexander lattices are given by 
\begin{align*}
\dd_1(Q,t) &= (t-1)\hspace{1pt}\ZZ \oplus x\hspace{1pt}\ZZ \quad \text{(cf.\ Remark~\ref{rem:d_1})};\\
\dd_2(Q,t) &=(t^2+t(-2+x^2)+1)\hspace{1pt}\ZZ \quad \text{(cf.\ Remark~\ref{rem:d_n}).}
\end{align*}

\end{example}

\begin{example}[$n=3$]
\label{eg:3vert-inv}
Let $Q$ be a quiver on three linearly ordered vertices, with 
\begin{align*}
U_Q=\begin{bmatrix} 
1 &  \! \!-x \! \! & -z \\
0 &  \! \!1 \! \! & -y \\
0 &  \! \!0 \! \! & 1
\end{bmatrix}, 
\end{align*}
cf.\ Example~\ref{eg:U-n=3}. 
Then 
\begin{align*}
P_Q(t)&=\begin{bmatrix} 
t-1 &  \! \!-tx \! \! & -tz \\
x &  \! \!t-1 \! \! & -ty \\
z &  \! \!y \! \! & t-1
\end{bmatrix}. 
\end{align*}
Computing the Alexander polynomial of $Q$ yields 
\begin{align}
\nonumber
\Delta_Q(t)&=\det(P_Q(t)) \\
\nonumber
&=t^3 + (-3 \!+\!x^2 \!+\! y^2 \!+\! z^2 \!+\!xyz)t^2 + (3\! -\!x^2 \!-\!y^2\! -\!z^2 \!-\!xyz)t - 1 \\
\label{eq:alex-n=3}
&=(t-1)^3 + M_Q\cdot t\,(t-1), 
\end{align}
where $M_Q$ is the Markov invariant, given by 
\begin{equation}
\label{eq:markov-n=3}
M_Q 
=x^2 + y^2 + z^2 +xyz .
\end{equation}
(Cf.\ the ``Markov constant'' of a 3-vertex quiver appearing in  \cite[Section~3]{BBH}.) 


In view of Remark~\ref{rem:d_1}, we have $\dd_1(Q,t)=(t-1)\hspace{1pt}\ZZ\oplus \gcd(x,y,z)\hspace{1pt}\ZZ$. 

The Alexander lattice $\dd_2(Q,t)$ is the $\ZZ$-span of the $2\times 2$ minors of~$P_Q(t)$: 
\begin{equation*}
\begin{matrix}
t^2+(y^2-2)t+1 & (x+yz)t-x & -zt + z+xy \\[4pt]
-xt^2+(x+yz)t  \quad & t^2 + (z^2-2)t +1 \quad & (y+xz)t-y \\[4pt]
(z+xy)t^2-zt   & -yt^2 +(y+xz)t & t^2 +(x^2-2)t +1
\end{matrix}
\end{equation*}
We can also think of $\dd_2(Q,t)$ as the column $\ZZ$-span of the matrix 
\begin{equation*}
\begin{bmatrix}
1		& -x		& z+xy	& 0		& 1		& -y		& 0		& 0		& 1 \\
y^2-2	& x+yz	& -z		& x+yz	& z^2-2	& y+xz	& -z		& y+xz	& x^2-2 \\
1		& 0		& 0		& -x		& 1		& 0 		& z+xy	& -y		& 1
\end{bmatrix} \!,
\end{equation*}
where every column $\left[ \begin{smallmatrix} p\\ q\\ r\end{smallmatrix}\right]$ 
records the coefficients $p,q,r$ of a polynomial $p+q\,t+r\,t^2$. 
Applying appropriate column transformations, we conclude that the Alexander lattice $\dd_2(Q,t)$
is the $\ZZ$-span of the columns of the matrix
\begin{equation}
\label{eq:d2-3vert}
\begin{bmatrix}
1		& 1		& 1		& 0		& 0		& 0		& 0		& 0		& 0 \\
x^2\!-\!2	& y^2\!-\!2	& z^2\!-\!2	& x\!+\!yz	& y\!+\!xz	& z\!+\!xy	& x^3\!+\!2yz	& y^3\!+\!2xz	& z^3\!+\!2xy \\
1		& 1		& 1		& -x		& -y		& -z 		& 0		& 0		& 0
\end{bmatrix} \!\!.
\end{equation}

Finally, $\dd_3(Q,t)=\Delta_Q(t)\hspace{1pt}\ZZ $, see \eqref{eq:d_n}. 
\end{example}

\begin{example}[$n=4$]
Let $Q$ be a quiver on four linearly ordered vertices, with %
\begin{equation*}
U_Q\!=\begin{bmatrix} 
1 & -x & -z & -w \\
0 & 1 & -y & -v\\
0 & 0 & 1 & -u \\
0 & 0 & 0 & 1
\end{bmatrix} ,  
\end{equation*}
cf.\ Example~\ref{ex:C(Q)-4vert}.
Computing the Alexander polynomial $\Delta_Q(t)=\det(tU_Q-U_Q^T)$ yields 
\begin{align*}
\Delta_Q(t) 
&=t^4 + (-4 + x^2+y^2+u^2+z^2+v^2+w^2+xyz+xvw+yuv+uzw +xyuw)t^3 \\
&+ (6  +x^2u^2 +z^2v^2 + y^2w^2 - 2yzvw - 2xuzv \\
& \quad  -2x^2 -2 y^2 -2 u^2 -2 z^2 -2v^2 -2w^2 -2 xyz-2xvw-2yuv-2uzw )t^2 \\
&+  (-4 + x^2+y^2+u^2+z^2+v^2+w^2+xyz+xvw+yuv+uzw +xyuw)t + 1\\
&=t^4 + (-4 +M_Q)t^3 + (6 +\det(B_Q)-2M_Q)t^2 
+  (-4 + M_Q)t + 1, 
\end{align*}
where $M_Q$ is the Markov invariant, given by  
\begin{equation}
\label{eq:markov-4vert}
M_Q=x^2+y^2+u^2+z^2+v^2+w^2+xyz+xvw+yuv+uzw +xyuw. 
\end{equation}
Thus
\begin{equation}
\label{eq:alex-poly-4vert}
\Delta_Q(t)=(t-1)^4+M_Q\cdot t\,(t-1)^2 +\det(B_Q)\cdot t^2. 
\end{equation}

We omit the tedious calculation of Alexander lattices $\dd_k(Q,t)$ of 4-vertex quivers.
Suffice to note that $\dd_2(Q,t)$ (resp.,~$\dd_3(Q,t)$) is spanned by 36 (resp.,~16) minors of the $4\times4$ matrix~$P_Q(t)$. 
\end{example}

\begin{remark}
Formulas \eqref{eq:alex-n=2}, \eqref{eq:markov-n=2}, and~\eqref{eq:alex-n=3} show that
for quivers on $n\le 3$ vertices, the Markov invariant contains the same information
as the Alexander polynomial. 
For~4-vertex quivers, the Alexander polynomial \eqref{eq:alex-poly-4vert} encodes two quantities:
the Markov invariant~$M_Q$ given by~\eqref{eq:markov-4vert} and the determinant of the exchange matrix~$B_Q$. 
\end{remark}

\pagebreak[3]

\begin{example}
\label{eg:2mm-continued}
Let us revisit Example~\ref{eg:2mm}.
The 3-vertex quiver~$Q_m$ (see~\eqref{eq:Qm-2mm}) is the quiver in Example~\ref{eg:3vert-inv} specialized at $x=y=m$ and $z=-2$. 
Substituting these values into~\eqref{eq:d2-3vert}, we see that $\dd_2(Q_m,t)$
is the $\ZZ$-span of the columns of the matrix
\begin{equation*}
\begin{bmatrix}
1		& 1		& 0		& 0		& 0		& 0	 \\
m^2\!-\!2	& 2		& -m		& m^2-2	& m^3-4m	& -8+2m^2 \\
1		& 1		& -m		& 2		& 0		& 0
\end{bmatrix} \!,
\end{equation*}
or equivalently the $\ZZ$-span of the columns of 
\begin{equation*}
\begin{bmatrix}
1		& 0		& 0		& 0		 \\
2		& 2		& m		& m^2-4 \\
1		& 2		& m		& 0		
\end{bmatrix} \!. 
\end{equation*}
It follows that the Alexander lattices $\dd_2(Q_m,t)$ are distinct for different values of~$m$,
so proper mutation classes of~$Q_m$ are distinguished from each other by these lattices. 
\end{example}

\hide{
To see this, substitute $t=-1$ into the parametrized companion: 
\begin{equation*}
P_Q(-1)
= - U_Q - U_Q^T
=\begin{bmatrix} -2 & m & -2 \\ m & -2 & m \\ -2 &  m & -2 \end{bmatrix}; 
\end{equation*}
then observe that $d_2(P_Q(-1))=(m^2-4)\ZZ$. 
\cS{todo: new}
This can be seen by either taking the SNF of $P_{C_m}(-1)$, or by computing $d_2$ of $P_{C_m}(t)$.
Specifically, if $m$ is odd (resp. even) then the SNF of $P_{C_m}(-1)$ has $\alpha_1 = 1$ (resp. $2$), $\alpha_2 = m^2-4$ (resp. $\frac{1}{2}(m^2-4)$), and $\alpha_3=0$ (regardless of $m$):
\begin{equation*}
\begin{bmatrix}
1 & 0 & 0 \\
0 & m^2-4 & 0 \\
0 & 0 & 0
\end{bmatrix}
(m \text{ odd}), 
\quad
\begin{bmatrix}
2 & 0 & 0 \\
0 & \frac{1}{2}m^2-2 & 0 \\
0 & 0 & 0
\end{bmatrix}
(m \text{ even}). 
\end{equation*}
An (admittably more tedious) computation reveals that $d_2$ is a lattice spanned by $1 - t^2, m^2-4$, and either $2(1+t)$ or $1+t$, depending on if $m$ is even or not.
} 

\begin{example}
\label{eg:2m-delta}
For $m>0$ and $\delta\ge 2m$, let $Q_{m,\delta}$ be the linearly ordered quiver with
\begin{align*}
U_{Q_{m,\delta}}=\begin{bmatrix} 
1 &  \! \!-m \! \! & m-\delta \\
0 &  \! \!1 \! \! & -2 \\
0 &  \! \!0 \! \! & 1
\end{bmatrix}. 
\end{align*}
Pictorially, $Q_{m,\delta}$ is represented by the diagram
\begin{equation}
\label{eq:Qlm}
 \begin{tikzcd}[arrows={-stealth}, sep=2em]
  a  \arrow[r,  "m"]  \arrow[rr, bend right=45, "\delta-m"]  & b  \arrow[r, "2"] & c 
\end{tikzcd}. 
\end{equation}
(We note that in the case $\delta=0$, which does not satisfy the requirement $\delta\ge 2m$,
we recover the COQ $Q_m$ from Examples~\ref{eg:2mm} and~\ref{eg:2mm-continued}.) 

In the notation of Example~\ref{eg:3vert-inv}, we have $x=m$, $y=2$, $z=\delta-m$. 
Computing the Markov invariant 
\begin{equation*}
M_{Q_{\ell,m}}=m^2+(\delta-m)^2+4+2m(\delta-m) = \delta^2+4 , 
\end{equation*}
we notice that it does not depend on~$m$.
So for fixed~$\delta$, we get the same Alexander polynomial~\eqref{eq:alex-n=3} for all values of~$m$. 

The Alexander lattice $\dd_2(Q_{m,\delta},t)$ (cf.\ \eqref{eq:d2-3vert}) 
is the column $\ZZ$-span of the matrix
\begin{equation*}
\left[
\begin{smallmatrix}
1		& 1		& 1				& 0			& 0			& 0	 	& 0 & 0 & 0\\[3pt]
m^2-2{\ }	& {\ }2{\ }		& {\ }(\delta-m)^2-2{\ }	& {\ }2\delta-m{\ }	& {\ }\delta m -m^2+2{\ }& {\ }\delta+m{\ }& {\ }m^3+4\delta-4m{\ }& {\ }8+2\delta m-2m^2{\ }& {\ }(\delta-m)^3+4m\\[3pt]
1		& 1		& 1				& -m			& -2			& m-\delta& 0 & 0 & 0
\end{smallmatrix} 
\right]\!,
\end{equation*}
which can be simplified to
\begin{equation*}
\dd_2(Q_{m,\delta},t)=
\text{column $\ZZ$-span}
\!\left(\left[
\begin{smallmatrix}
1		& 0				& 0			& 0			& 0	 	& 0 & 0 & 0\\[3pt]
2{\ }		& {\ }2\delta-m{\ }	& {\ }\delta+m{\ }	& {\ }\delta m -m^2+2{\ }& {\ }m^2-4{\ }& {\ }\delta^2{\ }& {\ }4\delta{\ }& {\ }2\delta m\\[3pt]
1		& -m				& m-\delta			& -2			& 0 & 0 & 0 & 0
\end{smallmatrix} 
\right]\right)\!. 
\end{equation*}
The quivers $Q_{m,\delta}$ are acyclic and therefore pairwise mutation inequivalent when $\delta>0$ and $m$ varies, 
cf.~\cite{CK}. Some of their mutation classes are distinguished from each other by their respective lattices
$\dd_2(Q_{m,\delta},t)$, but some are not. 
To illustrate, if $\delta=10$, then the only coincidences are
$\dd_2(Q_{0,10},t)=\dd_2(Q_{4,10},t)$ and $\dd_2(Q_{1,10},t)=\dd_2(Q_{5,10},t)$. 
These pairs of quivers are however distinguished by their multisets of vertex gcd's, 
see Observation~\ref{obs:gcd-seven}. 
\end{example}

\section{Tree quivers}
\label{sec:tree-quivers}

In this section, we compute Alexander polynomials 
of several families of \hbox{\emph{tree quivers}}
(cf.\ Definition~\ref{def:tree quiv}). 
We also discuss some of the corresponding Alexander lattices. 

Tree quivers provide a convenient data set for testing the relative power of mutation invariants. 
All orientations of the same tree are mutation equivalent to each other. 
It~is~also~well known, although nontrivial to prove, that orientations of non-isomorphic trees
are mutation-inequivalent, see \cite{CK} or \cite[Corollary~2.6.13]{fwz1-3}. 

By Proposition~\ref{prop:tree wiggles}, all cyclic orderings of a tree quiver are wiggle equivalent to each other.
Hence the Alexander polynomial does not depend 
on the choice of a cyclic ordering, cf.\ Corollary~\ref{cor:similarity-U}. 
It will also transpire that iterated mutations
of the quivers examined below are always proper (cf.\ Theorem~\ref{thm:finite type} and Remark~\ref{rem:atp}), 
so the proper mutation class coincides with the ordinary one. 

We first examine the Dynkin quivers of finite types $ADE$. 

\begin{example}[Type~$A_n$]
\label{eg:An-Alexander}
Consider the COQ 
\begin{equation*}
Q=(v_1\rightarrow v_2\rightarrow\cdots\rightarrow v_n)
\end{equation*}
of type~$A_n$, with the cyclic ordering $\sigma=(v_1,\dots,v_n)$; 
cf.\ Examples~\ref{eg:A3}--~\ref{eg:A4}.
Then
\begin{equation*}
U_Q= \begin{bmatrix} 
1 & -1 & 0 & \cdots & 0 & 0 \\ 
0 &  1 & -1 & \cdots & 0 & 0 \\ 
0 &  0 & 1  & \cdots & 0 & 0 \\
\vdots & \vdots & \vdots & \ddots & \vdots & \vdots \\
0 &  0 & 0  & \cdots & 1 & -1 \\
0 &  0 & 0  & \cdots & 0 & 1 \\
\end{bmatrix}\!, 
\quad
U_Q^{-T}U_Q= \begin{bmatrix} 
1 & -1 & 0 & \cdots & 0 & 0 \\ 
1 &  0 & -1 & \cdots & 0 & 0 \\ 
1 &  0 & 0  & \cdots & 0 & 0 \\
\vdots & \vdots & \vdots & \ddots & \vdots & \vdots \\
1 &  0 & 0  & \cdots & 0 & -1 \\
1 &  0 & 0  & \cdots & 0 & 0 \\
\end{bmatrix}\!, 
\end{equation*}
\begin{equation*}
\Delta_Q(t)=t^n-t^{n-1}+t^{n-2}+\cdots+(-1)^n = \frac{t^{n+1}+(-1)^n}{t+1}, 
\end{equation*}
\begin{equation*}
\label{eq:}
M_Q=n-1, \quad 
\det(B_Q)=1 \text{ ($n$ even),}  \quad  \det(B_Q)=0 \text{ ($n$ odd).} 
\end{equation*}
\end{example}

\hide{
\begin{example}
\label{eg:A3-continued}
Consider the COQ $Q=(a \rightarrow b \rightarrow c)$ of type~$A_3$ with the cyclic ordering $\sigma=(a,b,c)$,
cf.\ Example~\ref{eg:A3}.
In this case, we get 
the Alexander polynomial $\Delta(t) = t^3 - t^2 + t - 1$
and the Markov invariant $M_Q=2$.
\end{example}

\begin{example}
\label{eg:A4-continued}
Consider the quiver $Q = (a \rightarrow b \rightarrow c \rightarrow d)$ of type~$A_4$, 
with the cyclic ordering $\sigma=(a,b,c,d)$, cf.\ Example~\ref{eg:A4}. 
In this case, we get 
the Alexander polynomial $\Delta(t) = t^4 - t^3 + t^2 - t + 1$, 
the Markov invariant $M_Q=3$, and the determinant $\det(B_q)=1$. 
\end{example}

\begin{example}
\label{eg:D4-continued}
Consider the quiver $Q$ of type $D_4$ with arrows $a \rightarrow b \rightarrow c$ and $b \rightarrow d$  
and the cyclic ordering $\sigma=(a,b,c,d)$, cf.\ Example~\ref{eg:D4}. 
In this case, we get 
the Alexander polynomial $\Delta(t) = t^4 - t^3 - t + 1$, the Markov invariant $M_Q=3$, and the determinant $\det(B_Q)=0$. 
\end{example}
}

\begin{example}[Type~$D_n$]
\label{eg:Dn-Alexander}
For $n\ge 4$, consider the COQ 
\begin{equation*}
Q\ =\begin{tikzcd}[arrows={-stealth}, sep=1em]
  v_1  \arrow[r]   & v_3  \arrow[r] & v_4 \arrow[r] & \cdots \arrow[r] & v_n \\
 & v_2  \arrow[u] 
   & 
\end{tikzcd}
\end{equation*}
of type~$D_n$, with the cyclic ordering $\sigma=(v_1,\dots,v_n)$.
Then
\begin{equation*}
U_Q= \begin{bmatrix} 
1 &  0 & -1 & 0 & \cdots & 0 & 0 \\ 
0 &  1 & -1 & 0 & \cdots & 0 & 0 \\ 
0 &  0 & 1  & -1 & \cdots & 0 & 0 \\
0 &  0 & 0  & 1 & \cdots & 0 & 0 \\
\vdots & \vdots & \vdots & \vdots & \ddots & \vdots & \vdots \\
0 &  0 & 0  & 0 & \cdots & 1 & -1 \\
0 &  0 & 0  & 0 & \cdots & 0 & 1 \\
\end{bmatrix}\!, 
\quad
U_Q^{-T}U_Q= \begin{bmatrix} 
1 &  0 & -1 & 0 & \cdots & 0 & 0 \\ 
0 &  1 & -1 & 0 & \cdots & 0 & 0 \\ 
1 &  1 & -1  & -1 & \cdots & 0 & 0 \\
1 &  1 & -1  & 0 & \cdots & 0 & 0 \\
\vdots & \vdots & \vdots & \vdots & \ddots & \vdots & \vdots \\
1 &  1 & -1  & 0 & \cdots & 0 & -1 \\
1 &  1 & -1  & 0 & \cdots & 0 & 0 \\
\end{bmatrix}\!, 
\end{equation*}
\begin{equation*}
\Delta_Q(t)=t^n-t^{n-1}+(-1)^{n-1}t+(-1)^n , 
\end{equation*}
\begin{equation*}
M_Q=n-1, \quad
\det(B_Q)= 0. 
\end{equation*}
\end{example}

\begin{example}[Type~$E_n$]
\label{eg:En-Alexander}
For $n\ge6$, consider the COQ 
\begin{equation*}
Q\ =\begin{tikzcd}[arrows={-stealth}, sep=1em]
  v_1  \arrow[r]   & v_2  \arrow[r] & v_4 \arrow[r] & \cdots \arrow[r] & v_n \\
 & & v_3  \arrow[u] 
   & 
\end{tikzcd}
\end{equation*}
of type~$E_n$, with the cyclic ordering $\sigma=(v_1,\dots,v_n)$.
Then
\begin{align*}
U_Q^{-T}U_Q&= \begin{bmatrix} 
1 &  -1 & 0 & 0 & 0 & \cdots & 0 & 0 \\ 
1 &  0 & 0 & -1 & 0 & \cdots & 0 & 0 \\ 
0 &  0 & 1  &-1 & 0 & \cdots & 0 & 0 \\
1 &  0 & 1  &-1 & -1 & \cdots & 0 & 0 \\
1 &  0 & 1  &-1 & 0 & \cdots & 0 & 0 \\
\vdots & \vdots & \vdots & \vdots & \vdots & \ddots & \vdots & \vdots \\
1 &  0 & 1 & -1  & 0 & \cdots & 0 & -1 \\
1 &  0 & 1 & -1  & 0 & \cdots & 0 & 0 \\
\end{bmatrix}\!, 
\end{align*}
\begin{align*}
\Delta_Q(t)&=t^n-t^{n-1}
+t^{n-3}-t^{n-4}+\cdots+(-1)^{n-1} t^4+(-1)^n t^3
+(-1)^{n-1}t+(-1)^n \\
&=(t-1)(t^{n-1}+(-1)^{n-1})+t^3 \,\frac{t^{n-5}+(-1)^{n-4}}{t+1}, 
\end{align*}
\begin{equation*}
M_Q=n-1, \quad  \det(B_Q)=1 \text{ ($n$ even),}  \quad  \det(B_Q)=0 \text{ ($n$ odd)}. 
\end{equation*}
\end{example}

\begin{remark}
Neither $\det(B_Q)$ nor $\operatorname{rank}(B_Q)$ 
distinguish between the three finite types~$A_n$, $D_n$,~and~$E_n$ when $n$ is odd;
or between $A_n$ and~$E_n$ when $n$ is even.
On the other hand, the Alexander polynomials of these quivers (or of any COQs in the corresponding
proper mutation classes) are distinct from each other. 
\end{remark}

\pagebreak[3]

\begin{example}[Tree quivers on six vertices]
\label{eg:6vert-Alexander}
There are six pairwise non-isomorphic trees on 6~vertices: 
the Dynkin diagrams of types $A_6$, $D_6$, and~$E_6$, plus three more. 
The corresponding Alexander polynomials are shown below: 
\begin{equation*}
\begin{array}{ll}
\hspace{.5in} A_6 & \begin{array}{l} \Delta_Q(t)=t^6 - t^5 + t^4 - t^3 + t^2 - t + 1 \end{array}
\\[10pt]
\hspace{.5in} D_6 & \begin{array}{l} \Delta_Q(t)=t^6 - t^5  - t + 1 \\
\hspace{31pt} = (t^4 + t^3 + t^2 + t + 1)(t-1)^2 \end{array}
\\[15pt]
\hspace{.5in} E_6 & \begin{array}{l} \Delta_Q(t)=t^6 - t^5 +t^3 - t + 1 \\
\hspace{31pt} =(t^2 - t + 1) (t^4 - t^2 + 1) \end{array}
\\[5pt]
\begin{tikzcd}[arrows={-stealth}, sep=1em]
  \scriptstyle\bullet  \arrow[r]   & \scriptstyle\bullet  \arrow[r] & \scriptstyle\bullet \arrow[r] & \scriptstyle\bullet \\
 & \scriptstyle\bullet  \arrow[u] &  \scriptstyle\bullet \arrow[u]
   & 
\end{tikzcd}
\quad & \begin{array}{l} \ \\[-5pt]
\Delta_Q(t)=t^6 - t^5 - t^4 + 2t^3 - t^2 - t + 1
\\ \hspace{31pt}=(t^2 - t + 1)(t - 1)^2(t + 1)^2
\end{array}
\\[20pt]
\begin{tikzcd}[arrows={-stealth}, sep=1em]
& \scriptstyle\bullet \arrow[d] \\
  \scriptstyle\bullet  \arrow[r]   & \scriptstyle\bullet  \arrow[r] & \scriptstyle\bullet \arrow[r] & \scriptstyle\bullet \\
 & \scriptstyle\bullet  \arrow[u] &  
   & 
\end{tikzcd}
& \begin{array}{l}\ \\[-9pt]
\Delta_Q(t)=t^6 - t^5 - 2t^4 + 4t^3 - 2t^2 - t + 1\\ \hspace{31pt} 
   =(t^4 + t^3 - t^2 + t + 1)(t - 1)^2 \end{array}
\\[32pt]
\begin{tikzcd}[arrows={-stealth}, sep=1em]
\scriptstyle\bullet \arrow[dr] & \scriptstyle\bullet \arrow[d] & \scriptstyle\bullet \\
  \scriptstyle\bullet  \arrow[r]   & \scriptstyle\bullet  \arrow[r] \arrow[ru] & \scriptstyle\bullet 
\end{tikzcd}
& \begin{array}{l}\ \\[-6pt]
\Delta_Q(t)=t^6 - t^5 - 5t^4 + 10t^3 - 5t^2 - t + 1 \\ \hspace{31pt} = (t^2 + 3t + 1)(t - 1)^4\end{array}
\end{array}
\vspace{10pt}
\end{equation*}
\end{example}

\begin{example}
\label{eg:8vert-trees}
In general, the Alexander polynomial does not necessarily distinguish between 
tree quivers whose underlying trees are non-isomorphic.
The sole example with $n\le8$ involves the 8-vertex tree quivers
\begin{equation}
\label{eq:8-vertex-trees}
\begin{tikzcd}[arrows={-stealth}, sep=1em]
& \scriptstyle\bullet \arrow[d] & \scriptstyle\bullet \arrow[d]\\
  \scriptstyle\bullet  \arrow[r]   & \scriptstyle\bullet  \arrow[r] & \scriptstyle\bullet \arrow[r] & \scriptstyle\bullet \\
 & \scriptstyle\bullet  \arrow[u] &  \scriptstyle\bullet \arrow[u]
   & 
\end{tikzcd}
\qquad  \text{and}\qquad 
\begin{tikzcd}[arrows={-stealth}, sep=1em]
\scriptstyle\bullet \arrow[dr] & \scriptstyle\bullet \arrow[d] & \scriptstyle\bullet \\
  \scriptstyle\bullet  \arrow[r]   & \scriptstyle\bullet  \arrow[r] \arrow[ru] & \scriptstyle\bullet \arrow[r]   & \scriptstyle\bullet \arrow[r]   & \scriptstyle\bullet 
\end{tikzcd}
\end{equation}
which have the same Alexander polynomial
\begin{align*}
\Delta(t)&=
t^8 - t^7 - 5t^6 + 13t^5 - 16t^4 + 13t^3 - 5t^2 - t + 1 \\
&=(t^4 + 3t^3 + t^2 + 3t + 1)(t - 1)^4. 
\end{align*}
On the other hand, these COQs are distinguished from each other by their respective Alexander lattices
$\dd_7(Q,t)$ (hence by the $\GL(n,\ZZ)$ conjugacy classes
of the corresponding cosquare matrices $U_Q^{-T}U_Q$). 
This can be shown by verifying that for the left tree quiver in~\eqref{eq:8-vertex-trees}
every polynomial
\begin{equation*}
c_0+c_1 t + c_2 t^2 +\cdots+ c_7 t^7 \in \dd_7(Q,t)
\end{equation*}
satisfies the congruence
\begin{equation*}
c_3+c_4+c_5-c_6-c_7\equiv 0 \bmod 3.  
\end{equation*}
(Equivalently, substitute $t\leftarrow t+1$ into every polynomial in the spanning set of $\dd_7(Q,t)$
and verify that the coefficient of $t^3$ vanishes modulo~3.)
For the right quiver in~\eqref{eq:8-vertex-trees}, this congruence does not generally hold. 

\hide{
This can be shown by computing the lattice $d_7$.
Both lattices have rank $5$. 
For the tree on the left, when projected onto $1, t, \ldots, t^4$ the lattice $d_7$ has determinant $3$.
For the tree on the right, when projected onto the same subspace the corresponding lattice has determinant $1$ (indeed, we get all integer polynomials of degree $\leq 4$). 
} 
Alternatively, evaluate the polynomial
$t^4 + 3t^3 + t^2 + 3t + 1$ at each cosquare $U_Q^{-T}U_Q$
and verify that the first evaluation vanishes $\bmod 3$ 
(i.e., it yields a zero $8\times 8$ matrix) 
whereas the second one does~not.
\end{example}

\begin{example}
\label{eg:9-vertex-trees}
Among 47 pairwise non-isomorphic trees on 9 vertices, 
37~trees give rise to tree quivers with unique Alexander polynomials. 
The remaining 10~trees form five
``collision pairs,'' 
see Figures~\ref{fig:9-vertex-trees-2}--\ref{fig:9-vertex-trees-3}. 

For the two pairs shown in Figure~\ref{fig:9-vertex-trees-2}, 
the corresponding $\GL(n,\ZZ)$ conjugacy classes of cosquares of unipotent companions are distinct; 
moreover, the latter fact can be certified by inspecting the corresponding Alexander lattices~$\dd_8(Q,t)$, as follows.

Similarly to Example~\ref{eg:8vert-trees}, consider polynomials 
\begin{equation*}
c_0+c_1 t + c_2 t^2 +\cdots+ c_8 t^8 \in \dd_8(Q,t). 
\end{equation*}
It turns out that in each pair of 9-vertex tree quivers shown in Figure~\ref{fig:9-vertex-trees-2}, 
the polynomials in one of the two Alexander lattices~$\dd_8(Q,t)$---but not in the other---satisfy the congruences
\begin{align*}
c_0+c_3+c_6 \equiv c_1+c_4+c_7 \equiv c_2 + c_5+c_8 \bmod 2. 
\end{align*}
The claim follows. 

\pagebreak[3]


\begin{figure}[ht]
\begin{equation*}
\begin{array}{|l|l|} 
\hline
\multicolumn{2}{|c|}{}\\[-8pt]
\multicolumn{2}{|c|}{t^9 - t^8 - t^7 + 3t^6 - 4t^5 + 4t^4 - 3t^3 + t^2 + t - 1}\\
\multicolumn{2}{|c|}{= (t^6 + t^5 - t^4 - t^2 + t + 1)(t^2 - t + 1)(t - 1)}\\
\hline
\ \\[-10pt]
\begin{tikzcd}[arrows={-stealth}, sep=1em]
  \scriptstyle\bullet  \arrow[r]   & \scriptstyle\bullet  \arrow[r] & \scriptstyle\bullet \arrow[r] & \scriptstyle\bullet \arrow[r] & \scriptstyle\bullet \arrow[r] & \scriptstyle\bullet \\
 \scriptstyle\bullet  \arrow[r] & \scriptstyle\bullet  \arrow[u] &  \scriptstyle\bullet \arrow[l]
   & & &
\end{tikzcd}
\quad & 
\begin{tikzcd}[arrows={-stealth}, sep=1em]
  \scriptstyle\bullet  \arrow[r]   & \scriptstyle\bullet  \arrow[r] & \scriptstyle\bullet \arrow[r] & \scriptstyle\bullet \arrow[r] & \scriptstyle\bullet \arrow[r] & \scriptstyle\bullet \\
 & \scriptstyle\bullet  \arrow[u] &  & \scriptstyle\bullet \arrow[u] & \scriptstyle\bullet \arrow[l] & 
\end{tikzcd}
\\[15pt] 
\hline
\multicolumn{2}{c}{}\\[-10pt]
\hline
\multicolumn{2}{|c|}{}\\[-8pt]
\multicolumn{2}{|c|}{t^9 - t^8 - 3t^7 + 9t^6 - 14t^5 + 14t^4 - 9t^3 + 3t^2 + t - 1}\\
\multicolumn{2}{|c|}{= (t^2 + 3t + 1)(t^2 - t + 1)(t^2 + 1)(t - 1)^3}\\
\hline
\ \\[-10pt]
\begin{tikzcd}[arrows={-stealth}, sep=1em]
  \scriptstyle\bullet  \arrow[r]   & \scriptstyle\bullet  \arrow[r] & \scriptstyle\bullet \arrow[r] & \scriptstyle\bullet \arrow[r] & \scriptstyle\bullet \\
 \scriptstyle\bullet  \arrow[ur] & \scriptstyle\bullet  \arrow[u] &  \scriptstyle\bullet \arrow[u] &  \scriptstyle\bullet \arrow[l] & 
\end{tikzcd}
\quad & 
\begin{tikzcd}[arrows={-stealth}, sep=1em]
  \scriptstyle\bullet  \arrow[r]   & \scriptstyle\bullet  \arrow[r] & \scriptstyle\bullet \arrow[r] & \scriptstyle\bullet \arrow[r] & \scriptstyle\bullet \arrow[r] & \scriptstyle\bullet \\
 & & \scriptstyle\bullet  \arrow[ur]  & \scriptstyle\bullet \arrow[u] & \scriptstyle\bullet \arrow[u] & 
\end{tikzcd}
\\[15pt] 
\hline
\end{array}
\end{equation*}
\caption{Pairs of 9-vertex tree quivers that have the same Alexander polynomial (shown) but different Alexander lattices.
}
\vspace{-10pt}
\label{fig:9-vertex-trees-2}
\end{figure}

For each of the remaining three pairs of 9-vertex tree quivers shown in Figure~\ref{fig:9-vertex-trees-3}, 
the $\GL(n,\ZZ)$ conjugacy classes for the two quivers coincide,
so neither these conjugacy classes, nor the Alexander lattices that they determine, 
certify the mutation inequivalence of the two quivers. 

We do not know whether these quivers are distinguished from each other 
by the integral congruence classes of their unipotent companions,
which might potentially carry more information than the aforementioned conjugacy classes.
Cf.\ Remark~\ref{rem:conjugacy-vs-mut-equiv}.
\end{example}

\begin{figure}[ht]
\begin{equation*}
\begin{array}{|l|l|} 
\hline
\multicolumn{2}{|c|}{}\\[-8pt]
\multicolumn{2}{|c|}{t^9 - t^8 - 2t^7 + 4t^6 - 4t^5 + 4t^4 - 4t^3 + 2t^2 + t - 1}\\
\multicolumn{2}{|c|}{=(t^6 + 2t^5 + t^4 + 2t^3 + t^2 + 2t + 1)(t - 1)^3} \\
\hline
\ \\[-10pt]
\begin{tikzcd}[arrows={-stealth}, sep=1em]
  \scriptstyle\bullet  \arrow[dr]   & \scriptstyle\bullet  \arrow[r] & \scriptstyle\bullet \arrow[r] & \scriptstyle\bullet \arrow[r] & \scriptstyle\bullet \arrow[r] & \scriptstyle\bullet \\
 \scriptstyle\bullet  \arrow[r] & \scriptstyle\bullet  \arrow[u] &  \scriptstyle\bullet \arrow[l]
   & & &
\end{tikzcd}
& 
\begin{tikzcd}[arrows={-stealth}, sep=1em]
  \scriptstyle\bullet  \arrow[r]   & \scriptstyle\bullet  \arrow[r] & \scriptstyle\bullet \arrow[r] & \scriptstyle\bullet \arrow[r] & \scriptstyle\bullet \arrow[r] & \scriptstyle\bullet \\
 & \scriptstyle\bullet  \arrow[u] &  & \scriptstyle\bullet \arrow[u] & \scriptstyle\bullet \arrow[u] & 
\end{tikzcd}
\\[15pt]
\hline
\multicolumn{2}{c}{}\\[-10pt]
\hline
\multicolumn{2}{|c|}{}\\[-8pt]
\multicolumn{2}{|c|}{t^9 - t^8 - 2t^7 + 6t^6 - 8t^5 + 8t^4 - 6t^3 + 2t^2 + t - 1}\\
\multicolumn{2}{|c|}{=(t^8 - 2t^6 + 4t^5 - 4t^4 + 4t^3 - 2t^2 + 1)(t - 1)} \\
\hline
\ \\[-10pt]
\begin{tikzcd}[arrows={-stealth}, sep=1em]
  \scriptstyle\bullet  \arrow[r]   & \scriptstyle\bullet  \arrow[r] & \scriptstyle\bullet \arrow[r] & \scriptstyle\bullet \arrow[r] & \scriptstyle\bullet \arrow[r] & \scriptstyle\bullet \\
 \scriptstyle\bullet  \arrow[ur] & \scriptstyle\bullet  \arrow[u] &  \scriptstyle\bullet \arrow[l]
   & & &
\end{tikzcd}
\quad & 
\begin{tikzcd}[arrows={-stealth}, sep=1em]
  \scriptstyle\bullet  \arrow[r]   & \scriptstyle\bullet  \arrow[r] & \scriptstyle\bullet \arrow[r] & \scriptstyle\bullet \arrow[r] & \scriptstyle\bullet \arrow[r] & \scriptstyle\bullet \\
 & & \scriptstyle\bullet  \arrow[u]  & \scriptstyle\bullet \arrow[u] & \scriptstyle\bullet \arrow[u] & 
\end{tikzcd}
\\[15pt] 
\hline
\multicolumn{2}{c}{}\\[-10pt]
\hline
\multicolumn{2}{|c|}{}\\[-8pt]
\multicolumn{2}{|c|}{t^9 - t^8 - 2t^7 + 6t^6 - 10t^5 + 10t^4 - 6t^3 + 2t^2 + t - 1}\\
\multicolumn{2}{|c|}{=(t^4 + 2t^3 + 2t + 1)(t^2 + 1)(t - 1)^3} \\
\hline
\ \\[-10pt]
\begin{tikzcd}[arrows={-stealth}, sep=1em]
  \scriptstyle\bullet  \arrow[r]   & \scriptstyle\bullet  \arrow[r] & \scriptstyle\bullet \arrow[r] & \scriptstyle\bullet \arrow[r] & \scriptstyle\bullet \\
 \scriptstyle\bullet  \arrow[ur] & \scriptstyle\bullet  \arrow[u] &  \scriptstyle\bullet \arrow[l] & \scriptstyle\bullet \arrow[l] &
\end{tikzcd}
\quad & 
\begin{tikzcd}[arrows={-stealth}, sep=1em]
  \scriptstyle\bullet  \arrow[r]   & \scriptstyle\bullet  \arrow[r] & \scriptstyle\bullet \arrow[r] & \scriptstyle\bullet \arrow[r]  & \scriptstyle\bullet \\
 & \scriptstyle\bullet  \arrow[u] &  \scriptstyle\bullet \arrow[u] & \scriptstyle\bullet \arrow[l] & \scriptstyle\bullet \arrow[u]&
\end{tikzcd}
\\[15pt] 
\hline
\end{array}
\end{equation*}
\caption{Pairs of 9-vertex tree quivers that have the same Alexander polynomial (shown) and the same Alexander lattices.
}
\vspace{-5pt}
\label{fig:9-vertex-trees-3}
\end{figure}

\begin{remark}
A.~Schwartz~\cite{SchwartzTree} gave a combinatorial formula for the Alexander polynomial of a tree quiver (computed from an associated link).
Our constructions~agree.
\end{remark}

\newpage 

\section{Proper COQs}
\label{sec:proper COQs}

\begin{definition}
\label{def:wiggle class proper}
A COQ $(Q,\sigma)$, or its wiggle equivalence class~$\wiggleQ$, 
is called \emph{proper} if every vertex~$j$ in~$Q$ is proper in $\wiggleQ$,
cf.\ Definition~\ref{def:wiggle-class-proper-vertex}.
To rephrase, a COQ is proper if every vertex in it can be made proper by a sequence of wiggles. 

For a quiver~$Q$, a \emph{proper cyclic ordering} is a cyclic ordering~$\sigma$ such that $(Q,\sigma)$ is a proper COQ. 
\end{definition}

\begin{remark}
\label{rem:subcoq}
The property of being proper is \emph{hereditary}, i.e., it passes from a COQ~$Q$ to any \emph{subCOQ}, 
i.e., a full subquiver of~$Q$ with the induced cyclic ordering. 
\end{remark}

\begin{remark}
The notion of a proper cyclic ordering is closely related to the notion of a  \emph{locally transitive tournament} 
investigated by several authors \cite{brouwer, cameron, moon}.  
\end{remark}

\begin{example}
An oriented $4$-cycle quiver has six cyclic orderings 
that form three wiggle equivalence classes, see Example~\ref{eg:6-orderings-D4}. 
Two of the three classes are proper; they are shown in Figure~\ref{fig:proper-4-cycle orderings}.
\end{example}

{ 
\newcommand\cyclicLabels[5]{%
	\filldraw[black] (#4,#5)++(135:1cm) circle (2pt) node[above left=-1pt] {$a$} coordinate (a);
	\filldraw[black] (#4,#5)++(45:1cm) circle (2pt) node[above right=-1pt] {$#1$} coordinate (#1);
	\filldraw[black] (#4,#5)++(-45:1cm) circle (2pt) node[below right=-1pt] {$#2$} coordinate (#2);
	\filldraw[black] (#4,#5)++(-135:1cm) circle (2pt) node[below left=-1pt] {$#3$} coordinate (#3);
	\draw[black, dashed, decoration={markings, mark=at position 0 with {\arrow{<}}}, postaction={decorate}] (#4,#5) circle (1cm);
	\draw[black, -{stealth}, shorten >=3pt, shorten <= 3pt] (a) -- (b);
	\draw[black, -{stealth}, shorten >=3pt, shorten <= 3pt] (b) -- (c);
	\draw[black, -{stealth}, shorten >=3pt, shorten <= 3pt] (c) -- (d);
	\draw[black, -{stealth}, shorten >=3pt, shorten <= 3pt] (d) -- (a);
}

\begin{figure}[ht]
\begin{tikzpicture}
\foreach \i\j\k\xp\yp in {b/c/d/0/0,b/d/c/3/0,c/d/b/6/0} 
{
	\cyclicLabels{\i}{\j}{\k}{\xp}{\yp}
} 
\end{tikzpicture}
\vspace{5pt}
\caption{Proper cyclic orderings of the 4-cycle quiver $Q=(a \rightarrow b \rightarrow c \rightarrow d \rightarrow a)$.
The second and third cyclic orderings are wiggle equivalent.
In the first ordering, all vertices are proper. 
In the second ordering, $a$ and $b$ are proper but $c$ and $d$ are not.
In the third ordering, $c$ and $d$ are proper while $a$ and $b$ are not.}
\label{fig:proper-4-cycle orderings}
\end{figure}
} 

Proper COQs are of interest to us because any mutation in a proper COQ preserves the invariants discussed in
Sections~\ref{sec:invariants-of-proper-mutations}--\ref{sec:alexander-polynomials}. 
Unfortunately, properness does not propagate under mutations:

\begin{example}
\label{eg:proper-doesn't-propagate}
Figure~\ref{fig:proper after mutation} shows a proper COQ whose (proper) mutation produces a non-proper~COQ. 
\end{example}

\begin{figure}[ht]
\vspace{-10pt}
\begin{tikzpicture}
\filldraw[black] (0,0)++(135:1cm) circle (1.5pt) node[above left=-1pt] {$a$} coordinate (a);
\filldraw[black] (0,0)++(45:1cm) circle (1.5pt) node[above right=-1pt] {$b$} coordinate (b);
\filldraw[black] (0,0)++(-45:1cm) circle (1.5pt) node[below right=-1pt] {$c$} coordinate (c);
\filldraw[black] (0,0)++(-135:1cm) circle (1.5pt) node[below left=-1pt] {$d$} coordinate (d);
\draw[black, dashed, decoration={markings, mark=at position 0 with {\arrow{<}}}, postaction={decorate}] (0,0) circle (1cm);
\draw[black, -{stealth}, shorten >=3pt, shorten <= 3pt] (a) -- (b);
\draw[black, -{stealth}, shorten >=3pt, shorten <= 3pt] (b) -- (d);
\draw[black, -{stealth}, shorten >=3pt, shorten <= 3pt] (c) -- (d);
\draw[black, -{stealth}, shorten >=3pt, shorten <= 3pt] (c) -- (a);
\draw (0,0-1.3) node[below] {$(a, b, c, d)$};
\end{tikzpicture}
\hspace{.3in} 
$\mathrel{\raisebox{2cm}{$\stackrel{\textstyle\mu_a}{\rule[.5ex]{3em}{0.5pt}}$}}$
\hspace{.3in} 
\begin{tikzpicture}
\draw[black, dashed, decoration={markings, mark=at position 0.875 with {\arrow{<}}}, postaction={decorate}] (0,0) circle (1cm);
\filldraw[red] (0,0)++(90:1cm) circle (1.5pt) node[above, black] {$b$} coordinate (b);
\filldraw[black] (0,0)++(0:1cm) circle (1.5pt) node[right] {$a$} coordinate (a);
\filldraw[black] (0,0)++(-90:1cm) circle (1.5pt) node[below] {$c$} coordinate (c);
\filldraw[black] (0,0)++(180:1cm) circle (1.5pt) node[left] {$d$} coordinate (d);

\draw[black, -{stealth}, shorten >=3pt, shorten <= 3pt] (b) -- (a);
\draw[black, -{stealth}, shorten >=3pt, shorten <= 3pt] (b) -- (d);
\draw[black, -{stealth}, shorten >=3pt, shorten <= 3pt] (c) -- (d);
\draw[black, -{stealth}, shorten >=3pt, shorten <= 3pt] (a) -- (c);
\draw[black, -{stealth}, shorten >=3pt, shorten <= 3pt] (c) -- (b);
\draw (0,0-1.3) node[below] {$(b, a, c, d)$};
\end{tikzpicture}

\caption{Left: a proper COQ $(Q,\sigma)$ with $\sigma=(a,b,c,d)$. 
Right: the mutated COQ $(Q',\sigma')=\mu_a(Q,\sigma)$. 
The cyclic ordering $\sigma'=(b,a,c,d)$ is not proper, as the path $c\rightarrow b\rightarrow d$ makes a left turn.
}
\label{fig:proper after mutation}
\end{figure}

We next discuss several classes of proper COQs. 

\begin{observation}
\label{ob:acyclic-proper}
Any acyclic quiver~$Q$ has a proper cyclic ordering, obtained from any linear ordering in which $v$ precedes $u$ whenever $v \rightarrow u$. 
\end{observation}

\pagebreak[3]

\begin{proposition}
\label{pr:3vertex proper}
 Any quiver on $n\le 3$ vertices possesses a proper cyclic ordering. 
\end{proposition}

\begin{proof}
Any COQ on $n\le2$ vertices is proper.
The case $n=3$ has been treated in Example~\ref{eg:3-vert-always-proper}. 
\end{proof}

\begin{lemma}
\label{lem:n=3-tot-proper}
A mutation of a proper COQ on $n\le 3$ vertices yields a proper~COQ. 
\end{lemma}

\begin{proof}
This claim is verified by a straightforward case-by-case analysis. 
\end{proof}

\begin{example}
\label{eg:trees-are-proper}
Let $Q$ be an oriented tree quiver (either weighted or not). 
All cyclic orderings of $Q$ are wiggle equivalent to each other, see Proposition~\ref{prop:tree wiggles}. 
It is easy to show (say, by induction on the number of vertices) that at least one of these orderings is proper. 
Thus the wiggle equivalence class of~$Q$ is proper. 
\end{example}

\begin{example} 
\label{eg:3-colors}
Let $Q$ be a quiver whose vertices are colored in three colors, say ${\{-1, 0,1\}}$, so that every arrow originating at a vertex of color $-1$ (resp., $0, 1$) points towards a vertex of color $0$ (resp., $1, -1$). 
Then $Q$ has a proper cyclic ordering obtained from a linear ordering in which the vertices of color $-1$ precede the vertices of color $0$, which in turn precede the vertices of color $1$. 
(The ordering among the vertices of the same color does not matter, as all choices are wiggle equivalent.) 
\end{example}

One class of quivers to which the construction in Example~\ref{eg:3-colors} applies are 
the quivers associated to plane \emph{divides}, see~\cite{FPST}. 
Let us briefly review this construction, whose origins go back to N.~A'Campo~\cite{acampo} and S.~Guse\u\i n-Zade~\cite{gusein-zade-1}. 

\begin{definition}
\label{def:divide-quiver}
Given a planar divide~$D$ (roughly, a collection of intervals and circles immersed in an ambient disk), 
the associated quiver $Q(D)$ is constructed  as follows, cf.\ Figure~\ref{fig:E6-morsifications}. 
At each node of~$D$ (a point of self-intersection), place a vertex of $Q(D)$ and color it~$0$. 
(These vertices are shown in red in the figures.) 
Inside each bounded region of~$D$, place one vertex of $Q(D)$ and color it either $1$ or~$-1$,
making sure that adjacent regions receive different colors.
(In the figures, these are the hollow blue and solid blue vertices.) 
Draw arrows from each node (colored~0) to the vertices colored~$1$ situated in adjacent regions.
Draw arrows towards each such node from the vertices colored~$-1$ situated in adjacent regions.
For each pair of regions sharing a segment of their boundaries,
draw an arrow from the vertex labeled~$1$ towards the vertex labeled~$-1$. 
The resulting quiver $Q(D)$ will have a natural (canonical up to wiggles) 
proper cyclic ordering, as described in Example~\ref{eg:3-colors}. 
\end{definition}

\begin{figure}[ht]
\begin{center}
\setlength{\unitlength}{2pt}
\begin{picture}(50,40)(0,0)
\thicklines
\lightgray{\qbezier(0,0)(60,60)(60,20)
\qbezier(60,0)(0,60)(0,20)
\qbezier(60,20)(60,0)(30,0)
\qbezier(0,20)(0,0)(30,0)
}
\linethickness{1pt}

\put(6.2,6){\red{\circle*{2.5}}}
\put(6.2,2){\makebox(0,0){\red{0}}}
\put(53.8,6){\red{\circle*{2.5}}}
\put(53.8,2){\makebox(0,0){\red{0}}}
\put(30,26.5){\red{\circle*{2.5}}}
\put(30,31){\makebox(0,0){\red{0}}}

\put(30,6){\lightblue{\circle{2.5}}}
\put(30,2){\makebox(0,0){\lightblue{1}}}
\put(6.2,26.5){\lightblue{\circle*{2.5}}}
\put(6.2,31){\makebox(0,0){\lightblue{-1}}}
\put(53.8,26.5){\lightblue{\circle*{2.5}}}
\put(53.8,31){\makebox(0,0){\lightblue{-1}}}

\put(9,6){{\vector(1,0){18}}}
\put(51,6){{\vector(-1,0){18}}}
\put(9,26.5){{\vector(1,0){18}}}
\put(51,26.5){{\vector(-1,0){18}}}
\put(6.2,24){{\vector(0,-1){15.5}}}
\put(30,24){{\vector(0,-1){15.5}}}
\put(53.8,24){{\vector(0,-1){15.5}}}
\put(28,8){{\vector(-8,7){19.5}}}
\put(32,8){{\vector(8,7){19.5}}}
\end{picture}

\end{center}
\vspace{5pt}
\caption{A divide $D$ (drawn in gray) and the associated quiver $Q(D)$ of type~$E_6$.
Any linear ordering in which the vertices labeled~$-1$ precede the vertices labeled~$0$,
which in turn precede the vertices labeled~$1$, gives rise to a proper cyclic ordering. 
}
\label{fig:E6-morsifications}
\end{figure}

\begin{example}
\label{eg:proper from alg}
Figure~\ref{fig:Q-plabic-hex} shows a plane divide and the associated $7$-vertex quiver of type~$E_6^{(1)}$. 
The cyclic ordering $\sigma=(1, 2, 3, 4, 5, 6)$ makes it into a proper COQ.
\end{example}

\begin{figure}[ht]
\begin{center}
\vspace{-10pt}
\begin{tabular}{cc}
\includegraphics[scale=0.08, trim=3cm 4cm 4cm 4cm, clip]{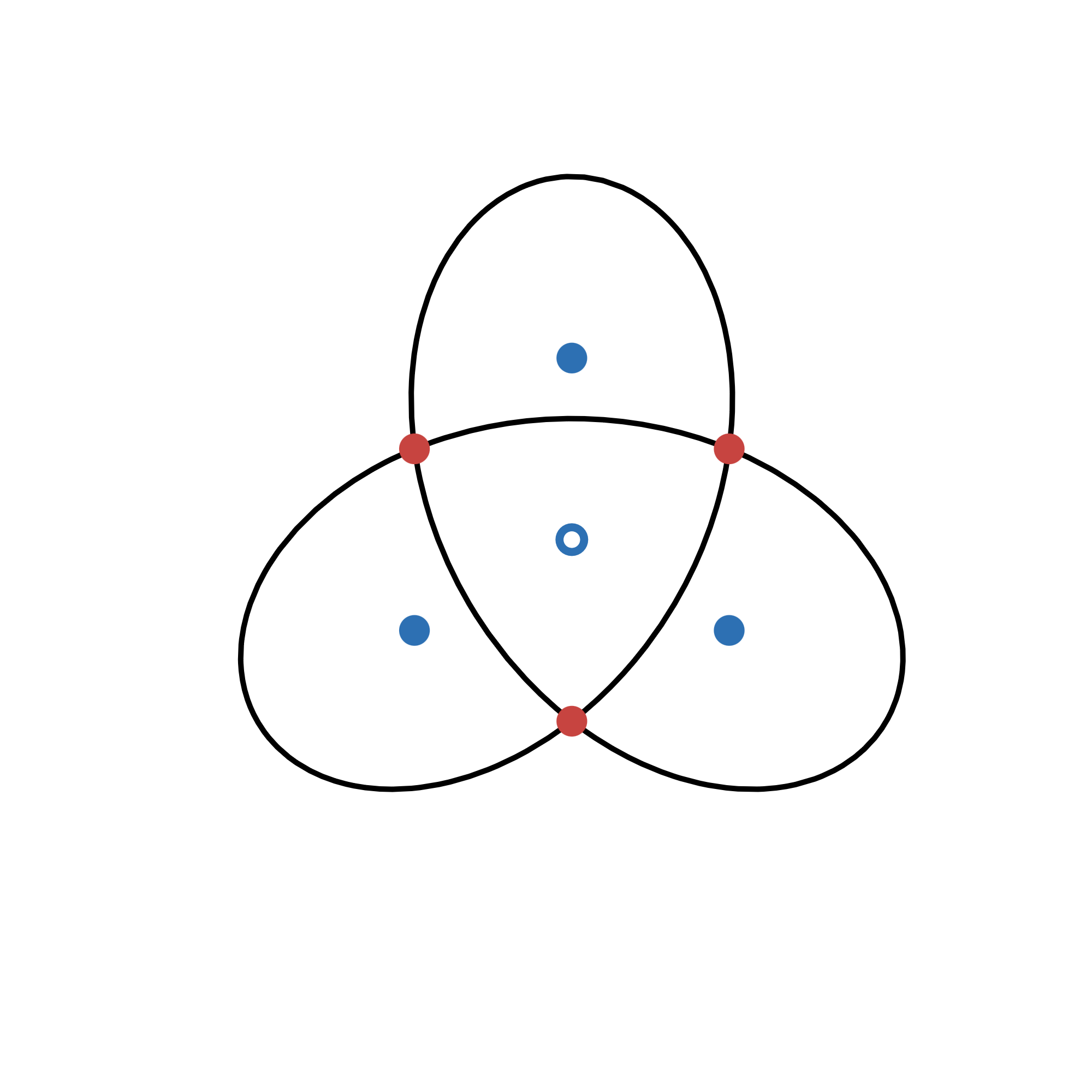}
&
\qquad $\begin{array}{c}
\ \\[-2.3in]
\begin{tikzcd}[arrows={-stealth}, sep=1.5em]
& \lightblue{\mathbf{1}} \arrow[ld] \arrow[rd] & \\[-13pt]
\red{6}  \arrow[rd]   & & \red{4} \arrow[ld]\\[-15pt]
&  \lightblue{7}  \arrow[uu] \arrow[ld] \arrow[rd] \\[-15pt]
\lightblue{\mathbf{3}} \arrow[uu] \arrow[rd] & & \lightblue{\mathbf{2}} \arrow[uu] \arrow[ld] \\[-13pt]
& \red{5} \arrow[uu]
\end{tikzcd}\end{array}$
\end{tabular}
\end{center}
\vspace{-25pt}
\caption{A plane divide and the corresponding quiver.
Solid (resp., hollow) blue vertices on the left correspond to bold (resp., regular) labels on the right. 
}
\vspace{-10pt}
\label{fig:Q-plabic-hex}
\end{figure}


There is a general construction, due to N.~A'Campo \cite{acampo-ihes, AC1}, 
that associates a link $L(D)$ to any plane divide~$D$;
see, e.g., \cite[Section~7]{FPST} and references therein.
The Alexander polynomial of a link associated to a divide~$D$ agrees with the Alexander polynomial
of the COQ associated with~$D$, as in Definition~\ref{def:divide-quiver}. 

\begin{example}
\label{eg:divides-collide}
E.~Yoshinaga and M.~Suzuki \cite{Yoshinaga} discovered non-isotopic divide links that share the same Alexander polynomial.
This observation produces interesting pairs of cyclically ordered quivers whose Alexander polynomials coincide.  
One example of such a collision is shown in Figures~\ref{fig:D12}--\ref{fig:divide-cubic+2lines}. 
In both cases, the Alexander polynomials are given by 
\begin{equation*}
\Delta_Q(t)=t^{12} - t^{11} -1 +1; 
\end{equation*}
furthermore, the corresponding cosquares are conjugate in $\GL(n,\ZZ)$. 
However, the two quivers are not mutation equivalent as the first one is of finite type but the second one is not.
\end{example}

\hide{
The corresponding exchange matrix appears below.
\begin{equation*}
B = 
\left[
\begin{array}{cccccccccccc}
0 & 0 & 0 & 1 & 1 & 1 & 0 & 0 & 0 & 0 & -1 & -1 \\
0 & 0 & 0 & 0 & 1 & 0 & 1 & 1 & 0 & 0 & -1 & -1 \\
0 & 0 & 0 & 0 & 0 & 0 & 0 & 1 & 1 & 1 & 0 & -1 \\
-1 & 0 & 0 & 0 & 0 & 0 & 0 & 0 & 0 & 0 & 1 & 0 \\
-1 & -1 & 0 & 0 & 0 & 0 & 0 & 0 & 0 & 0 & 1 & 1 \\
-1 & 0 & 0 & 0 & 0 & 0 & 0 & 0 & 0 & 0 & 0 & 1 \\
0 & -1 & 0 & 0 & 0 & 0 & 0 & 0 & 0 & 0 & 1 & 0 \\
0 & -1 & -1 & 0 & 0 & 0 & 0 & 0 & 0 & 0 & 0 & 1 \\
0 & 0 & -1 & 0 & 0 & 0 & 0 & 0 & 0 & 0 & 0 & 1 \\
0 & 0 & -1 & 0 & 0 & 0 & 0 & 0 & 0 & 0 & 0 & 0 \\
1 & 1 & 0 & -1 & -1 & 0 & -1 & 0 & 0 & 0 & 0 & 0 \\
1 & 1 & 1 & 0 & -1 & -1 & 0 & -1 & -1 & 0 & 0 & 0
\end{array}
\right]
.
\end{equation*}
} 

\begin{figure}[ht]
\begin{center}
\vspace{5pt}
\begin{tabular}{c}
\includegraphics[scale=0.18, trim=3cm 4cm 4cm 4cm, clip]{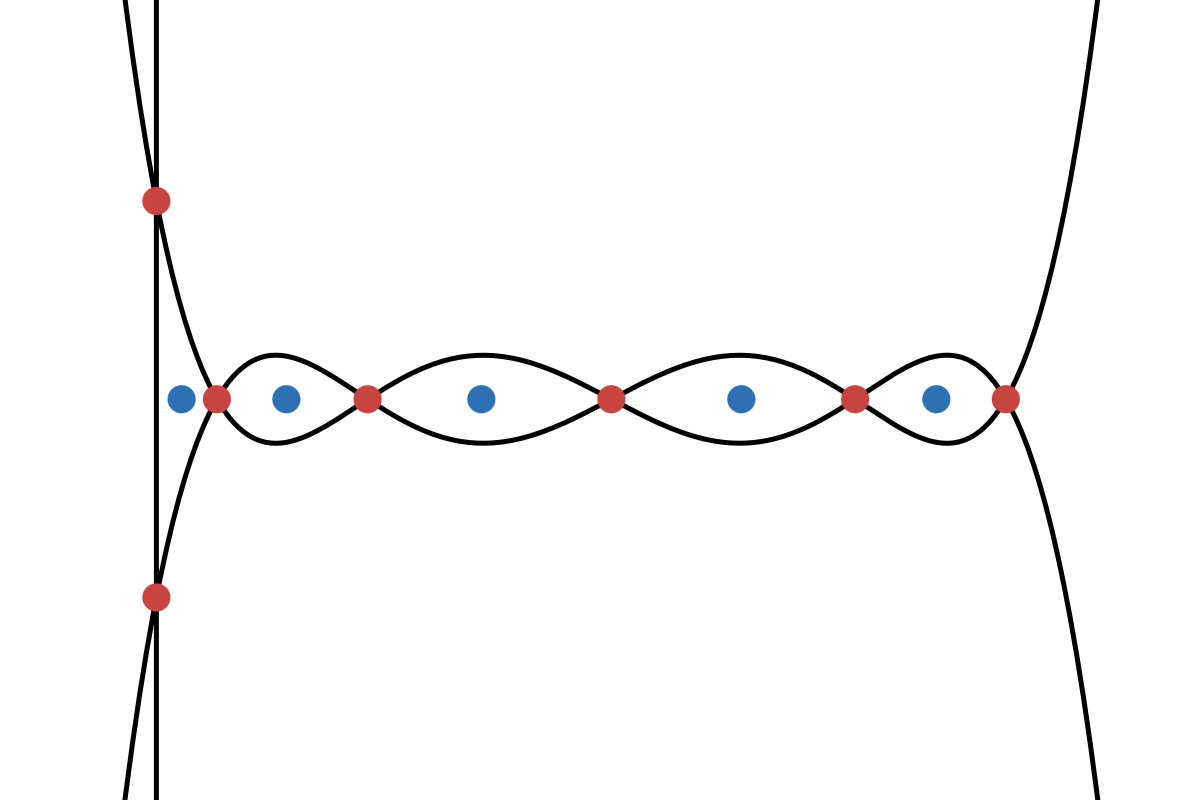}
\\
$\begin{array}{c}
\ \\[-.1in]
\begin{tikzcd}[arrows={-stealth}, sep=1em]
\red{6} \\
& \lightblue{\mathbf{1}} \arrow[lu] \arrow[ld] \arrow[r]& \red{8} & \lightblue{\mathbf{2}} \arrow[r] \arrow[l] & \red{9} & \lightblue{\mathbf{3}} \arrow[r] \arrow[l] & \red{10} & \lightblue{\mathbf{4}} \arrow[r] \arrow[l] & \red{11} & \lightblue{\mathbf{5}} \arrow[r] \arrow[l] & \red{12} 
\\
\red{7}  
\end{tikzcd}
\end{array}$
\end{tabular}
\end{center}
\caption{A plane divide and the associated COQ of type $D_{12}\,$.}
\label{fig:D12}
\end{figure}

\begin{figure}[ht]
\begin{center}
\begin{tabular}{cc}
\includegraphics[scale=0.18, trim=5cm 1cm 8cm 4cm, clip]{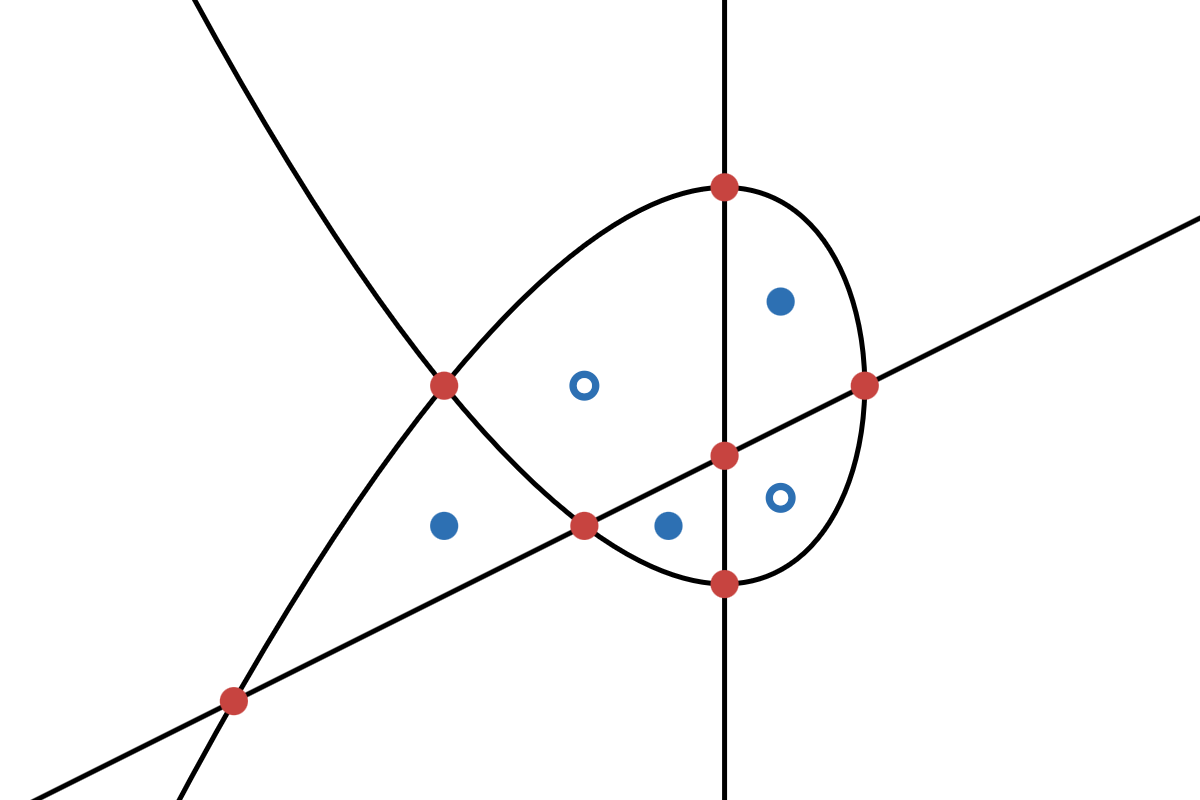}
&
\qquad $\begin{array}{c}
\ \\[-1.8in]
\begin{tikzcd}[arrows={-stealth}, sep=1.5em]
& & \red{6} \arrow[d] & \lightblue{\mathbf{1}} \arrow[l] \arrow[r] \arrow[d]& \red{4} \arrow[d] \\
& \red{9}  \arrow[r]   & \lightblue{{12}} \arrow[ru] \arrow[dl]\arrow[dr]  & \red{5} \arrow[l]\arrow[r] & \lightblue{{11}} \arrow[lu]\arrow[ld]\\
\red{10}  &  \lightblue{\mathbf{3}} \arrow[l] \arrow[r] \arrow[u] & \red{8} \arrow[u] & \lightblue{\mathbf{2}} \arrow[r] \arrow[l] \arrow[u] & \red{7}\arrow[u]  
\end{tikzcd}\end{array}$
\end{tabular}
\end{center}
\caption{Left: A plane divide involving a nodal cubic 
and two lines intersecting~it at $2$ and $3$ smooth points, respectively. 
Right: The associated cyclically ordered~quiver. 
Solid (resp., hollow) blue vertices on the left correspond to bold (resp., regular) labels on the right.
The Alexander polynomial of this divide/COQ is the same as the Alexander polynomial for the divide/COQ in Figure~\ref{fig:D12}.}
\vspace{-15pt}
\label{fig:divide-cubic+2lines}
\end{figure}

\begin{proposition}
\label{prop:n-cycle proper vertices 2}
Let $Q$ be a COQ with $n$ vertices and $n$ arrows whose underlying graph is an undirected $n$-cycle.
Assume that $r$ arrows in~$Q$ point in one direction and $\ell=n-r$ arrows point in the opposite direction.
(Thus $Q$ is of type $\tilde A(r,\ell)$, in the terminology of~\cite{cats1}.) 
Recall that $1-\ell\le \wind(C,\sigma)\le r-1$. 
If $\wind(Q) \notin \{1-\ell, r-1\}$, then $\wiggleQ$ is proper.
\end{proposition}

\begin{proof}
This is immediate from Proposition~\ref{prop:n-cycle proper vertices}.
\end{proof}

\begin{remark}
\label{rem:wiggle-inequiv-proper-orderings}
Proposition~\ref{prop:n-cycle proper vertices 2} provides many examples of quivers 
that have multiple wiggle-inequivalent proper cyclic orderings. 
\end{remark}

Recall from Definition~\ref{def:Graph-Q} that we denote by $\GQ$ the unoriented simple graph underlying a quiver~$Q$. 

\begin{lemma}
\label{lem:small winding wiggles}
For a vertex $j$ in a COQ~$Q$, the following are equivalent: 
\begin{itemize}[leftmargin=.25in]
\item[{\rm (a)}] 
$j$ is proper in the wiggle equivalence class of~$Q$ 
(cf.\ Definition~\ref{def:wiggle-class-proper-vertex});
\item[{\rm (b)}] 
for every chordless cycle~$C\ni j$ in~$\GQ$, 
vertex~$j$ is proper in the wiggle equivalence class of the subCOQ of~$Q$ supported on~$C$ (cf.\ Remark~\ref{rem:subcoq}). 
\end{itemize}
\end{lemma}

\begin{proof}
Condition (a) implies condition (b) by restriction to a subquiver. 

Assume that (b) holds. 
We will argue that if $Q$ contains left-turn two-arrow paths going through~$j$,
then $Q$ can be wiggled to decrease the number of such paths. 


Let $i \rightarrow j \rightarrow k$ be a left turn, and let $<$ be a compatible linear ordering with $i < k < j$. 
We may assume without loss of generality that every vertex $v$ with $i < v < k$ is not adjacent to $j$ in~$Q$,
i.e., there are no arrows $j\to v$ or $v\to j$. 

\noindent

\textbf{Case~1:} 
There exists a path ${(i = v_0 - \cdots - v_r= k)}$ in $K_Q$ with $v_0 < v_1 < \cdots < v_r$.
Take such a path with the smallest value of~$r$.
The cycle $C = (j - v_0 - \cdots - v_r - j)$ is chordless: 
by assumption, there are no chords incident to~$j$, 
and minimality of $r$ implies there are no chords between the~$v_i$.  
Further, $j$ is not proper in the subCOQ supported on~$C$ 
and moreover $C$ allows no wiggles. 
This contradicts~(b).

\textbf{Case~2:} 
There is no such path from $i$ to~$k$. 
Identify all vertices $k'<k$ that are connected to $i$ by such a path, 
then wiggle all of them (including~$i$) past $k$ to create a wiggle equivalent COQ with fewer left turns at $j$. 
\end{proof}

\begin{theorem}
\label{th:small winding proper}
A COQ $Q$ is proper if and only if every (full) subCOQ of $Q$ supported on a 
chordless cycle in $\GQ$ is proper. 
\end{theorem}

\begin{proof}
This follows directly from Lemma~\ref{lem:small winding wiggles}. 
%
%
\end{proof}

\begin{example}
\label{eg:E6-1-cycles}
The COQ $Q$ in Example~\ref{eg:proper from alg} has $7$ subCOQs supported on chordless cycles: 
\begin{align*}
C_1 &= (7 \rightarrow 1 \rightarrow 4 \rightarrow 7), \qquad
C_2 = (7 \rightarrow 1 \rightarrow 6 \rightarrow 7), \\
C_3 &= (7 \rightarrow 2 \rightarrow 4 \rightarrow 7), \qquad 
C_4 = (7 \rightarrow 2 \rightarrow 5 \rightarrow 7), \\
C_5 &= (7 \rightarrow 3 \rightarrow 5 \rightarrow 7), \qquad 
C_6 = (7 \rightarrow 3 \rightarrow 6 \rightarrow 7), \\
& \qquad C_7 =(1 \rightarrow 4 \leftarrow 2 \rightarrow 5 \leftarrow 3 \rightarrow 6 \leftarrow 1).
\end{align*}
By Theorem~\ref{th:small winding proper}, in order to verify that~$Q$ is proper,
it is enough to check that the subquivers $C_1,\dots,C_7$ are proper. 
Indeed,  each of the oriented cycles $C_1,\dots,C_6$ has the form $(7\rightarrow i \rightarrow j\rightarrow 7)$
for $i<j<7$, whereas $C_7$ contains no two-arrow paths. 
\end{example}

\begin{example}
\label{eg:proper from inherited}
Consider the quiver shown in Figure~\ref{fig:Q-punctured-annulus}, 
along with the cyclic ordering $\sigma=(a, b, c, d,e)$.
The chordless cycles are: 
\begin{align*}
C_1 &= (b \rightarrow c \rightarrow e \rightarrow  b), \\
C_2 &= (b \rightarrow d \rightarrow e \rightarrow b), \qquad C_4 = (a\rightarrow b \rightarrow c \leftarrow a), \\
C_3 &= (a \rightarrow b \rightarrow d \rightarrow a), \qquad\! C_5 = (a \rightarrow c \rightarrow e \leftarrow d \rightarrow a),. 
\end{align*}
One quickly checks that these subquivers are proper, so $Q$ is proper by Theorem~\ref{th:small winding proper}. 
\end{example}

\section{Opposite COQs}
\label{sec:opposite}

In this section, we discuss the variation of our construction wherein quivers are considered
\emph{up to taking opposites}, cf.\ Definition~\ref{def:opposite-quiver} below.

\begin{definition}
\label{def:opposite-quiver}
Let $Q$ be a COQ. 
The \emph{opposite} COQ, denoted $Q^{\textup{opp}}$, has all its arrows, as well as its cyclic ordering, 
reversed with respect to~$Q$.
\end{definition}

\begin{example}
The two COQs shown in Figure~\ref{fig:opposite-coq} are each other's opposites.

\end{example}

\begin{figure}[ht]
\begin{center}
\vspace{-10pt}
\begin{equation*}
\begin{tikzcd}[arrows={-stealth}, sep=2em]
  a  \arrow[r, "p"] & b \arrow[d, "q"]  \\
 d\arrow[u, "s"]  & c   \arrow[l, swap, "r"] 
\end{tikzcd}
\qquad\qquad
\begin{tikzcd}[arrows={-stealth}, sep=2em]
  a  \arrow[r, "s"] & d \arrow[d, "r"]  \\
 b\arrow[u, "p"]  & c   \arrow[l, swap, "q"] 
\end{tikzcd}
\end{equation*}
\end{center}
\caption{Two opposite cyclically ordered quivers.}
\vspace{-10pt}
\label{fig:opposite-coq}
\end{figure}

We omit the straightforward proofs of the following lemmas. 

\begin{lemma}
\label{lem:opposite-proper}
Passing to the opposite COQ is an involution that preserves properness of individual vertices. 
\end{lemma}

\begin{lemma}
Taking the opposites commutes with
\begin{itemize}[leftmargin=.2in]
\item proper mutations of COQs;
\item computing the integral congruence class of the unipotent companion $U$;
\item computing the integral conjugacy class of the cosquare of $U$.
\end{itemize}
\end{lemma}

\begin{lemma}
\label{lem:M-opp-vs-M}
Let $M$ (resp., $M^{\textup{opp}}$) denote the cosquare of a unipotent companion for a COQ~$Q$ 
(resp., its opposite~$Q^{\textup{opp}}$). 
Then $M^{\textup{opp}}$ is conjugate to $M^T$ in~$\GL(n,\ZZ)$.
\end{lemma}

\begin{remark}
\label{rem:rational transpose}
Over the rationals, every square matrix is similar to its transpose; see, e.g., \cite[Theorem~66]{Kaplansky69}. 
This is however false over the integers.
In particular, 
in the notation of Lemma~\ref{lem:M-opp-vs-M}, the matrices~$M$ and $M^{\textup{opp}}$ are not always 
conjugate to each other in $\GL(n,\ZZ)$.
For example, see the first two quivers in Figure~\ref{fig:5-quivers}. 
\end{remark}

It will be convenient to introduce the following notions.

\begin{definition}
\label{def:sim transpose}
We say that two matrices $A, B\in \GL(n,\ZZ)$ are \emph{``conjugate over $\ZZ$ up to transpose''} 
if $B$ is conjugate in $\GL(n,\ZZ)$ to either $A$ or~$A^T$.
It is easy to check that conjugacy up to transpose 
is an equivalence relation on $\GL(n,\ZZ)$.
\end{definition}

Corollary~\ref{cor:proper-mutations-preserve-cosquare} and 
Lemmas~\ref{lem:opposite-proper}--\ref{lem:M-opp-vs-M}
imply the following result. 

\begin{corollary}
\label{cor:sim trans vs opposites}
If two COQs, viewed up to opposites, are related to each other via proper mutations, 
then their respective cosquares are conjugate over $\ZZ$ up to transpose.  
\end{corollary}

The converse to Corollary~\ref{cor:sim trans vs opposites} is false in general, 
even if we require the COQs to be proper (or totally proper, see Definition~\ref{def:totally-proper});
cf., for instance, Example~\ref{eg:9-vertex-trees}.
\hide{
\begin{remark}
For COQs $Q$ and $Q'$, consider the following statements:
\begin{enumerate}
\item[$(1)$] $Q$ and $Q'$ are related to each other by wiggles and proper mutations;
\item[$(2)$] the cosquares of their unipotent companions are conjugate in $\GL(n,\ZZ)$.
\end{enumerate}
By Corollary~\ref{cor:proper-mutations-preserve-cosquare}, (1)$\Rightarrow$(2). 
The converse implication $(2)\Rightarrow(1)$ is false in general, 
even if we require $Q$ and $Q'$ to be proper (or totally proper, see Definition~\ref{def:totally-proper}) 
and even if we allow replacing $Q$ or $Q'$ by their opposite COQ. 
See, e.g., Example~\ref{eg:9-vertex-trees}.
\end{remark}
}
On the other hand, experimental evidence suggests that the converse to
Corollary~\ref{cor:sim trans vs opposites} holds for 3-vertex quivers:

\begin{conjecture}
\label{conj:3 vert mu equiv}
Let $Q$ and $Q'$ be two $3$-vertex proper COQs.
(Recall that by Proposition~\ref{pr:3vertex proper}, any $3$-vertex quiver has a proper cyclic ordering.)
If the cosquares of the unipotent companions $U_Q$ and $U_{Q'}$ are conjugate over~$\ZZ$ up to transpose,
then  $Q$ and $Q'$ are mutation equivalent, up to taking opposites. 
\end{conjecture}

We have verified Conjecture~\ref{conj:3 vert mu equiv} for all 3-vertex quivers with $|M_Q| < 5000$.
Cf.\ also Theorem~\ref{th:signed-braid-n=3}. 

\begin{example}
\label{eg:frob 3 vert}
Consider 5 cyclically oriented $3$-vertex quivers 
shown in Figure~\ref{fig:5-quivers}. 
These quivers are minimal in their respective mutation classes;
consequently these mutation classes are all distinct, cf.~\cite{BBH}. 

All these quivers have the same Markov invariant $M_Q=-50$ and consequently the same Alexander polynomial. 

The $\GL(3,\ZZ)$ conjugacy classes of unipotent companions of the first four quivers in Figure~\ref{fig:5-quivers}
are distinct. 
For the last pair of opposite quivers (with multiplicities ${4,5,7}$), these conjugacy classes coincide. 
\end{example}

\begin{figure}[ht]
\begin{center}
\vspace{-10pt}
\begin{equation*}
 \begin{tikzcd}[arrows={-stealth}, sep=2em]
  a  \arrow[r,  "9"]   
  & b  \arrow[dl, "7"]
  \\
 c  \arrow[u, "3"] 
   & 
\end{tikzcd}
\qquad 
 \begin{tikzcd}[arrows={-stealth}, sep=2em]
  a  \arrow[r,  "7"]   
  & b  \arrow[dl, "9"]
  \\
 c  \arrow[u, "3"] 
   & 
\end{tikzcd}
\qquad 
 \begin{tikzcd}[arrows={-stealth}, sep=2em]
  a  \arrow[r,  "5"]   
  & b  \arrow[dl, "5"]
  \\
 c  \arrow[u, "5"] 
   & 
\end{tikzcd}
\qquad  
 \begin{tikzcd}[arrows={-stealth}, sep=2em]
  a  \arrow[r,  "7"]   
  & b  \arrow[dl, "5"]
  \\
 c  \arrow[u, "4"] 
   & 
\end{tikzcd}
\qquad
 \begin{tikzcd}[arrows={-stealth}, sep=2em]
  a  \arrow[r,  "5"]   
  & b  \arrow[dl, "7"]
  \\
 c  \arrow[u, "4"] 
   & 
\end{tikzcd}
\end{equation*}
\end{center}
\caption{Five 3-vertex quivers with $M_Q=-50$.
The first two quivers are opposites of each other; so are the last two quivers.
The middle quiver is opposite to itself.}
\vspace{-10pt}
\label{fig:5-quivers}
\end{figure}

\begin{remark}
\label{rem:3vert mu acyclic opposites}
We do not know of any unlabeled $3$-vertex \emph{acyclic} quiver $Q$ such~that 
$U_Q^{-T} U_Q$ and its transpose (cf.\ Lemma~\ref{lem:M-opp-vs-M}) 
are conjugate in $\GL(n,\ZZ)$ 
but $Q$ is not mutation equivalent to~$M^{\textup{opp}}$. 
That is, one may be able to drop ``up to taking opposites'' and ``up to transpose'' from Conjecture~\ref{conj:3 vert mu equiv} when $Q$ and $Q'$ are mutation-acyclic. 
\end{remark}

\newpage

\section{Vortices and vortex-free quivers}

We next discuss examples of quivers that do not possess any proper cyclic orderings. 

\begin{definition}
\label{def:vortex}
A \emph{vortex} is a complete $4$-vertex quiver $Q$ such that
one of the vertices of $Q$ is a source or a sink, and 
the remaining three vertices of $Q$ support an oriented $3$-cycle. 
Equivalently, a vortex is a complete 4-vertex quiver that contains an oriented 3-cycle but not an oriented 4-cycle. 

The unique sink/source of a vortex is called its~\emph{apex}. 
See Figure~\ref{fig:vortices}. 

We say that a quiver $Q$ \emph{contains a vortex} if one of its $4$-vertex (induced) subquivers is a vortex. 
A quiver that does not contain a vortex is called \emph{vortex-free}.
This terminology goes back to D.~E.~Knuth \cite[Section~4]{Knuth}. 
\end{definition}

\begin{figure}[ht]
\begin{equation*}
\begin{tikzcd}[arrows={-stealth}, sep=15, cramped]
  & a \ar[d] \ar[ddl]  & \\ 
  & d & \\[-13pt]
  b  \ar[ur] \ar[rr]
  & & c \ar[ul] \ar[uul]
\end{tikzcd} 
\hspace{.5in} 
\begin{tikzcd}[arrows={-stealth}, sep=15, cramped]
  & a  \ar[ddl]  & \\ 
  & d \ar[u] \ar[dl] \ar[dr]& \\[-13pt]
  b   \ar[rr]
  & & c  \ar[uul]
\end{tikzcd}
\hspace{.5in} 
\begin{tikzcd}[arrows={-stealth}, sep=15, cramped]
  & a \ar[d] \ar[ddr]  & \\ 
  & d & \\[-13pt]
  b  \ar[ur] \ar[uur]
  & & c \ar[ul] \ar[ll]
\end{tikzcd} 
\hspace{.5in} 
\begin{tikzcd}[arrows={-stealth}, sep=15, cramped]
  & a  \ar[ddr]  & \\ 
  & d \ar[u] \ar[dl] \ar[dr]& \\[-13pt] 
  b   \ar[uur]
  & & c \ar[ll]
\end{tikzcd}
\end{equation*}
\vspace{2pt}
\caption{Four vortices with an apex at vertex $d$.}
\vspace{-7pt}
\label{fig:vortices}
\end{figure}

\begin{proposition}
\label{prop: vortex equiv to 4 vertex proper}
 A $4$-vertex quiver has a proper cyclic ordering if and only if it is not a vortex. 
\end{proposition}

\begin{proof}
Let $Q$ be a $4$-vertex quiver. 

\noindent
\textbf{Case 1:} $Q$ is acyclic (hence not a vortex).
Then $Q$ has a proper cyclic ordering, see Observation~\ref{ob:acyclic-proper}.


\noindent
\textbf{Case 2:} $Q$ has an oriented $4$-cycle (hence $Q$ is not a vortex). 
Then the cyclic ordering induced by this cycle is proper, regardless of the orientations of the remaining~arrows.

\noindent
\textbf{Case 3:} 
$Q$ has no oriented $4$-cycle but has an oriented $3$-cycle $C = (a \rightarrow b \rightarrow c \rightarrow a)$.

\noindent
\textbf{Case 3A:} 
The remaining vertex $d$ is a source or a sink. 
If $d$ is adjacent to all three vertices $a,b,c$, then $Q$ is a vortex;
furthermore it is not proper since  any location of $d$ (with respect to the clockwise 3-cycle~$C$)
will create a left turn at some vertex in~$C$.
If $d$ is adjacent to at most two of the remaining vertices (so $Q$ is not a vortex), 
then it's easy to see that we can always complete 
the cyclic ordering $(a,b,c)$ to a proper ordering of~$Q$. 

\pagebreak[3]

\noindent
\textbf{Case 3B:} 
The vertex $d$ is neither a source nor a sink. This means that $d$ lies in the middle of some oriented 2-arrow path.
(Also, $Q$ is not a vortex, as it has no sink/source vertex.) 
Up to symmetries, there are two cases: 
(1) a 2-arrow path $a \rightarrow d \rightarrow b$ is ruled out since it would create an oriented $4$-cycle; 
(2) a 2-arrow path $a \rightarrow d \rightarrow c$ would allow a proper ordering
(either $(a,b,d,c)$ or $(a,d,b,c)$, depending on the orientation of the arrows between $b$ and~$d$, if any). 
\end{proof}

\begin{corollary}
\label{cor:vortex improper}
A quiver that contains a vortex has no proper cyclic ordering. 
\end{corollary}

\begin{remark}
\label{rem:proper-doesn't-propagate}
 As shown in Example~\ref{eg:proper-doesn't-propagate}, 
mutation of a proper COQ does not necessarily produce a proper COQ, 
even if the mutated COQ is vortex-free (and does possess a totally proper cyclic ordering). 

Furthermore, some quivers that possess a proper cyclic ordering can be mutated to a quiver that does not have this property.
An example is given in Figure~\ref{fig:vortex mutation}.
\end{remark}

\pagebreak[3]


\begin{figure}[ht]
\begin{center}
\begin{tikzcd}[arrows={-stealth}, sep=36, cramped]
 b \ar[r] \ar[rd, "\!2", near end] & c \arrow[ld, "2\!", swap, near end] \ar[d] \\[-5pt]
a \ar[u] & d \ar[l] 
\end{tikzcd} 
\hspace{.3in} 
$\mathrel{\raisebox{0cm}{$\stackrel{\textstyle\mu_a}{\rule[.5ex]{3em}{0.5pt}}$}}$
\hspace{.3in} 
\begin{tikzcd}[arrows={-stealth}, sep=36, cramped]
 a \ar[r, "2"] \ar[rd] & c \arrow[ld] \ar[d] \\[-5pt]
b \ar[u] \ar[r] & d 
\end{tikzcd} 
\end{center}
\vspace{3pt}
\caption{The quiver $Q$ on the left has proper cyclic ordering $(a,b,c,d)$.
The quiver $\mu_a(Q)$ on the right is a vortex, so by Proposition~\ref{prop: vortex equiv to 4 vertex proper}, 
it has no proper cyclic ordering.}
\vspace{-10pt}
\label{fig:vortex mutation}
\end{figure}

The following result appears, in different but equivalent form, in the work of D.~E.~Knuth \cite[Section~4]{Knuth} 
and A.~Brouwer \cite[Section~1.B]{brouwer}. 


\begin{proposition} 
\label{prop:vf complete proper }
A complete quiver has at most one proper cyclic ordering. 
Given a complete quiver~$Q$, the following are equivalent:
 \begin{itemize}[leftmargin=1cm] 
\item[{\rm (P)}] $Q$ has a proper cyclic ordering;
\item[{\rm (VF)}] $Q$ is vortex-free.
\end{itemize}
\end{proposition}

\begin{remark}
By Corollary~\ref{cor:vortex improper}, (P) implies (VF) for \emph{any} quiver~$Q$. 
\end{remark}

\pagebreak[3]

\begin{remark}
For incomplete quivers, (VF) does not imply (P), see Figure~\ref{fig:VF but not proper}. 
Cf.\ also Figure~\ref{fig:Q-punctured-annulus}. 
\end{remark}

\begin{figure}[ht]
\begin{center}
\begin{tikzcd}[arrows={-stealth}, sep=small, cramped]
  a\ar[rr]&    & b\ar[dd] \\[2pt]
  & e \ar[ul] \ar[ur] \ar[dr] \ar[dl] & \\ 
  d  \ar[uu] & & c \ar[ll]
\end{tikzcd}
\end{center}
\vspace{5pt}
\caption{A vortex-free quiver with no proper cyclic ordering.}
\vspace{-5pt}
\label{fig:VF but not proper}
\end{figure}

We say that a quiver $Q$ \emph{has a vortex-free completion} if one can add arrows (but not vertices) 
to~$Q$ to get a complete vortex-free quiver. 

\begin{corollary}
\label{cor:VF-completion}
 If $Q$ has a vortex-free completion, then $Q$ has a proper cyclic ordering. 
\end{corollary}

\begin{remark}
The converse to Corollary~\ref{cor:VF-completion} is false:
a quiver that allows a proper cyclic ordering does not necessarily have a vortex-free completion, see Figure~\ref{fig:ProperButNotVFcompletable}.
\end{remark}

\enlargethispage{5pt}

\begin{figure}[ht]
\vspace{-5pt}
\begin{center}
{
\setlength{\unitlength}{.7pt}
\newcommand{\radius}{1.5cm} 
\newcommand{\vertsize}{1.5pt}
\begin{tikzpicture}
\filldraw[black] (0,0)++(180:\radius) circle (2pt) node[left] {$a$} coordinate (a);
\filldraw[black] (0,0)++(120:\radius) circle (2pt) node[above left] {$b$} coordinate (b);
\filldraw[black] (0,0)++(60:\radius) circle (2pt) node[above right] {$c$} coordinate (c);
\filldraw[black] (0,0)++(0:\radius) circle (2pt) node[right] {$d$} coordinate (d);
\filldraw[black] (0,0)++(-60:\radius) circle (2pt) node[below right] {$e$} coordinate (e);
\filldraw[black] (0,0)++(-120:\radius) circle (2pt) node[below left] {$f$} coordinate (f);
\draw[black, dashed, decoration={markings, mark=at position 0.1 with {\arrow{<}}}, postaction={decorate}, opacity=0.5] (0,0) circle (\radius);
\draw[black, -{stealth}, shorten >=3pt, shorten <= 3pt] (a) -- (b);
\draw[black, -{stealth}, shorten >=3pt, shorten <= 3pt] (b) -- (c);
\draw[black, -{stealth}, shorten >=3pt, shorten <= 3pt] (c) -- (d);
\draw[black, -{stealth}, shorten >=3pt, shorten <= 3pt] (d) -- (e);
\draw[black, -{stealth}, shorten >=3pt, shorten <= 3pt] (e) -- (f);
\draw[black, -{stealth}, shorten >=3pt, shorten <= 3pt] (f) -- (a);

\draw[black, -{stealth}, shorten >=3pt, shorten <= 3pt] (c) -- (a);
\draw[black, -{stealth}, shorten >=3pt, shorten <= 3pt] (e) -- (a);
\draw[black, -{stealth}, shorten >=3pt, shorten <= 3pt] (b) -- (d);
\draw[black, -{stealth}, shorten >=3pt, shorten <= 3pt] (f) -- (d);

\draw[black, dashed, shorten >=3pt, shorten <= 3pt] (a) -- (d);
\end{tikzpicture}
}
\end{center}
\vspace{-5pt}
\caption{A proper COQ that cannot be completed to be vortex-free. 
Any orientation of the missing edge $a - d$ would create a vortex.}
\label{fig:ProperButNotVFcompletable}
\end{figure}

\begin{remark}
As shown by D.~E.~Knuth \cite[Section~6]{Knuth}, 
it is NP-hard to determine whether a quiver has a vortex-free completion. 
\end{remark}

The next result will be our primary tool for propagating the properness property. 

\begin{proposition}
\label{pr:propagation-under-VF}
Let $Q$ be a complete proper COQ. 
If the COQ $Q'=\mu_b(Q)$ is vortex-free, then $Q'$ is proper. 
\end{proposition}

\begin{proof}
The case of quivers with at most $3$ vertices follows from Lemma~\ref{lem:n=3-tot-proper}.
We henceforth assume that $Q$ has $\ge4$ vertices. 

A COQ is proper if every $3$-vertex subCOQ of it is proper. 
Any $3$-vertex subCOQ appears in a $4$-vertex subCOQ along with the vertex~$b$. 
The proof will not involve any wiggles, so we may assume, without loss of generality, 
that $Q$ is a complete $4$-vertex COQ with a distinguished vertex~$b$. 
Up to taking the opposite COQ (cf.\ Definition~\ref{def:opposite-quiver}), 
there are only four possible COQs of this kind, shown in Figure~\ref{fig:proper outgoing orientations}. 
(Since $Q$ is proper, it must be vortex-free by Proposition~\ref{prop:vf complete proper }.)

\pagebreak[3]


\vspace{5pt}

\begin{figure}[ht]
{
\newcommand{\radius}{1.2cm} 
\newcommand{\out}{-{stealth}}
\newcommand{\ine}{{stealth}-}
\newcommand\OriQuiv[9]{%
	\filldraw[black] (#1,#2)++(135:\radius) circle (2pt) node[above left=-1pt] {$a$} coordinate (a);
	\filldraw[black] (#1,#2)++(45:\radius) circle (2pt) node[above right=-1pt] {$b$} coordinate (b);
	\filldraw[black] (#1,#2)++(-45:\radius) circle (2pt) node[below right=-1pt] {$c$} coordinate (c);
	\filldraw[black] (#1,#2)++(-135:\radius) circle (2pt) node[below left=-1pt] {$d$} coordinate (d);
	\draw[black, dashed, decoration={markings, mark=at position 0 with {\arrow{<}}}, postaction={decorate}] (#1,#2) circle (\radius);
	\draw[black, #3, shorten >=3pt, shorten <= 3pt] (d) -- (a);
	\draw[black, #4, shorten >=3pt, shorten <= 3pt] (a) -- (b);
	\draw[black, #5, shorten >=3pt, shorten <= 3pt] (b) -- (c);
	\draw[black, #6, shorten >=3pt, shorten <= 3pt] (d) -- (b);
	\draw[black, #7, shorten >=3pt, shorten <= 3pt] (a) -- (c);
	\draw[black, #8, shorten >=3pt, shorten <= 3pt] (d) -- (c);
	\draw (#1,#2-\radius) node[below] {\textbf{#9}};
}
\begin{tikzpicture} 
\OriQuiv{0}{0.2}{\out}{\ine}{\out}{\ine}{\ine}{\ine}{Case 1}
\end{tikzpicture}
{\ \ }
\begin{tikzpicture} 
\OriQuiv{0}{0.2}{\ine}{\out}{\out}{\ine}{\out}{\ine}{Case 2}
\end{tikzpicture}
{\ \ }
\begin{tikzpicture} 
\OriQuiv{0}{0.2}{\out}{\out}{\out}{\ine}{\out}{\ine}{Case 3}
\end{tikzpicture}
{\ \ }
\begin{tikzpicture} 
\OriQuiv{0}{0.2}{\out}{\out}{\out}{\ine}{\ine}{\ine}{Case 4}
\end{tikzpicture}
} 
\caption{Non-vortex 4-vertex complete proper COQs with $|\Out(b)|\!\geq\!2$.
Arrow multiplicities are not shown.}
\label{fig:proper outgoing orientations}
\end{figure}

\textbf{Case 1:} vertex $b$ is a source. 
Mutation at $b$ does not change the cyclic ordering and $\mu_b(Q)$ is again proper.

\textbf{Case 2:} $\Out(b)=\{c,d\}$ and $Q$ is acyclic. 
Mutation at $b$ reverses the arrows incident to $b$ and leaves all other arrow orientations unchanged.
Thus $\mu_b(Q)$ is proper, with cyclic ordering $(b,a,c,d)$.

\textbf{Cases 3 and~4:} $\Out(b)=\{c,d\}$ 
and $Q$ has an oriented $4$-cycle.
%
In this case, we know the arrow orientations of $\mu_b(Q)$ shown in Figure~\ref{fig:mutated b cyclic}.
Regardless of the missing orientations, the vertices $b, c,$ and $d$ are proper in $\mu_b(Q)$. 
By assumption, $\mu_b(Q)$ is vortex-free. 
Hence we cannot have the two-arrow path $c \rightarrow a \rightarrow d$ in~$\mu_b(Q)$. 
Therefore $a$ is also proper.
\end{proof}


\hide{
\begin{figure}[ht]
{
\newcommand{\radius}{1.2cm} 
\newcommand{\out}{-{stealth}}
\newcommand{\ine}{{stealth}-}

\newcommand\OriQuiv[8]{%
	\filldraw[black] (#1,#2)++(135:\radius) circle (2pt) node[above left=-1pt] {$b$} coordinate (b);
	\filldraw[black] (#1,#2)++(45:\radius) circle (2pt) node[above right=-1pt] {$a$} coordinate (a);
	\filldraw[black] (#1,#2)++(-45:\radius) circle (2pt) node[below right=-1pt] {$c$} coordinate (c);
	\filldraw[black] (#1,#2)++(-135:\radius) circle (2pt) node[below left=-1pt] {$d$} coordinate (d);
	\draw[black, dashed, decoration={markings, mark=at position 0 with {\arrow{<}}}, postaction={decorate}] (#1,#2) circle (\radius);
	\draw[black, #3, shorten >=3pt, shorten <= 3pt] (d) -- (a);
	\draw[black, #4, shorten >=3pt, shorten <= 3pt] (a) -- (b);
	\draw[black, #5, shorten >=3pt, shorten <= 3pt] (b) -- (c);
	\draw[black, #6, shorten >=3pt, shorten <= 3pt] (d) -- (b);
	\draw[black, #7, shorten >=3pt, shorten <= 3pt] (a) -- (c);
	\draw[black, #8, shorten >=3pt, shorten <= 3pt] (d) -- (c);
}
\begin{tikzpicture}
\OriQuiv{0}{0}{-, dashed}{\ine}{\ine}{\out}{\out}{\ine}
\end{tikzpicture}
} 
\caption{Known arrow orientations in $\mu_b(Q)$. Arrow multiplicities are not shown.}
\label{fig:mutated b cyclic 2}
\end{figure}
} 

\begin{figure}[ht]
{
\newcommand{\radius}{1.2cm} 
\newcommand{\out}{-{stealth}}
\newcommand{\ine}{{stealth}-}
\newcommand\OriQuiv[8]{%
	\filldraw[black] (#1,#2)++(135:\radius) circle (2pt) node[above left=-1pt] {$b$} coordinate (b);
	\filldraw[black] (#1,#2)++(45:\radius) circle (2pt) node[above right=-1pt] {$a$} coordinate (a);
	\filldraw[black] (#1,#2)++(-45:\radius) circle (2pt) node[below right=-1pt] {$c$} coordinate (c);
	\filldraw[black] (#1,#2)++(-135:\radius) circle (2pt) node[below left=-1pt] {$d$} coordinate (d);
	\draw[black, dashed, decoration={markings, mark=at position 0 with {\arrow{<}}}, postaction={decorate}] (#1,#2) circle (\radius);
	\draw[black, #3, shorten >=3pt, shorten <= 3pt] (d) -- (a);
	\draw[black, #4, shorten >=3pt, shorten <= 3pt] (a) -- (b);
	\draw[black, #5, shorten >=3pt, shorten <= 3pt] (b) -- (c);
	\draw[black, #6, shorten >=3pt, shorten <= 3pt] (d) -- (b);
	\draw[black, #7, shorten >=3pt, shorten <= 3pt] (a) -- (c);
	\draw[black, #8, shorten >=3pt, shorten <= 3pt] (d) -- (c);
}
\begin{tikzpicture}
\OriQuiv{0}{0}{-, dashed}{\ine}{\ine}{\out}{-, dashed}{\ine}
\end{tikzpicture}
} 
\caption{Known arrow orientations in $\mu_b(Q)$. Arrow multiplicities are not shown.}
\label{fig:mutated b cyclic}
\end{figure}

\newpage

\section{Totally proper COQs: requirements}
\label{sec:totally proper}

\begin{definition}
\label{def:totally-proper}
 A COQ is \emph{totally proper} if all COQs in its proper mutation class are proper.
 A cyclic ordering of a COQ is totally proper if that COQ is totally proper.
\end{definition}

The following is immediate from Theorem~\ref{thm:I-N-action}
and Corollaries~\ref{cor:proper-mutations-preserve-cosquare}, \ref{cor:alex-poly-inv}, and~\ref{cor:alex-lattices-inv}.

\begin{corollary}
\label{cor:proper preserves}
Proper mutations of a totally proper COQ preserve the integral congruence class of the unipotent companion~$U$---hence 
the $\GL(n,\ZZ)$ conjugacy class of the cosquare of~$U$ and the associated Alexander polynomial and Alexander lattices.
\end{corollary}

It turns out that if a totally proper cyclic ordering exists, then it is unique:

\begin{theorem}
\label{th:totally-proper-unique}
A quiver may possess at most one totally proper cyclic ordering (up~to wiggles). 
\end{theorem}


\begin{remark}
Corollary~\ref{cor:proper preserves} and Theorem~\ref{th:totally-proper-unique} 
highlight the usefulness of the concept of a totally proper COQ:
within the class of quivers that allow a totally proper cyclic ordering, 
the partial invariants discussed in the previous sections
become true mutation invariants.
(Since a totally proper cyclic ordering is unique, there is no ambiguity involved in defining these invariants.) 
\end{remark}

\begin{remark}
If a COQ $Q$ is totally proper, then any subCOQ of~$Q$ 
is also totally proper.
Thus being totally proper is a hereditary property of COQs, cf.\ Remark~\ref{rem:subcoq}.

In practice, the contrapositive statement is more useful: 
if a COQ has a subCOQ that is not totally proper, then the whole COQ is not totally proper.
\end{remark}

\begin{remark}
As observed in Remark~\ref{rem:proper-doesn't-propagate}, 
properness of COQs does not propagate under (proper) mutations.
On the other hand, the existence of a totally proper cyclic ordering is a mutation-invariant 
and hereditary property of quivers.
\end{remark}

\begin{remark}
\label{rem:testing-totally-proper}
In general, it is hard to determine whether a given quiver has~a totally proper cyclic ordering,
or whether a given COQ is totally proper.
Corollary~\ref{cor:vortex improper} provides a~necessary condition: 
a totally proper quiver, as well as all quivers in its mutation class, must be vortex-free.
This condition is not sufficient: while the quiver
\begin{equation*}
\begin{tikzcd}[arrows={-stealth}, sep=2em]
  c  \arrow[r,"3"] \arrow[rd,swap,"3"]  & a \arrow[r, "2"] & b \arrow[ld, "2"] \\
   & d   \arrow[u, "2"] 
\end{tikzcd}
\end{equation*}
is not mutation-equivalent to a vortex, 
it has no totally proper cyclic ordering. 
\end{remark}

The proof of Theorem~\ref{th:totally-proper-unique} will require some preparations.

\begin{lemma}
\label{lem:new cycle winding number}
Let $(Q,\sigma)$ be a COQ such that the undirected simple graph~$\GQ$ 
is a chordless $n$-cycle \eqref{eq:n-cycle-C}. 
Suppose that $Q$ contains the arrows $v_0 \rightarrow v_1 \rightarrow v_2$;
the directions of the arrows connecting $v_i$ and $v_{i+1}$ for $i\ge2$ can be arbitrary. 
Assume that the vertex $v_1$ is proper in~$Q$.
The COQ $(Q',\sigma')=\mu_{v_1}(Q,\sigma)$ contains an arrow $v_0 \rightarrow v_2$,  
so its underlying undirected graph contains the chordless $(n-1)$-cycle
\begin{equation*}
C'=(v_0 \rightarrow v_2 - \cdots -v_{n-1} - v_n = v_0). 
\end{equation*}
Then $\wind(C',\sigma') = \wind(C',\sigma) = \wind(Q,\sigma)$.
\end{lemma}

\begin{proof}
The restrictions of $\sigma$ and~$\sigma'$ onto~$C'$ coincide, so $\wind(C',\sigma') = \wind(C',\sigma)$.

The summations \eqref{eq:wind-n-cycle}, for $\wind(Q)$ and $\wind(C')$ respectively, are very similar. 
Since proper mutation does not change the cyclic ordering of $\{v_0, v_2, v_3, \ldots, v_n\}$, 
the total number of revolutions remains the same:
\begin{equation*}
\tfrac{1}{n} (\theta(\sigma', v_0, v_2) + \sum_{2 \leq i \leq n-1} \theta(\sigma', v_i, v_{i+1}))
=\tfrac{1}{n} (\theta(\sigma, v_0, v_2) + \sum_{2 \leq i \leq n-1} \theta(\sigma, v_i, v_{i+1})) .
\vspace{-3pt}
\end{equation*}
Since $Q$ is chordless, $Q$ and $C'$ contain the same number of indices~$i$ 
with backward-oriented arrows $v_i \leftarrow v_{i+1}$.
So to establish $\wind(C')=\wind(Q)$, it suffices to verify~that
$\theta(\sigma, v_0, v_2)  = \theta(\sigma, v_0, v_1) + \theta(\sigma, v_1, v_2)$. 
\end{proof}

\begin{proposition}
\label{prop:cycle proper}
Let $Q$ be a COQ such that the undirected simple graph~$\GQ$ 
is a chordless $n$-cycle, cf.\ Proposition~\ref{prop:n-cycle with no wiggles}.
Suppose that $Q$ is totally proper.
Then one of the following situations must occur:
\begin{itemize}[leftmargin=.2in]
\item 
$Q$ is an oriented cycle (with multiplicities) and $\wind(Q) =\pm 1$; 
\item
$Q$ is acyclic (i.e., is not an oriented cycle with multiplicities) and $\wind(Q)=0$. 
\end{itemize}
\end{proposition}

\begin{proof}
We argue by induction on $n$.
\textbf{Base:} $n=3$. 
If $Q$ is proper and acyclic $3$-vertex quiver, then $\wind(Q)=0$.
If $Q$ is a proper oriented 3-cycle with multiplicites, 
then $\wind(Q)=1$ or $\wind(Q)=-1$ depending on the direction of traversal of the cycle. 

\textbf{Induction step.}
Suppose the claim is true for cycles of length $n-1$.
We denote $Q=(v_0 - v_1 -\cdots -v_n=v_0)$. 
Performing a sink mutation if necessary, we find a vertex $v_j$ with $v_{j-1} \rightarrow v_j \rightarrow v_{j+1}$ 
(or the same with arrows reversed).
Then the undirected graph of the quiver $\mu_{v_j}(Q)$ contains the $(n-1)$-cycle $C' = (v_{j-1} - v_{j+1} - \cdots - v_{j-2} - v_{j-1})$. 
By Lemma~\ref{lem:new cycle winding number}, $\wind(C') = \wind(Q)$.
Furthermore, $C'$ is oriented if and only if $Q$ is. 
Since $C'$ has the required winding number, so does~$Q$.  
\end{proof}

Since total properness is hereditary, we obtain the following corollary.

\begin{corollary}
\label{cor:totally-proper-wind}
In a totally proper COQ, every full subquiver $C$ whose underlying simple undirected graph 
is a chordless $n$-cycle 
has winding number $\pm 1$ (if this cycle is oriented; the sign depends on the direction of traversal) or~$0$ (otherwise).
\end{corollary}

\begin{proof}[Proof of Theorem~\ref{th:totally-proper-unique}]
The first homology of a simple graph is spanned by chordless cycles. 
The winding numbers of these cycles are uniquely determined by Corollary~\ref{cor:totally-proper-wind}.
The~claim follows by Theorem~\ref{th:wiggle/winding}. 
\end{proof}

\begin{example}
The COQ in Example~\ref{eg:proper from inherited} is not totally proper because 
the subCOQ supported by $C_5$ does not satisfy the condition in Corollary~\ref{cor:totally-proper-wind}.
(Indeed, vertex~$c$ is not proper in~$\mu_a(Q)$.) 
Cf.\ also Proposition~\ref{pr:punctured-annulus}.
\end{example}

Corollary~\ref{cor:totally-proper-wind} suggests the following algorithm based on Theorem~\ref{thm:construct an ordering}. 


\begin{algorithm}[Attempting to construct a candidate cyclic ordering]
\label{alg:sigma-Q}
\ \\
\emph{Input:} quiver $Q$ on a vertex set $V$. \\
\emph{Output:} either empty set~$\varnothing$, or a cyclic ordering $\sigma_Q$ on~$V$.

Choose a set of chordless cycles $C_i$ that form a spanning set of~$H_1(\GQ)$.

By Theorem~\ref{thm:construct an ordering}, 
there is an efficient algorithm to construct a cyclic ordering $\sigma_Q$ 
for which the cycles $C_i$ have the winding numbers prescribed by Corollary~\ref{cor:totally-proper-wind}, 
or else determine that no such cyclic ordering exists.

In the former case, output~$\sigma_Q$. 
%
In the latter case, output~$\varnothing$. 
\end{algorithm}

The above discussion implies the following characterization of Algorithm~\ref{alg:sigma-Q}. 

\begin{proposition}
\label{pr:sigma-Q}
If a quiver $Q$ has a totally proper cyclic ordering~$\sigma$, 
then running  Algorithm~\ref{alg:sigma-Q} with input~$Q$
will produce a cyclic ordering~$\sigma_Q$ wiggle equivalent to~$\sigma$. 
\end{proposition}

\begin{remark}
Algorithm~\ref{alg:sigma-Q} does not necessarily determine whether
$Q$ has a totally proper cyclic ordering or not.
If the algorithm outputs~$\varnothing$, then $Q$ definitely has no totally proper cyclic ordering.
But if the algorithm outputs an actual cyclic ordering~$\sigma_Q$, 
then all we can say is that $\sigma_Q$ is the sole possible candidate for being totally proper (up to wiggles). 
\end{remark}

\begin{example}
Applying Algorithm~\ref{alg:sigma-Q} to the quiver shown in Figure~\ref{fig:Q-plabic-hex}
will yield a cyclic ordering~$\sigma_Q$ wiggle equivalent to $\sigma=(1, 2, 3, 4, 5, 6)$, cf.\ Example~\ref{eg:proper from alg}. 
(All chordless cycles of $Q$ listed in Example~\ref{eg:E6-1-cycles} have correct winding numbers.)
In~this particular case, the COQ $(Q,\sigma_Q)$ turns out to be totally proper,
because it is mutation-acyclic, cf.\ Remark~\ref{rem:atp}. 
In fact, $Q$ is mutation equivalent to a tree, namely an orientation of the affine Dynkin diagram of type~$E_6^{(1)}$. 
\end{example}

\begin{remark}
\label{rem:ambiguous-sigma-Q}
For a general quiver~$Q$, the output of Algorithm~\ref{alg:sigma-Q} may turn out to depend---even working 
modulo wiggles---on the choice of a collection of spanning cycles.
That is, different choices may produce different output cyclic orderings;
also, some choices may lead to the empty output while others would not.
None of this ambiguity can occur if $Q$ possesses a totally proper cyclic ordering,
in light of Proposition~\ref{pr:sigma-Q}. 
\end{remark}

\hide{
In general there may be exponentially many chordless cycles. 
When there are not so many, they may be enumerated fairly quickly. 
}

\begin{remark}
\label{rem:construct proper}
One way to remove the ambiguity of Algorithm~\ref{alg:sigma-Q} discused in Remark~\ref{rem:ambiguous-sigma-Q} is to 
take the spanning set consisting of \emph{all} chordless cycles in (the underlying undirected graph of)
a given quiver~$Q$. 
This would output a \emph{canonical} cyclic ordering~$\sigma_Q$ that is independent of any choices, 
or else conclude that $Q$ has no totally proper ordering. 
Furthermore, increasing the size of the spanning set might help to arrive at the latter conclusion,  
provided the restrictions coming from some collection of cycles turn out to be incompatible. 

There will however be a price to pay: 
in general, there may be exponentially many chordless cycles (as a function of the number of vertices). 
Thus this version of Algorithm~\ref{alg:sigma-Q}, for all of its advantages, 
is not always computationally feasible.
\end{remark}

\begin{definition}
\label{def:TP-Q}
We say that a quiver $Q$ is \emph{totally proper} 
if there exists a cyclic ordering~$\sigma$ such that the COQ $(Q,\sigma)$ is totally proper. 
\end{definition}

In light of Proposition~\ref{pr:sigma-Q}, the existence quantifier in Definition~\ref{def:TP-Q} can be eliminated: 

\begin{corollary}
A quiver $Q$ is totally proper if and only if Algorithm~\ref{alg:sigma-Q} outputs a cyclic ordering~$\sigma_Q$
such that the COQ $(Q,\sigma)$ is totally proper. 
\end{corollary}

The following recipe proves useful in many applications.

\begin{definition}
\label{def:sigma ori}
Let $Q$ be a quiver such that $H_1(\GQ)$ is spanned by oriented chord\-less cycles.   
(A cycle in~$\GQ$ is \emph{oriented} if it lifts to an oriented cycle in~$Q$.) 
A~cyclic ordering~$\sigma_\circ$ of~$Q$ is called \emph{synchronous} if 
the winding number of \underbar{any} oriented chordless cycle~$C$ in~$\GQ$
is equal to~1: $\wind(C, \sigma_\circ)=1$. 
(Here we choose the traversal order of~$C$ that is consistent with the orientations of the corresponding arrows in~$Q$.) 
\end{definition}

\begin{corollary}
\label{cor:}
If $H_1(\GQ)$ is spanned by oriented chordless cycles, then: 
\begin{itemize}[leftmargin=.2in]
\item 
all synchronous orderings (if they exist) are wiggle equivalent to each other; 
\item
we can construct a synchronous ordering (or determine that none exists) using 
the corresponding version of Algorithm~\ref{alg:sigma-Q}; 
\item
if $(Q,\sigma)$ is totally proper for some cyclic ordering~$\sigma$, then $\sigma$~must be synchronous. 
\end{itemize}
\end{corollary}

\begin{proof}
These statements follow from Theorem~\ref{th:wiggle/winding} and Proposition~\ref{pr:sigma-Q}. 
\end{proof}

\begin{example}
\label{eg:type D totally proper}
Consider the chordless oriented cycle quiver $Q$ of type $D_n$, with vertices and edges 
$v_1 \rightarrow v_2 \rightarrow \cdots \rightarrow v_n \rightarrow v_1$.
There is only one chordless cycle in~$\GQ$, and this cycle is oriented. 
We thus obtain the synchronous cyclic ordering $\sigma_{\circ}= (v_1, \ldots, v_n)$. 
(This cyclic ordering is in fact totally proper, see Theorem~\ref{thm:finite type}.)
%
\end{example}

\begin{example}
Consider the quiver $Q$ of type $A_6$ shown below: 
\begin{equation*}
\begin{tikzcd}[arrows={-stealth}, sep=1.5em]
  v_6  \arrow[rd]   &&& v_1 \arrow[d] \\
v_5 \arrow[u] & v_4 \arrow[l] & v_3 \arrow[ru] \arrow[l] & v_2 \arrow[l] 
\end{tikzcd}
\end{equation*}
Here  $H_1(\GQ)$ is generated by oriented chordless cycles 
$C_1 = (v_1\rightarrow v_2 \rightarrow v_3 \rightarrow v_1)$ and 
$C_2 = (v_4 \rightarrow v_5 \rightarrow v_6 \rightarrow v_4)$.
Hence $\sigma_{\circ}= (v_1, \ldots, v_6)$ is the unique, up to wiggle equivalence, synchronous cyclic ordering of~$Q$.
(It is totally proper by Theorem~\ref{thm:finite type}.)
\end{example}


\begin{example}
Both quivers shown in Figure~\ref{fig:square and triangular grid quivs} 
satisfy the assumption in Definition~\ref{def:sigma ori}:
in each case, $H_1(\GQ)$ is spanned by the boundaries of the bounded faces,
and each of those boundary cycles is oriented. 
Both quivers have synchronous cyclic orderings~$\sigma_{\circ}$,
which can be obtained from coloring the vertices in 4 and 3~colors, respectively, cf.\ Example~\ref{eg:3-colors}. 

The $4\times 4$ grid quiver in Figure~\ref{fig:square and triangular grid quivs} is not totally proper, 
as the acyclic chordless 12-cycle traversing its perimeter has nonzero winding number for~$\sigma_{\circ}$. 
The 10-vertex triangular grid quiver is totally proper, which can be shown by exhaustive search.
(This quiver has a finite mutation class.) 
\end{example}

\begin{figure}[ht]
\begin{center}
\vspace{5pt}
{
\begin{tikzpicture}[scale=0.8]
\foreach \x in {0,2} {
	\foreach \y in {0,2} {
		\filldraw (\x, \y) circle (2pt); 
		\def\test{3}
		\ifx\y\test
		\else\filldraw[black,-{stealth}, shorten >=5pt, shorten <= 5pt] (\x, \y+1) -- (\x, \y)
		\fi;
		\ifx\x\test
		\else\filldraw[black,-{stealth}, shorten >=5pt, shorten <= 5pt] (\x, \y) -- (\x+1, \y)
		\fi;
	}
}
\foreach \x in {1,3} {
	\foreach \y in {0,2} {
		\filldraw (\x, \y) circle (2pt); 
		\def\test{3}
		\ifx\y\test
		\else\filldraw[black,-{stealth}, shorten >=5pt, shorten <= 5pt] (\x, \y) -- (\x, \y+1)
		\fi;
		\ifx\x\test
		\else\filldraw[black,-{stealth}, shorten >=5pt, shorten <= 5pt] (\x+1, \y) -- (\x, \y)
		\fi;
	}
}
\foreach \x in {0,2} {
	\foreach \y in {1,3} {
		\filldraw (\x, \y) circle (2pt); 
		\def\test{3}
		\ifx\y\test
		\else\filldraw[black,-{stealth}, shorten >=5pt, shorten <= 5pt] (\x, \y) -- (\x, \y+1)
		\fi;
		\ifx\x\test
		\else\filldraw[black,-{stealth}, shorten >=5pt, shorten <= 5pt] (\x+1, \y) -- (\x, \y)
		\fi;
	}
}
\foreach \x in {1,3} {
	\foreach \y in {1,3} {
		\filldraw (\x, \y) circle (2pt); 
		\def\test{3}
		\ifx\y\test
		\else\filldraw[black,-{stealth}, shorten >=5pt, shorten <= 5pt] (\x, \y+1) -- (\x, \y)
		\fi;
		\ifx\x\test
		\else\filldraw[black,-{stealth}, shorten >=5pt, shorten <= 5pt] (\x, \y) -- (\x+1, \y)
		\fi;
	}
}


\end{tikzpicture}
}
\qquad
\qquad
\quad
{
\newcommand\height{0.36cm}
\begin{tikzpicture}
\def\x{0}
\foreach \y in {0,...,3} {
	\filldraw ({(2*\x+\y)*\height}, 2*\y*\height) circle (1.6pt); 
	\def\test{3}
	\ifx\y\test
	\else\filldraw[black,-{stealth}, shorten >=5pt, shorten <= 5pt] ({(2*\x+\y)*\height}, 2*\y*\height) -- ({(2*\x+\y+1)*\height}, {(2*\y+2)*\height})
	\fi;
	\ifx\y\test
	\else\filldraw[black,-{stealth}, shorten >=5pt, shorten <= 5pt]  ({(2*\x+\y+2)*\height}, {(2*\y)*\height}) -- ({(2*\x+\y)*\height}, 2*\y*\height)
	\fi;
	\ifx\y\test
	\else\filldraw[black,-{stealth}, shorten >=5pt, shorten <= 5pt]  ({(2*\x+\y+1)*\height}, {(2*\y+2)*\height}) -- ({(2*\x+\y+2)*\height}, {(2*\y)*\height})
	\fi;
}
\def\x{1}
\foreach \y in {0,...,2} {
	\filldraw ({(2*\x+\y)*\height}, 2*\y*\height) circle (1.6pt); 
	\def\test{2}
	\ifx\y\test
	\else\filldraw[black,-{stealth}, shorten >=5pt, shorten <= 5pt] ({(2*\x+\y)*\height}, 2*\y*\height) -- ({(2*\x+\y+1)*\height}, {(2*\y+2)*\height})
	\fi;
	\ifx\y\test
	\else\filldraw[black,-{stealth}, shorten >=5pt, shorten <= 5pt]  ({(2*\x+\y+2)*\height}, {(2*\y)*\height}) -- ({(2*\x+\y)*\height}, 2*\y*\height)
	\fi;
	\ifx\y\test
	\else\filldraw[black,-{stealth}, shorten >=5pt, shorten <= 5pt]  ({(2*\x+\y+1)*\height}, {(2*\y+2)*\height}) -- ({(2*\x+\y+2)*\height}, {(2*\y)*\height})
	\fi;
}
\def\x{2}
\foreach \y in {0,1} {
	\filldraw ({(2*\x+\y)*\height}, 2*\y*\height) circle (1.6pt); 
	\def\test{1}
	\ifx\y\test
	\else\filldraw[black,-{stealth}, shorten >=5pt, shorten <= 5pt] ({(2*\x+\y)*\height}, 2*\y*\height) -- ({(2*\x+\y+1)*\height}, {(2*\y+2)*\height})
	\fi;
	\ifx\y\test
	\else\filldraw[black,-{stealth}, shorten >=5pt, shorten <= 5pt]  ({(2*\x+\y+2)*\height}, {(2*\y)*\height}) -- ({(2*\x+\y)*\height}, 2*\y*\height)
	\fi;
	\ifx\y\test
	\else\filldraw[black,-{stealth}, shorten >=5pt, shorten <= 5pt]  ({(2*\x+\y+1)*\height}, {(2*\y+2)*\height}) -- ({(2*\x+\y+2)*\height}, {(2*\y)*\height})
	\fi;
}

\filldraw ({(2*3)*\height}, 0) circle (1.6pt);


\end{tikzpicture}
}
\vspace{2pt}
\end{center}
\caption{A square grid quiver and a triangular grid quiver.}
\vspace{-10pt}
\label{fig:square and triangular grid quivs}
\end{figure}



\begin{remark}
\label{rem:plabics}
Many---though not all---quivers associated with plabic graphs satisfy the assumptions of Definition~\ref{def:sigma ori}.
\end{remark}

\enlargethispage{5pt}

\begin{remark}
The reader may wonder whether it is possible to introduce powerful mutation invariants defined for \emph{all} quivers,
rather than just the totally proper ones. 
As~shown by G.~Chelnokov~\cite{ChelnokovNoInvs}, 
any mutation-invariant
piecewise-polynomial function in the entries of a $4\times 4$ exchange matrix~$B$ 
must be a function of~$\det(B)$. 
Thus any new piecewise-polynomial mutation invariants, such as the coefficients of the Alexander polynomial, 
must involve some additional input beyond~$B$. 
\end{remark}

\newpage

\section{Totally proper COQs: examples}
\label{sec:totally proper examples}

The following result is immediate from Lemma~\ref{lem:n=3-tot-proper}.

\begin{proposition}
\label{prop:3-vert proper}
Any proper $3$-vertex COQ is totally proper.
\end{proposition}

\begin{example}
Continuing with Example~\ref{eg:D4}, 
let $Q$ be an oriented $4$-cycle quiver 
\begin{equation*}
\begin{tikzcd}[arrows={-stealth}, sep=2em]
  a  \arrow[r] & b \arrow[d]  \\
 d\arrow[u]  & c   \arrow[l] 
\end{tikzcd}
\end{equation*}
of type~$D_4$. 
Up to wiggle equivalence, $Q$ possesses three cyclic orderings 
$\sigma_1=(a, b, c, d)$, $\sigma_2=(a,b,d,c)$, and $\sigma_3=(a, d, c, b)$,
with respective winding numbers $1$, $2$, and~$3$. 
As explained in Example~\ref{eg:D4}, 
the COQ $(Q,\sigma_1)$ is totally~proper. 
The COQs $(Q,\sigma_2)$ and $(Q,\sigma_3)$ are not.
Furthermore, some quivers mutation-equivalent to~$Q$ do not appear in the corresponding proper mutation classes,
see Figure~\ref{fig:mu class D4}. 
\end{example}

\begin{theorem}
\label{thm:finite type}
Any quiver of finite type is totally proper, cf.\ Definition~\ref{def:TP-Q}. 
\end{theorem}

\begin{proof} 

For quivers of type~$E_6, E_7, E_8$, the claim can be verified by a computer check. 

Let $Q$ be a quiver of type $A_n$. 
Any such quiver can be colored in $3$ colors, say $1,2,3$, so that every arrow is oriented 
in one of the three ways: $1 \rightarrow 2$, $2 \rightarrow 3$, or $3 \rightarrow 1$.
Fix a particular such $3$-coloring. 
Choose a linear ordering of the vertices so that for any vertices $a,b,c$ of colors $1,2,3$ respectively, we have $a < b < c$.
All such choices of linear ordering are wiggle equivalent, as vertices of the same color are pairwise non-adjacent.
Further, any such choice of coloring is entirely determined by the color of one vertex, 
so all these choices give linear orders that are cyclic shifts of each other, modulo wiggles.
Let $\sigma$ be a cyclic ordering compatible with these linear orderings;
it is defined uniquely up to wiggles. 
The resulting COQ $(Q, \sigma)$ is proper, cf.\ Example~\ref{eg:3-colors}.

Pick a vertex $v$ to mutate~at. The above construction can be applied to the quiver $Q'=\mu_v(Q)$,
resulting in a proper COQ $(Q',\sigma')$. 
Alternatively, we can mutate the COQ $(Q, \sigma)$ at~$v$, yielding a COQ $\mu_v(Q,\sigma)=(Q',\sigma'')$. 
It remains to show that in fact, the cyclic orderings $\sigma'$ and $\sigma''$ are wiggle equivalent to each other. 
In light of Theorem~\ref{th:wiggle/winding} and Remark~\ref{rem:homotopy gen invariant}, 
this can be established by proving that $\sigma'$ and $\sigma''$ have the same winding numbers
for some basis of the first homology of the underlying undirected graph~$\GQ'$. 
The standard description of quivers of type~$A_n$ (see~\cite{ca2})
provides a construction of such basis consisting of (oriented) triangles inscribed in the triangles
of the underlying triangulation of an $(n+3)$-gon. 
Since the COQ $(Q',\sigma')$ is proper, all these triangles have winding numbers equal to~1. 
To complete the proof, we need to show that the same is true for $(Q',\sigma'')$. 

Every triangle in $(Q',\sigma'')$ that is entirely disjoint from $v$ has the same induced cyclic ordering in both $Q$ and~$Q'$.
Further, $v$ is proper inside any oriented triangle of $Q'$ involving~$v$ (viewed as a 3-vertex subCOQ). 
But if an oriented $3$-cycle quiver has one proper vertex, then all three of its vertices are proper.
Hence every oriented triangle of $Q'$ has winding number~$1$.

In type~$D_n$, the argument is similar but more complicated. We omit~it. 
\end{proof}

\begin{remark}
All chordless cycles in a quiver of finite type are oriented,
see \cite[Proposition~9.7]{ca2} or \cite[Theorem~1.2]{BGZ}. 
Moreover, they span (hence contain a basis for) their respective first homology groups.
Thus for these quivers, $\sigma_Q = \sigma_{\circ}$.
\end{remark}

\begin{definition}[{\cite[Definition~2.1]{Warkentin}, cf.\ \cite[Section~6]{LMC}}]
\label{def:fork}
A quiver $Q$ is called a~\emph{fork}~if 
\begin{itemize}[leftmargin=.2in]
\item 
$Q$ is abundant, i.e., $|b_{ij}| \geq 2$ for all $i\neq j$; 
\item
$Q$ has a distinguished vertex $r$ (the \emph{point of return}) such that whenever $b_{ir}, b_{rj} > 0$, we have 
$b_{ji} > \max(b_{ir}, b_{rj})$; 
\item
the full subquiver of $Q$ obtained by removing vertex $r$ is acyclic.
\end{itemize}
\end{definition}

\begin{proposition}
\label{prop: fork vf}
Any fork quiver has a unique proper cyclic ordering. 
\end{proposition}

\begin{proof}
It follows from \cite[Section~6]{LMC}, or by a straighforward case analysis, 
that any fork quiver is vortex-free (and complete).
The claim follows by Proposition~\ref{prop:vf complete proper }. 
\end{proof}

The following key result is immediate from Proposition~\ref{pr:propagation-under-VF}. 

\begin{theorem}
\label{th:completeVF}
Suppose that every quiver mutation-equivalent to a quiver $Q$ is complete and vortex-free. 
If $Q$ has a proper cyclic orientation, then it is totally~proper.
\end{theorem}

There are many examples of quivers to which Theorem~\ref{th:completeVF} applies.

\begin{corollary}
\label{cor:abundant-acyclic-proper}
Let $Q$ be an abundant acyclic quiver, with a cyclic ordering~$\sigma$ 
constructed as described in Observation~\ref{ob:acyclic-proper}.
Then  $(Q,\sigma)$ is a totally proper COQ. 
\end{corollary}

\begin{proof}
A.~Seven \cite[p.~473]{SevenAMatrices} shows that a certain vortex quiver is not mutation-equi\-valent to an acyclic quiver. 
The same argument establishes that any quiver mutation-equivalent to an acyclic quiver 
(such as~$Q$) is vortex-free.
By \cite[Section~6]{LMC}, any quiver in the mutation class of~$Q$ is abundant.
The claim then follows by Theorem~\ref{th:completeVF}. 
\end{proof}

\begin{conjecture}
Any acyclic quiver, endowed with the standard cyclic ordering (cf.\ Observation~\ref{ob:acyclic-proper}), 
is totally proper. 
\end{conjecture}

\begin{remark}
\label{rem:atp}
This conjecture has been recently proved by the second author~\cite{SN-acyclic-TP}. 
\end{remark}

\begin{remark}
G.~Muller's \emph{local acyclicity} property~\cite{Muller-loc-acyclic}
does not guarantee the existence of a totally proper cyclic ordering. 
A~quiver~$Q$ of the kind described in \cite[Figure~1]{LMC} is a vortex, 
so it has no totally proper cyclic ordering. 
On the other hand, one can use the Banff algorithm \cite[Theorem~5.5]{Muller-loc-acyclic} 
to show that $Q$ is locally acyclic.  
\end{remark}

\begin{corollary}
\label{cor:minimal cyclic totally proper eg}
Let $Q$ be a complete quiver such that for every vertex~$k$, the quiver $\mu_k(Q)$ is a fork with point of return~$k$,
cf.\ Definition~\ref{def:fork}. 
(In the language of \cite[Section~6]{LMC}, any mutation of $Q$ is an \emph{exit}.)
Then $Q$ is totally proper.
\end{corollary}

\begin{proof}
By \cite[Lemma~2.5]{Warkentin}, a mutation $\mu_j$ of a fork yields another fork, provided that $j$ is not the point of return.
It follows that every quiver in the mutation class of $Q$ is complete and vortex-free. 
($Q$~itself cannot contain a vortex, since mutating at its apex would yield a quiver with a vortex, hence not a fork.)
By Proposition~\ref{prop:vf complete proper }, $Q$ has a proper cyclic ordering. 
By Theorem~\ref{th:completeVF} it is totally proper.
\end{proof}

\begin{example}
\label{eg:minimal cyclic}
Consider the family of $5$-vertex COQs shown in Figure~\ref{fig:minimal cyclic}.
We will use Corollary~\ref{cor:minimal cyclic totally proper eg} to show that 
every COQ $Q$ in this family is totally proper. 
It is clear that $Q$ is complete and proper, hence vortex-free.
By \cite[Proposition~6.13]{LMC}, it suffices to check, for each vertex $v_i$:
\begin{itemize}
\item if $j \rightarrow v_i \rightarrow k$ for some pair of vertices $j,k$, then $|b_{jk}| < |b'_{jk}|$,
where we use the notation $B_Q = (b_{ij})$ and $B_{\mu_{v_i}(Q)} = (b'_{ij})$. 
Equivalently, $v_i$ is an ``ascent'' \cite[Definition~3.3]{LMC} in every oriented $3$-cycle that contains~$v_i\,$; 
\item $v_i$ is not a sink/source in $Q$;
\item $v_i$ is not the apex of a vortex in $Q$.
\end{itemize}
The second and third conditions are trivial, as $Q$ has no sinks, sources, or vortices.
We check the first condition for $v_2$; the arguments for all other vertices are similar.
All paths through $v_2$ contain $v_1 \rightarrow v_2$ and one of $v_2 \rightarrow v_3$, $v_2 \rightarrow v_4$, or $v_2 \rightarrow v_5$.
Mutation at $v_2$ increases the number of arrows $v_1\rightarrow v_3$ and $v_1\rightarrow v_4$.
Finally, the number of arrows between $v_1$ and $v_5$ in $\mu_{v_2}(Q)$ is equal to 
\begin{align*}
|2j - (a+f) (aj+f)| &= |2j - (a^2 j + af + ajf + f^2)| = (a^2 - 2) j + af(j+1) + f^2,
\end{align*}
which is larger than $2j$, the number of arrows $v_1\rightarrow v_5$ in~$Q$.
\end{example}

\begin{figure}[ht]
\vspace{3pt}
{
\newcommand{\radius}{2.8cm} 
\newcommand{\smoltxt}[1]{{\scriptstyle #1}}
\begin{tikzpicture}
\filldraw[black] (0,0)++(135:\radius) circle (2pt) node[above left=-1pt] {$v_1$} coordinate (v1);
\filldraw[black] (0,0)++(135-72:\radius) circle (2pt) node[above right=-1pt] {$v_2$} coordinate (v2);
\filldraw[black] (0,0)++(135-144:\radius) circle (2pt) node[below right=-1pt] {$v_3$} coordinate (v3);
\filldraw[black] (0,0)++(135-216:\radius) circle (2pt) node[below left=-1pt] {$v_4$} coordinate (v4);
\filldraw[black] (0,0)++(135-288:\radius) circle (2pt) node[below left=-1pt] {$v_5$} coordinate (v5);
\draw[black, dashed, decoration={markings, mark=at position .1 with {\arrow{<}}}, postaction={decorate}] (0,0) circle (\radius);
\draw[black, -{stealth}, shorten >=3pt, shorten <= 3pt] (v1) -- node[above, pos=0.5] {$\smoltxt{a+f}$} (v2); 
\draw[black, -{stealth}, shorten >=3pt, shorten <= 3pt] (v1) -- node[above right, pos=0.2] {$\smoltxt{b+h}$} (v3); 
\draw[black, -{stealth}, shorten >=3pt, shorten <= 3pt] (v1) -- node[right, pos=0.5] {$\smoltxt{c+i}$}(v4); 

\draw[black, -{stealth}, shorten >=3pt, shorten <= 3pt] (v2) -- node[right, pos=0.5] {$\smoltxt{d}$} (v3); 
\draw[black, -{stealth}, shorten >=3pt, shorten <= 3pt] (v2) -- node[right, pos=0.5] {$\smoltxt{e}$} (v4); 
\draw[black, -{stealth}, shorten >=3pt, shorten <= 3pt] (v2) -- node[right, pos=0.55] {$\smoltxt{aj+f}$} (v5); 

\draw[black, -{stealth}, shorten >=3pt, shorten <= 3pt] (v3) -- node[right, pos=0.5] {$\smoltxt{g}$} (v4); 
\draw[black, -{stealth}, shorten >=3pt, shorten <= 3pt] (v3) -- node[below, pos=0.25] {$\smoltxt{bj+h}$} (v5); %

\draw[black, -{stealth}, shorten >=3pt, shorten <= 3pt] (v4) -- node[right, pos=0.6] {$\smoltxt{cj+i}$} (v5); %

\draw[black, -{stealth}, shorten >=3pt, shorten <= 3pt] (v5) -- node[right, pos=0.5] {$\smoltxt{2j}$} (v1); 
\end{tikzpicture}
} 
\vspace{-5pt}
\caption{For any values of $a,b,c,d,e,f,g,h,i,j \geq 4$, this COQ is totally proper.}
\label{fig:minimal cyclic}
\end{figure}

\vspace{-10pt}

\begin{remark} 
\label{rem:forkless-part eg}
The \emph{forkless part} of a mutation class of quivers, 
introduced by M.~Warkentin~\cite{Warkentin},  is the set of quivers in the mutation class which are not forks. 
Each of the quivers discussed in Example~\ref{eg:minimal cyclic} is the unique quiver in the forkless part 
of its mutation class. 
In order to show that a given COQ is totally proper, 
it is sufficient to verify that the forkless part of its mutation class is complete and vortex-free, 
cf.\ Proposition~\ref{prop: fork vf} and Theorem~\ref{th:completeVF}.
Such a verification is particularly straightforward if the forkless part is finite.
For example, the fully generic mutation cycles constructed in \cite[Examples 10.1--10.2]{LMC} 
have finite forkless parts all of whose quivers are complete and vortex-free.
Thus all these quivers allow (unique) totally proper cyclic orderings.
Also, many of the mutation classes considered in \cite{MR2805200}
have finite forkless parts consisting of complete vortex-free quivers;
whenever this happens, total properness follows. 
Similar arguments can be applied to quivers with a finite ``pre-forkless part'' studied by T.~Ervin~\cite{Ervin}. 
\end{remark}

There are many more families of quivers with finite forkless part to which Theorem~\ref{th:completeVF} applies. 
The examples discussed above in Corollaries~\ref{cor:abundant-acyclic-proper}--\ref{cor:minimal cyclic totally proper eg}
and in Remark~\ref{rem:forkless-part eg} are just a small selection. 


\begin{remark}
\label{rem:random-walk-in-exchange-graph}
Recall that a quiver has \emph{finite} (resp., \emph{infinite}) \emph{mutation type}
if its mutation equivalence class is finite (resp., infinite). 
%
%
%
Let $Q$ be a connected quiver of infinite mutation type. 
As shown by M.~Warkentin~\cite[Proposition~5.2]{Warkentin}, 
a simple random walk in the exchange graph of~$Q$ 
will almost surely leave the forkless part~and never come back.
Therefore a random sequence of mutations starting at~$Q$
will almost surely reach a quiver that can be upgraded to a proper COQ in such a way that 
all subsequent mutations will become proper in the resulting (proper) COQs. 
\hide{
Let 
\begin{equation*}
Q=Q_0 \stackrel{\mu_{v_1}}{\longrightarrow} Q_1 \stackrel{\mu_{v_2}}{\longrightarrow} Q_2
\stackrel{\mu_{v_3}}{\longrightarrow}  \cdots 
\end{equation*}
be a simple random walk in the exchange graph of~$Q$.
As shown by M.~Warkentin~\cite[Proposition~5.2]{Warkentin}, 
with probability~1 there exists $N>0$ such that every quiver $Q_n$ with $n>N$ is a fork. 
In particular, every $Q_n$ with $n>N$ 
has a unique proper cyclic ordering---so all mutations $\mu_{v_n}$ for $n>N$ are proper. 
}
\end{remark}


\subsection*{Testing mutation equivalence of totally proper quivers}

\begin{remark}
\label{rem:conjugacy-vs-mut-equiv}
As discussed in Remark~\ref{rem:algorithms-conjugacy}, the conjugacy problem in $\GL(n,\mathbb{Z})$~has an algorithmic solution with a workable implementation. 
This gives a testable necessary condition for mutation equivalence 
of two totally proper COQs $Q_1$ and~$Q_2$:  
if the cosquares of their respective unipotent companions $U_1$ and~$U_2$ 
are not conjugate in $\GL(n,\mathbb{Z})$, 
then the quivers $Q_1$ and $Q_2$ are not mutation-equivalent. 

In our experience, this test rarely produces ``false positives:'' if $Q_1$ and~$Q_2$ 
are totally proper COQs such that the cosquares of their unipotent companions $U_1$ and~$U_2$
are conjugate  in $\GL(n,\mathbb{Z})$,
then $Q_1$ and~$Q_2$ are likely mutation equivalent. 
Counter\-examples to this phenomenon are not common, but they do exist. 
As mentioned in Example~\ref{eg:9-vertex-trees}, 
there are three pairs of nonisomorphic (hence mutation-inequivalent)
9-vertex trees which give rise to cosquare matrices that are conjugate in $\GL(n,\mathbb{Z})$. 

For 3-vertex quivers, we expect no ``collisions'' of this kind, cf.\ Conjecture~\ref{cor:sim trans vs opposites}. 

For 4-vertex quivers, such collisions exist, but seem to be rare.  
The simplest one that we found involves two quivers $Q_1$ and $Q_2$ with the exchange matrices
\begin{equation*}
B_{Q_1} = 
\begin{bmatrix}
0  & -2 & -18&   21 \\
2  &  0 & -13 &  -9 \\
18  & 13 &   0 &  -6 \\
-21 &   9  &  6  &  0
\end{bmatrix}
, \quad
B_{Q_2}=
\begin{bmatrix}
  0  & -2  & -9  & 23 \\
 2& 0 & -15 & -10 \\
  9  & 15 &   0 &  -6 \\
-23 &  10 &   6  &  0
\end{bmatrix}
\end{equation*}
We first verify that $Q_1$ and $Q_2$ are mutation-inequivalent. 
Both mutation classes have finite forkless part (cf.\ Remark~\ref{rem:forkless-part eg})
where every quiver is abundant \hbox{and vortex-free.} 
Hence both $Q_1$ and $Q_2$ have totally proper cyclic orderings. 
Neither is a fork, so~it \linebreak[3] 
is enough to check that $Q_1$ is not in the forkless part of $Q_2$. 
The only fork-avoiding mutation sequence for $Q_2$ is $v_2, v_4, v_1$ (applied left-to-right);  
it does not produce~$Q_1$.

At the same time, the cosquares of $Q_1$ and $Q_2$, given by 
\begin{equation*}
U_{Q_1}^{-T}U_{Q_1}=
\begin{bmatrix}
1  &  -2  &  -18  &   21 \\
2   &  -3  &  -49  &   33 \\
44 &   -75 &  -960 &   801 \\
261 &  -435 & -5823 &  4663
\end{bmatrix}
\!\!, \quad
U_{Q_2}^{-T}U_{Q_2}=
\begin{bmatrix}
 1  &   -2  &   -9  &   23 \\
  2  &   -3 &   -33  &   36 \\
 39 &   -63 &  -575  &  741 \\
 231 &  -362 & -3573  & 4278
\end{bmatrix}
\!\!, 
\end{equation*}
turn out to be conjugate in $\GL(4,\ZZ)$. 
The \textsc{Magma} algorithm~\cite{EHO} certifies this fact
by delivering a conjugating matrix whose entries have thousands of decimal digits~(!).
\end{remark}

\begin{problem}
\label{rem:conj-class-for-tp-determines-mut-class}
As explained in Remark~\ref{rem:conjugacy-vs-mut-equiv},
$\GL(n,\mathbb{Z})$ conjugacy of cosquares does not guarantee mutation equivalence,
even for totally proper COQs. 
Does the potentially stronger assumption of 
integral congruence of unipotent companions
imply mutation equivalence?
Put differently, is the integral congruence class of a unipotent companion 
a complete mutation invariant (say, in the case of totally proper COQs)?
\end{problem}


\newpage

\section{Admissible quasi-Cartan companions}
\label{sec:admissible cartan}


\begin{definition}[M.~Barot--C.~Geiss--A.~Zelevinsky \cite{BGZ}]
Let $Q$ be an $n$-vertex quiver.
A \emph{quasi-Cartan companion} for~$Q$ (or for the corresponding exchange matrix~$B=B_Q$)
is an $n \times n$ symmetric matrix $A = (a_{ij})$ 
with $|a_{ij}| = |b_{ij}|$ for $i \neq j$ and $a_{ii} = 2$ for all~$i$.
We note that a typical quiver has many distinct  quasi-Cartan companions. 
\end{definition}

\begin{remark}
Given an $n\times n$ integer matrix $U=(u_{ij})$ such that $U - U^T = -B$ and $u_{ii}=1$ for all~$i$, 
we can construct a quasi-Cartan companion~$A$ of~$B=B_Q$ by setting $A=A_U = U+U^T$.
In particular, for any linear ordering of the vertices~of~$Q$, the corresponding unipotent companion~$U$ 
(cf.\ Definition~\ref{def:unipotent companions}) gives rise to the quasi-Cartan companion~$A_U= U+U^T$. 
\end{remark}


\begin{definition}[A.~Seven \cite{SevenAMatrices}]
\label{def:admissible quasiCartan}
A quasi-Cartan companion $A=(a_{ij})$ for a quiver~$Q$ is \emph{ad\-missible} if 
every full subquiver whose underlying undirected simple graph (cf.\ Definition~\ref{def:Graph-Q}) 
is a chordless~cycle 
\begin{equation}
\label{eq:cordles--cycle-k}
C = (v_0 - v_1 - \cdots - v_k=v_0)
\end{equation}
satisfies the following conditions: 
\begin{itemize}
\item if $C$ is oriented (that is, either $v_i \rightarrow v_{i+1}$ for all $i<k$ or $v_i \leftarrow v_{i+1}$ for all $i<k$), 
then the count $\#\{i\mid 0 \leq i < k \text{ and } a_{v_i v_{i+1}}>0\}$ is odd; 
\item if $C$ is not oriented, then this count is even. 
\end{itemize}
\end{definition}



\begin{proposition}[\cite{SevenAMatrices}, Lemma 3.3]
\label{prop:Seven acyclic catchy}
Let $Q$ be an acyclic quiver.
We can choose an admissible quasi-Cartan companion for every quiver in the mutation class of~$Q$, 
so that all these quasi-Cartan companions are pairwise congruent over~$\ZZ$.
\end{proposition}

\begin{proposition}
\label{pr:admissible-qC-vs-U}
If $U$ is a unipotent companion of a totally proper COQ $Q$, then $A_U$ is an admissible quasi-Cartan companion of $Q$.
\end{proposition}

\begin{proof}
Let $C$ be a chordless oriented (resp., non-oriented) cycle in~$Q$
of the form shown in~\eqref{eq:cordles--cycle-k}. 
The winding number of~$C$ is given by 
\begin{equation}
\label{eq:wind formula 2}
\wind(C) = \# \{v_i \rightarrow v_{i+1} | v_i > v_{i+1} \} - \# \{ v_i \leftarrow v_{i+1} | v_i < v_{i+1} \}. 
\end{equation}

Let $<$ be the linear order associated to $U$. 
For $p\ne q$, the entry $u_{pq}$ of~$U$ is positive if and only if $p\leftarrow q$ and $p< q$. 
(Here $p\leftarrow q$ means that $Q$ contains an arrow $p\leftarrow q$.) 
This gives a criterion for positivity of an entry $a_{pq}$ of~$A_U$: 
\begin{equation*}
a_{pq}>0 \Leftrightarrow (u_{pq}>0\text{ or }u_{qp}>0) \Leftrightarrow 
\text{(($p\leftarrow q$ and $p< q$) or ($p\rightarrow q$ and $p>q$)).}
\end{equation*}
It follows that
\begin{equation*}
\#\{ i | a_{v_i v_{i+1}} > 0 \} 
= \#\{v_i \rightarrow v_{i+1} | v_i > v_{i+1} \} + \#\{ v_i \leftarrow v_{i+1} | v_i < v_{i+1} \}, 
\end{equation*}
which together with \eqref{eq:wind formula 2} implies that 
\begin{equation*}
\#\{ i | a_{v_i v_{i+1}} > 0 \} \equiv \wind(C)  \bmod 2. 
\end{equation*}
Since $Q$ is totally proper, Corollary~\ref{cor:totally-proper-wind}
implies that $\wind(C)=\pm1$ or $\wind(C)=0$ depending on whether $C$ is oriented or not. 
It follows that $\#\{ i | a_{v_i v_{i+1}} > 0 \}$ is odd (resp., even) for oriented (resp., non-oriented) chordless cycles.
In other words, $A_U$ is admissible.
\end{proof}

\begin{remark}
The na\"ive converse of Proposition~\ref{pr:admissible-qC-vs-U} is false:
there exist admissible quasi-Cartan companions that do not come from a unipotent companion. 
On the other hand, by \cite[Theorem~2.11]{SevenOtherPaper}, 
all admissible quasi-Cartan companions of a given quiver are related to each other
by simultaneously changing the signs in a subset of rows and columns.
\end{remark}

The above discussion of admissible quasi-Cartan companions, taken together with
Corollary~\ref{cor:totally-proper-wind}, suggests the following notion. 

\begin{definition}
\label{def:admissible-quiver}
Let $Q$ be a quiver. Recall that $\GQ$ denotes the unoriented simple graph associated to $Q$,
cf.\ Definition~\ref{def:Graph-Q}.
A~homo\-morphism 
\begin{equation*}
\varphi:H_1(\GQ)\to \ZZ/2\ZZ
\end{equation*}
(cf.\ Remark~\ref{rem:homotopy gen invariant}) 
is  \emph{admissible} 
if for every chordless cycle~$C$ in~$\GQ$, we~have
\begin{equation*}
\varphi(C)=\begin{cases}
1 & \text{if $C$ lifts to an oriented cycle in $Q$;} \\
0 & \text{otherwise.}
\end{cases}
\end{equation*}
\end{definition}

Corollary~\ref{cor:totally-proper-wind} directly implies the following necessary condition
that every quiver that possesses a totally proper cyclic ordering must satisfy. 

\begin{corollary}
Let $Q$ be a totally proper quiver.
Then every quiver mutation equivalent to $Q$  
allows an admissible homo\-morphism $\varphi:H_1(\GQ)\to \ZZ/2\ZZ$. 
\end{corollary}

The above corollary, in turn, implies the following necessary condition for the existence of a totally proper cyclic ordering. 

\begin{corollary}
\label{cor:even-paving}
Let $Q$ be a totally proper quiver.
Suppose that a collection of chordless cycles in $\GQ$ 
covers every edge of~$\GQ$ an even number of times. 
Then this collection must contain an even number of cycles whose lifts are oriented in~$Q$. 
\end{corollary}

In the rest of this section, we use Corollary~\ref{cor:even-paving} 
to provide an example of a quiver of finite mutation type that is not totally proper. 

We briefly recall the classification of quivers of finite mutation type that was given by A.~Felikson, M.~Shapiro, and P.~Tumarkin~\cite{FeShTu},
following earlier work in~\cite{cats1, Derksen-Owen}. 
They showed that, apart from 11 exceptional mutation classes (cf.\ Proposition~\ref{pr:exceptional-11} below), 
all quivers of finite mutation type come from triangulated surfaces with boundary. 
We refer the reader to \cite{cats1} for the description of the latter construction. 

Combinatorics of quivers of finite mutation types is well understood~\cite{FeShTu, cats1}.   
In particular, algorithms exist that determine which of these quivers are mutation-equivalent~\cite{WeiwenGu}. 
So in this context, mutation invariants have less practical utility.
On the other hand, understanding which mutation-finite COQs are totally proper 
may provide useful insights into the study of total properness for general quivers.

\begin{proposition}
\label{pr:punctured-annulus}
Any quiver arising from a triangulation of a once-punctured annulus is not totally proper. 
\end{proposition}

\begin{proof}
It suffices to treat the case where each of the two boundary components
of the annulus contains a single marked point. The general case will follow by restriction to a full subquiver. 

It will be sufficient to establish the claim for a single quiver
in the given mutation class. We will use the quiver shown in 
Figure~\ref{fig:Q-punctured-annulus}. 
%
The collection of 5 chordless cycles 
$(a \rightarrow b \rightarrow d\rightarrow a)$ (oriented),
$(b \rightarrow d \rightarrow e\rightarrow b)$ (oriented),
$(b \rightarrow c \rightarrow e\rightarrow b)$ (oriented),
$(a \rightarrow b \rightarrow c \leftarrow a)$ (unoriented), 
$(a \rightarrow c \rightarrow e \leftarrow d \rightarrow a)$ (unoriented)
covers every arrow in~$Q$ exactly twice. 
The claim then follows by Corollary~\ref{cor:even-paving}. 
\end{proof}

\begin{figure}[ht]
\begin{center}
\vspace{5pt}
\begin{tabular}{ccc}
\setlength{\unitlength}{1.4pt} 
\begin{picture}(60,50)(0,7) 

\linethickness{1.2pt} 

\qbezier(27.5,0)(60,0)(60,30)
\qbezier(27.5,0)(-5,0)(-5,30)
\qbezier(27.5,60)(-5,60)(-5,30)
\qbezier(27.5,60)(60,60)(60,30)

\qbezier(10,30)(10,24)(20,24)
\qbezier(30,30)(30,24)(20,24)
\qbezier(10,30)(10,36)(20,36)
\qbezier(30,30)(30,36)(20,36)

\thinlines

\put(45,30){\line(1,0){15}}

\qbezier(10,30)(10,20)(27.5,20)
\qbezier(45,30)(45,20)(27.5,20)
\qbezier(10,30)(10,40)(27.5,40)
\qbezier(45,30)(45,40)(27.5,40)

\qbezier(10,30)(10,10)(35,10)
\qbezier(60,30)(60,10)(35,10)
\qbezier(10,30)(10,50)(35,50)
\qbezier(60,30)(60,50)(35,50)

\put(60,30){\circle*{2}}
\put(10,30){\circle*{2}}
\put(45,30){\circle*{2}}

\put(25,7){\makebox(0,0){$c$}}
\put(25,53){\makebox(0,0){$a$}}
\put(53,34){\makebox(0,0){$b$}}
\put(37,17.5){\makebox(0,0){$e$}}
\put(37,43.5){\makebox(0,0){$d$}}

\end{picture}
&
\hspace{.5in}
&
\begin{tabular}{c}
\ \\[-1in]
\begin{tikzcd}[arrows={-stealth}, sep=small, cramped, sep=20]
  a\ar[rr] \ar[rd] &    & c\ar[dd] \\[-3pt]
  & b \ar[ur] \ar[dl] & \\[-3pt]
  d  \ar[uu] \ar[rr] & & e \ar[ul]
\end{tikzcd}
\end{tabular}
\end{tabular}
\end{center}
\vspace{10pt}
\caption{A triangulation of a once-punctured annulus and the corresponding quiver.}
\label{fig:Q-punctured-annulus}
\vspace{-12pt}
\end{figure}

\begin{remark}
Proposition~\ref{pr:punctured-annulus} can be easily extended to a large class of bordered surfaces with 
sufficiently many ``features'' (holes, punctures, and/or handles) by making additional cuts and
invoking the fact that the existence of a totally proper cyclic ordering is a hereditary property. 
In this way, we can for example show that the quivers coming from 
a disk with $\ge 3$ punctures, a torus with $\ge 2$ punctures,
or a sphere with $\ge 5$ punctures are not totally proper. 

On the other hand, many surfaces with a small number of ``features'' give rise to 
totally proper COQs. 
A~couple of such examples are shown in Figure~\ref{fig:4 punctured sphere}. 
The corresponding mutation classes contain four (resp.,~one) 
non-isomorphic quivers, so verification of total properness is straightforward. 
\end{remark}

\vspace{-5pt}

\begin{figure}[ht]
{
\newcommand{\radius}{1.4cm} 
\newcommand{\vertsize}{1.4pt}
\newcommand{\smoltxt}[1]{{\scriptstyle #1}}
\begin{tikzpicture}
\filldraw[black] (0,0)++(135:\radius) circle (\vertsize) node[above left=-1pt] {$a$} coordinate (a);
\filldraw[black] (0,0)++(135-60:\radius) circle (\vertsize) node[above right=-1pt] {$b$} coordinate (b);
\filldraw[black] (0,0)++(135-120:\radius) circle (\vertsize) node[ right] {$c$} coordinate (c);
\filldraw[black] (0,0)++(135-180:\radius) circle (\vertsize) node[below right=-1pt] {$d$} coordinate (d);
\filldraw[black] (0,0)++(135-240:\radius) circle (\vertsize) node[below left=-1pt] {$e$} coordinate (e);
\filldraw[black] (0,0)++(135-300:\radius) circle (\vertsize) node[ left] {$f$} coordinate (f);

\draw[black, dashed, decoration={markings, mark=at position .95 with {\arrow{<}}}, postaction={decorate}] (0,0) circle (\radius);

\draw[black, -{stealth}, shorten >=3pt, shorten <= 3pt] (a) -- (b);
\draw[black, -{stealth}, shorten >=3pt, shorten <= 3pt] (b) -- (d);

\draw[black, -{stealth}, shorten >=3pt, shorten <= 3pt] (a) -- (c);
\draw[black, -{stealth}, shorten >=3pt, shorten <= 3pt] (c) -- (d);

\draw[black, -{stealth}, shorten >=3pt, shorten <= 3pt] (d) -- (e);
\draw[black, -{stealth}, shorten >=3pt, shorten <= 3pt] (e) -- (a);

\draw[black, -{stealth}, shorten >=3pt, shorten <= 3pt] (d) -- (f);
\draw[black, -{stealth}, shorten >=3pt, shorten <= 3pt] (f) -- (a);
\end{tikzpicture}
} 
\qquad
{
\newcommand{\radius}{1.4cm} 
\newcommand{\vertsize}{1.4pt}
\newcommand{\smoltxt}[1]{{\scriptstyle #1}}
\begin{tikzpicture}
\filldraw[black] (0,0)++(90:\radius) circle (\vertsize) node[above] {$a$} coordinate (a);
\filldraw[black] (0,0)++(0:\radius) circle (\vertsize) node[right] {$b$} coordinate (b);
\filldraw[black] (0,0)++(-90:\radius) circle (\vertsize) node[below] {$c$} coordinate (c);
\filldraw[black] (0,0)++(-180:\radius) circle (\vertsize) node[left] {$d$} coordinate (d);

\draw[black, dashed, decoration={markings, mark=at position .1 with {\arrow{<}}}, postaction={decorate}] (0,0) circle (\radius);

\draw[black, -{stealth}, shorten >=3pt, shorten <= 3pt] (a) -- (b);
\draw[black, -{stealth}, double, shorten >=3pt, shorten <= 3pt] (b) -- (c);
\draw[black, -{stealth}, shorten >=3pt, shorten <= 3pt] (c) -- (d);
\draw[black, -{stealth}, shorten >=3pt, shorten <= 3pt] (d) -- (a);

\draw[black, -{stealth}, shorten >=3pt, shorten <= 3pt] (c) -- (a);
\draw[black, -{stealth}, shorten >=3pt, shorten <= 3pt] (d) -- (b);
\end{tikzpicture}
} 
\caption{
Totally proper COQs whose quivers arise from triangulations of a $4$-punctured sphere (on the left)
and a one-holed torus (on the right). 
}
\vspace{-10pt}
\label{fig:4 punctured sphere}
\end{figure}

We conclude this section by looking at exceptional quivers of finite mutation type. 

\begin{proposition}
\label{pr:exceptional-11}
Quivers of finite mutation types $E_k$, $E_k^{(1)}$, $E_k^{(1,1)}$, for $k=6,7,8$, are totally proper. 
Quivers of type $X_6$ or $X_7$ are not totally proper. 
\end{proposition}

\begin{proof}
The first statement can be verified on a computer. To prove the second statement, consider the quiver ~$Q$ 
with 5 vertices $a, b, c, d, e$ and 8 arrows $a\to c \to b\stackrel{2}{\longrightarrow} a$, $e\to c\to d\stackrel{2}{\longrightarrow} e$. 
This quiver comes from a triangulation of a once-punctured annulus; in particular, mutating at $c$ and then at $a$ produces the quiver from Figure~\ref{fig:Q-punctured-annulus}. 
It~then follows from Proposition~\ref{pr:punctured-annulus} that $Q$ is not totally proper.
It remains to note that $Q$ is a full subquiver of a quiver of type $X_6$ (hence of a quiver of type~$X_7$). 
\end{proof}

\newpage 

\section{Signed braid group action on upper triangular matrices} 
\label{sec:bondal}

In this section, we discuss a connection between our theory of cyclically ordered quivers 
and the well-known action of the braid group (and of a larger ``signed braid group'') 
on unipotent upper-triangular integer matrices. 
This action goes back to the work of A.~N.~Rudakov~\cite{Rudakov},
B.~Dubrovin~\cite[App.~F]{Dub},
S.~Cecotti--C.~Vafa~\cite[p.~605]{CV},
and A.~I.~Bondal--A.~E.~Polishchuk~\cite{BP}. 

\medskip

We begin by reviewing the basic construction, following A.~Bondal~\cite[Section~2]{Bondal}. 
For more recent work, see, e.g., \cite{FW} and references therein. 

\begin{definition}
\label{def:Bn-action}
Let $\mathbf{B}_n$ denote the \emph{braid group} on $n$ strands, with Artin generators
$\sigma_1, \ldots, \sigma_{n-1}$. 
Let $\mathcal{U}(n,\ZZ)$ denote the set of all $n\times n$ unipotent upper-triangular matrices 
with integer entries. 
The braid group $\mathbf{B}_n$ acts on the set $\mathcal{U}(n,\ZZ)$ in the following way. \linebreak[3]
For $U=(u_{ij})\in\mathcal{U}(n,\ZZ)$ and $k\in\{1,\dots,n-1\}$, 
define the symmetric matrix $G\in\GL(n,\ZZ)$ (which depends on~$k$ and, importantly, on~$U$) by
\vspace{-1pt}
\begin{equation}
\label{eq:G-sigma}
G=s_k (I - u_{k,k+1} E_{k+1,k}) 
\vspace{-1pt}
\end{equation}
where 
\vspace{-1pt}
\begin{itemize}[leftmargin=.2in]
\item 
$E_{k+1,k}$ is the matrix whose only nonzero entry is a $1$ in row $(k+1)$ and column~$k$; 
\item
$s_k$ is the permutation matrix for the adjacent transposition $(k,k+1)$.
\end{itemize}
We then define the action of the Artin generator $\sigma_k$ on $U$ by 
\vspace{-1pt}
\begin{equation}
\label{eq:sigma_i(U)}
\sigma_k(U) = G  U G^T. 
\end{equation}
It is straightforward to check that $\sigma_k(U)$ is again a unipotent upper triangular matrix
and that the above construction gives an action of the braid group~$\mathbf{B}_n$ on $\mathcal{U}(n,\ZZ)$. 
Put differently, the transformations $\sigma_k$ satisfy the braid relations 
\begin{equation*}
\sigma_i \sigma_{i+1} \sigma_i = \sigma_{i+1} \sigma_i \sigma_{i+1}\,. 
\end{equation*}
\end{definition}

\begin{example}
\label{eg:U4x4}
Let $n=4$. For $k=2$ and 
\begin{equation}
\label{eq:U4u}
U=\begin{bmatrix}
1 & u_{12} & u_{13} & u_{14} \\
0 & 1 & u_{23} & u_{24} \\
0 & 0 & 1 & u_{34} \\
0 & 0 & 0 & 1
\end{bmatrix}\!\!,
\end{equation}
we get
\begin{equation*}
G=\begin{bmatrix}
1 & 0 & 0 & 0 \\
0 & \!\!-u_{2 3}\!\! & 1 & 0 \\
0 & 1 & 0 & 0 \\
0 & 0 & 0 & 1
\end{bmatrix}\!\!,
\end{equation*}
and
\begin{equation}
\label{eq:U'4u}
\sigma_2(U)=GUG^T= 
\begin{bmatrix}
1 & \!\!-u_{12}u_{23}+u_{13}\! & u_{12} & u_{14} \\
0 & 1 & \!-u_{23}\! & \!-u_{23}u_{24}+u_{34} \\
0 & 0 & 1 & u_{24} \\
0 & 0 & 0 & 1
\end{bmatrix}\!\!. 
\end{equation}
\end{example}

\begin{remark}
\label{rem:sigma-inverse}
It is easy to see that the inverse of the Artin generator $\sigma_k$ acts by
\begin{equation}
\label{eq:sigmainverse_i(U)}
\sigma^{-1}_k(U) = H  U H^T, 
\end{equation}
where the symmetric matrix $H\in\GL(n,\ZZ)$ is defined by 
\begin{equation}
\label{eq:H-sigma}
H=s_k (I - u_{k,k+1} E_{k,k+1}) ,
\vspace{-1pt}
\end{equation}
cf.\ \eqref{eq:G-sigma}--\eqref{eq:sigma_i(U)}. 
To illustrate, in Example~\ref{eg:U4x4} we would get
\begin{align}
\nonumber
H&=\begin{bmatrix}
1 & 0 & 0 & 0 \\
0 & 0 & 1 & 0 \\
0 & 1 & \!\!-u_{2 3}\!\! & 0 \\
0 & 0 & 0 & 1
\end{bmatrix}\!\!,
\\[.05in]
\label{eq:sigmaiverse2U4u}
\sigma_2^{-1}(U)=HUH^T&= 
\begin{bmatrix}
1 & u_{13} & \!\!-u_{13}u_{23}+u_{12}\! & u_{14} \\
0 & 1 & \!-u_{23}\! & u_{34} \\
0 & 0 & 1 & \!-u_{23}u_{34}+u_{24} \\
0 & 0 & 0 & 1
\end{bmatrix}\!\!. 
\end{align}
\end{remark}

\hide{
Applying $\sigma_3$, we then obtain:
\begin{equation}
\label{eq:sigma3sigma2}
\sigma_3(\sigma_2(U))=
\begin{bmatrix}
1 & \!\!-u_{12}u_{23}+u_{13}\! & \!\!-u_{12}u_{24}+u_{14}\! & u_{12} \\
0 & 1 & u_{34} & \!-u_{23} \\
0 & 0 & 1 & \!-u_{24} \\
0 & 0 & 0 & 1
\end{bmatrix}\!\!. 
\end{equation}
} 

The action of the braid group~$\mathbf{B}_n$ described above extends to the action of the 
signed braid group:  


\begin{definition}
\label{def:signed-braid-group}
The \emph{signed braid group} $\mathbf{B}_n^\pm$ is the semidirect product
\begin{equation*}
\mathbf{B}_n^\pm = \mathbf{B}_n \rtimes \{\pm1\}^n
\end{equation*}
of the braid group $\mathbf{B}_n$ and the $2^n$-element group $\{\pm1\}^n$, 
the direct product of $n$ copies of the multiplicative two-element group~$\{\pm1\}$. 
For $i\in\{1,\dots,n\}$, we denote by $\rho_i$ the generator $(1,\dots,1,-1,1,\dots,1)\in\{\pm1\}^n$, 
where $-1$ is in position~$i$. 
The generators $\sigma_k$ of $\mathbf{B}_n$ and 
the generators $\rho_i$ of $\{\pm1\}^n$
satisfy the commutation relations
\begin{align}
\label{eq:rho-sigma1}
\rho_k \sigma_j &= \sigma_j \rho_k \ \  \text{if $k\notin \{j,j+1\}$;}\\
\label{eq:rho-sigma2}
\rho_j \sigma_j &= \sigma_j \rho_{j+1};\\
\label{eq:rho-sigma3}
\rho_{j+1} \sigma_j &= \sigma_j \rho_j\,. 
\end{align}

To extend the action of the braid group~$\mathbf{B}_n$ on $\mathcal{U}(n,\ZZ)$ 
described in Definition~\ref{def:Bn-action} to the action of~$\mathbf{B}_n^\pm$, 
we let each 
element $\varepsilon=(\varepsilon_1,\dots,\varepsilon_n)\in\{\pm1\}^n$ 
act by 
$\varepsilon(U)=JUJ^T$, 
where the diagonal matrix $J\in\GL(n,\ZZ)$ is given by 
\begin{equation*}
J=
\begin{bmatrix}
\varepsilon_1 & 0 & \cdots & 0 \\
0 & \varepsilon_2 & \cdots & 0 \\
\vdots & \vdots & \ddots & \vdots \\
0 & 0 & \cdots & \varepsilon_n
\end{bmatrix}. 
\end{equation*}
In particular, each generator $\rho_i$ acts by changing the signs of all non-diagonal elements located 
in row~$i$ or in column~$i$. 
The relations \eqref{eq:rho-sigma1}--\eqref{eq:rho-sigma3} are easily checked. 
\end{definition}

It is immediate from the above definitions that the congruence class of an upper-tri\-angular integer matrix~$U$---hence 
the conjugacy class of its cosquare, the corresponding characteristic/Alexander polynomial, etc.---are 
preserved by the $\mathbf{B}_n^\pm$-action.
Put differently, the intersection of each congruence class in $\GL(n,\ZZ)$ 
with the set $\mathcal{U}(n,\ZZ)$ of unipotent upper-triangular integer matrices
is a disjoint union of \hbox{$\mathbf{B}_n^\pm$-orbits}. 




\pagebreak[3]

\section{Signed braid group action on ordered quivers}

Since unipotent upper-triangular integer $n\times n$ matrices with integer entries encode 
quivers on a linearly ordered $n$-element set of vertices,
one could interpret the action of the braid group $\mathbf{B}_n$ on $\mathcal{U}(n,\ZZ)$ as an action on $n$-vertex quivers. 
It~turns out that this na\"ive translation from the language of matrices to the language of quivers is not the ``right'' one.
Instead, we propose the following setup. 

\pagebreak[3]

\begin{definition}
Fix an $n$-element set $V=\{v_1,\dots,v_n\}$.
Let $\LOQ(n)$ denote the~set of all pairs $(Q,\tau)$ where 
$Q$ is a quiver with the vertex set~$V$ and $\tau$ is a linear ordering of~$V$. 
Thus $\LOQ(n)$ is the set of all \emph{linearly ordered quivers} on the vertex set~$V$, cf.\ Definition~\ref{def:cyclic ordering}. 

We note that $(Q,\tau)$ contains the same information as $(U,\tau)$, 
where $U$ denotes the unipotent companion of $Q$ with respect to~$\tau$. 
To rephrase, linearly ordered quivers are crypto\-morphic to unipotent upper-triangular integer matrices
whose rows and columns are labeled by a permutation of the $n$-element set~$V$. 
(The same permutation is used for the rows and for the columns.)
\end{definition}

We next upgrade the action of the braid group $\mathbf{B}_n$ on 
upper-triangular matrices to its action on the set $\LOQ(n)$ of linearly ordered quivers.

\begin{definition}
\label{def:Bn-acts-on-COQ}
Let ${(Q,\tau)\in\LOQ(n)}$, and let $U$ be the unipotent companion of~$Q$ with respect to the linear ordering~$\tau$.
We define the action of a generator $\sigma_k\in\mathbf{B}_n$ on $(Q,\tau)$ by 
$\sigma_k(Q,\tau)=(Q',\tau')$, 
where 
\begin{itemize}[leftmargin=.2in]
\item 
the linear ordering $\tau'$ is obtained from $\tau$ by switching the $k$th and $(k+1)$st elements; 
\item
$Q'$ is the quiver whose unipotent companion with respect to~$\tau'$ is $\sigma_k(U)$ (cf.~\eqref{eq:sigma_i(U)}). 
\end{itemize}
Equivalently, $\sigma_k$ sends the pair $(U,\tau)$ to $(\sigma_k(U), \sigma_k(\tau))$
where $\sigma_k$ acts on $U$ as in Definition~\ref{def:Bn-action}
and acts on~$\tau$ by the adjacent transposition~$s_k$. 
Since the map $\sigma_k\mapsto s_k$ extends to a group homomorphism,
this rule gives a group action. 
\end{definition}

\begin{example}
\label{eg:sigma3sigma2-LOQ}
Let $n=4$. Let $\tau=(v_1,v_2,v_3,v_4)$ be the standard linear ordering~on~$V$. 
For simplicity, assume that $u_{ij}<0$ for all $1\le i<j\le 4$. 
The matrix $U$ given by~\eqref{eq:U4u} is the unipotent companion of the linearly ordered quiver
$(Q,\tau)$ shown below: 
\vspace{2pt}
\begin{equation}
\label{eq:Qtau-4vert}
(Q,\tau)\ =\  \begin{tikzcd}[arrows={-stealth}, sep=small, cramped, sep=60]
v_1 \ar[r, "-u_{12}"] \ar[rr, swap, bend right=15, "-u_{13}"] \ar[rrr,  bend right=30, "-u_{14}"] &    
v_2 \ar[r, "-u_{23}"]  \ar[rr, swap, bend right=15, "-u_{24}"]  &    
v_3 \ar[r, "-u_{34}"] &    
v_4
\end{tikzcd}\,.
\vspace{-5pt}
\end{equation}
Applying the Artin generator~$\sigma_2$ to $(U,\tau)$, we get the matrix $U'=\sigma_2(U)$ given by~\eqref{eq:U'4u}
and the linear ordering $\tau'=\sigma_2(\tau)=s_2(\tau)=(v_1,v_3,v_2,v_4)$. 
Translating back into the quiver language, we obtain the following linearly ordered quiver~$(Q',\tau')=\sigma_2(Q,\tau)$: \vspace{2pt}
\begin{equation*}
(Q',\tau')\ =\ \begin{tikzcd}[arrows={-stealth}, sep=small, cramped, sep=60]
v_1 \ar[r, "-u_{13}+u_{12}u_{23}"] \ar[rr, swap, bend right=15, "-u_{12}"] \ar[rrr,  bend right=30, "-u_{14}"] &    
v_3  \ar[rr, swap, bend right=15, "-u_{34}+u_{23}u_{24}"]  &    
v_2 \ar[r, "-u_{24}"] \ar[l, swap, "-u_{23}"] &    
v_4
\end{tikzcd}\,;
\vspace{-5pt}
\end{equation*}
note that we reordered the vertices in accordance with~$\tau'$. 

\pagebreak[3]

If we instead apply the inverse Artin generator $\sigma_2^{-1}$ to~$(Q,\tau)$,
we get the linearly ordered quiver associated with the matrix $\sigma_2^{-1}$
given by~\eqref{eq:sigmaiverse2U4u}:
\vspace{2pt}
\begin{equation}
\label{eq:sigma2inverse(Q,tau)}
\sigma_2^{-1}(Q,\tau)\ =\ \begin{tikzcd}[arrows={-stealth}, sep=small, cramped, sep=60]
v_1 \ar[r, "-u_{13}"] \ar[rr, swap, bend right=15, "-u_{12}+u_{13}u_{23}"] \ar[rrr,  bend right=30, "-u_{14}"] &    
v_3  \ar[rr, swap, bend right=15, "-u_{34}"]  &    
v_2 \ar[r, "-u_{24}+u_{23}u_{34}"] \ar[l, swap, "-u_{23}"] &    
v_4
\end{tikzcd}\,.
\vspace{-5pt}
\end{equation}

\hide{
Applying $\sigma_3$ to~$U'$ yields the matrix $U''=\sigma_3(\sigma_2(U))$ given by~\eqref{eq:sigma3sigma2}; 
accordingly, applying $\sigma_3$ to~$(Q',\tau')$ produces the linearly ordered quiver
$(Q'',\tau'')=\sigma_3(\sigma_2(Q,\tau))$ given by 
\vspace{2pt}
\begin{equation}
\label{eq:sigma3sigma2-LOQ}
(Q'',\tau'')\ =\ \begin{tikzcd}[arrows={-stealth}, sep=small, cramped, sep=60]
v_1 \ar[r, "-u_{13}+u_{12}u_{23}"] \ar[rr, swap, bend right=15, "-u_{14}+u_{12}u_{24}"] \ar[rrr,  bend right=30, "-u_{12}"] &    
v_3  \ar[r, "-u_{34}"] &    
v_4 &    
v_2  \ar[ll, bend left=15, "-u_{23}"] \ar[l, swap, "-u_{24}"] 
\end{tikzcd}\, . 
\vspace{-5pt}
\end{equation}
} 
\end{example}

\begin{remark}
In Example~\ref{eg:sigma3sigma2-LOQ}, as in other computations involving the $\mathbf{B}_n$-action, the assumption $u_{ij}<0$ can be lifted: 
the formulas continue to hold if we allow negative arrow~weights, 
with the convention that notation $a\stackrel{q}{\longrightarrow}b$, with $q<0$,
should be interpreted as saying that the quiver contains $-q$ arrows directed from $b$ to~$a$ (and no arrows $a\to b$).
\end{remark} 

The $\mathbf{B}_n$-action from Definition~\ref{def:Bn-acts-on-COQ}
extends in a natural way to a $\mathbf{B}_n^\pm$-action on $\LOQ(n)$
that corresponds to the action of $\mathbf{B}_n^\pm$
on $\mathcal{U}(n,\ZZ)$ described in Definition~\ref{def:signed-braid-group}: 

\begin{definition}
\label{def:quiver-reversal}
Let $(Q,\tau)\in\LOQ(n)$ be a linearly ordered quiver. 
A generator $\rho_i\in\{\pm1\}^n\subset \mathbf{B}_n^\pm$ 
acts on~$(Q,\tau)$ by $\rho_i(Q,\tau)=(Q',\tau)$, where the quiver~$Q'$ is obtained from $Q$ by
reversing all arrows incident to the $i$th vertex in the linear ordering~$\tau$. 
The~ordering $\tau$ does not change.

Together with the action of $\mathbf{B}_n$ described in Definition~\ref{def:Bn-acts-on-COQ},
this gives an action of the signed braid group $\mathbf{B}_n^\pm$ on $\LOQ(n)$. 

We call the transformation $\rho_i:(Q,\tau)\mapsto(Q',\tau)$ the \emph{reversal} at~$i$. 
Note that the construction of the quiver $Q'$ depends on both~$Q$ and~$\tau$. 

More generally, for a subset $S = \{ v_{s_1}, v_{s_2}, \dots \}\subset\{1,\dots,n\}$, 
the product of commuting reversals $\rho_S=\rho_{s_1}  \rho_{s_2}  \cdots\in\{\pm1\}^n$
acts by reversing all arrows between the vertices in positions~$s\in S$ (with respect to~$\tau$) 
and the vertices in positions $\bar s\in\bar S=\{1,\dots,n\}-S$. 
\end{definition}

\begin{remark} 
\label{rem:full-reversal}
It is immediate from Definition~\ref{def:quiver-reversal} that $\rho_S$ and $\rho_{\bar S}$ act in the same way.
Put differently, the product of all reversals $\rho_{1} \cdots\rho_{n}$ acts trivially on $\LOQ(n)$. 
\end{remark}

\begin{example}
Applying the reversals $\rho_1$ and $\rho_4$ to the linearly ordered quiver $(Q,\tau)\in\LOQ(4)$ 
in~\eqref{eq:Qtau-4vert}, we obtain:
\vspace{2pt}
\begin{equation*}
\rho_{1,4}(Q,\tau) = \rho_{2,3}(Q,\tau)
\ =\  \begin{tikzcd}[arrows={-stealth}, sep=small, cramped, sep=60]
v_1  \ar[rrr,  bend right=30, "-u_{14}"] &    
v_2 \ar[l, swap, "-u_{12}"] \ar[r, "-u_{23}"]    &    
v_3  \ar[ll, bend left=15, "-u_{13}"] &    
v_4 \ar[ll, bend left=15, "-u_{24}"] \ar[l, swap, "-u_{34}"]
\end{tikzcd}\,.
\vspace{-6pt}
\end{equation*}
\end{example}

The following result will allow us to push the $\mathbf{B}_n^\pm$-action down to the level of~COQs: 

\begin{proposition}
\label{pr:B-orbit-of-COQ}
Let $(Q,\tau)\in\LOQ(n)$ be a linearly ordered quiver. 
Suppose that a linear ordering $\tau'$ is obtained from $\tau$ by a cyclic shift. 
Then 
$(Q,\tau)$ and $(Q,\tau')$ lie in the same $\mathbf{B}_n$-orbit, and consequently in the same $\mathbf{B}_n^\pm$-orbit.
\end{proposition}

\begin{proof}
It suffices to treat the case where $\tau=(v_1<\dots<v_n)$ and $\tau'=(v_2<\dots<v_n<v_1)$. 
Let $U$ and $U'$ be the corresponding matrices in~$\mathcal{U}(n,\ZZ)$. 
It is not hard to see (cf.\ Proposition~\ref{prop:UQ cyclic shift congruent}) 
that in this case, 
\begin{equation}
\label{eq:bondal-coxeter}
U'=\sigma_{n-1} \cdots  \sigma_1(U). 
\end{equation}
The claim follows. 
\end{proof}

\begin{definition}
In light of Proposition~\ref{pr:B-orbit-of-COQ}, 
we can define the \emph{$\mathbf{B}_n^\pm$-orbit of a COQ} $(Q,\sigma)$ on an $n$-element vertex set
as the $\mathbf{B}_n^\pm$-orbit of any linearly ordered quiver $(Q,\tau)\in\LOQ(n)$ 
where $\tau$ is compatible with~$\sigma$, cf.\ Definition~\ref{def:cyclic ordering}. 

\end{definition}

\begin{remark}[{cf.\ \cite[Remark~2.3]{Bondal}}]
\label{rem:double-twist}
Equation~\eqref{eq:bondal-coxeter} implies that the central element $(\sigma_{n-1} \cdots  \sigma_1)^n\in\mathbf{B}_n$
(the full twist) acts trivially on $\mathcal{U}(n,\ZZ)$. 
Cf.\ Remark~\ref{rem:full-reversal}. 
\end{remark}

\section{From signed braid group action to proper mutations of COQs}
\label{sec:bondal-to-proper-mut}

We will now outline the connection between the above constructions
and the machinery of proper mutations of cyclically ordered quivers. 

\begin{theorem}
\label{th:braid-vs-COQ-mut}
If two COQs on $n$ vertices are related to each other 
by a sequence of wiggles and proper mutations, 
then they lie in the same $\mathbf{B}_n^\pm$-orbit. 
\end{theorem}

\begin{proof}[Proof (sketch)]
Let $(Q,\tau),(Q',\tau')\in\LOQ(n)$, with $\tau=(v_1<\cdots<v_n)$. 
In view of Proposition~\ref{pr:B-orbit-of-COQ}, it suffices to verify the following statements: 
\begin{itemize}[leftmargin=.2in]
\item 
Suppose that $(Q',\tau')$ is obtained from $(Q,\tau)$ by a wiggle $(v_k v_{k+1})$; 
in other words, $Q'=Q$, $v_k$~is not adjacent to~$v_{k+1}$ in~$Q$,
and $\tau'$ is obtained from~$\tau$ by switching the order of $v_k$ and~$v_{k+1}$. 
Then $(Q',\tau')= \sigma_k(Q,\tau)$. 
\item
Suppose that $(Q',\tau')$ is obtained from $(Q,\tau)$ by a mutation at a sink/source~$v_i$
(with $\tau'=\tau$). 
Then $(Q',\tau') = \rho_i(Q,\tau)$.
\item
Suppose that the vertex $v_k$ is neither a source nor a sink, and that 
$(Q',\tau')$ is obtained from $(Q,\tau)$ by a proper mutation at~$v_k$, in the following sense: 
\begin{itemize}[leftmargin=.2in]
\item[$\scriptstyle\blacktriangleright$]
$Q'=\mu_{v_k}(Q)$; 
\item[$\scriptstyle\blacktriangleright$]
$v_1\in\In(v_k)\subseteq\{v_1,\dots,v_{k-1}\}$ and $\Out(v_k)\subseteq\{v_{k+1},\dots,v_n\}$, 
so that $v_k$ is proper, cf.\ Remark~\ref{rem:In-Out-proper};
\item[$\scriptstyle\blacktriangleright$]
$\tau'=(v_k, v_1, v_2, \dots, v_{k-1}, v_{k+1},  \dots, v_n)$, as in the proof of Theorem~\ref{thm:I-N-action}. 
\end{itemize}
Then
\begin{equation}
\label{eq:U'-proper-mut}
(Q',\tau') = \mu_{v_k}(Q,\tau)=\rho_1 \sigma_1^{-1} \sigma_2^{-1} \cdots \sigma_{k-1}^{-1}( (Q, \tau) ). 
\end{equation}
\end{itemize}
We omit the details of this straightforward calculation. 
\end{proof}

\begin{example}
\label{eg:rho4sigma3sigma2-Qtau}
We continue with Example~\ref{eg:sigma3sigma2-LOQ}. 
From~\eqref{eq:sigma2inverse(Q,tau)}, we obtain:
\vspace{2pt}
\begin{equation*}
\sigma_1^{-1}(\sigma_2^{-1}(Q,\tau))\ =\ \begin{tikzcd}[arrows={-stealth}, sep=small, cramped, sep=60]
v_3 \ar[rrr,  bend right=30, "-u_{34}"] &    
v_1 \ar[l, swap, "-u_{13}"] \ar[r, "-u_{12}"] \ar[rr, swap, bend right=15, "-u_{14}+u_{13}u_{34}"]  &    
v_2 \ar[r, "-u_{24}+u_{23}u_{34}"] \ar[ll, bend left=15, "-u_{23}"]  &    
v_4
\end{tikzcd}\,.
\vspace{-5pt}
\end{equation*}
Applying the reversal $\rho_1$, we obtain:
\vspace{2pt}
\begin{equation*}
\rho_1(\sigma_1^{-1}(\sigma_2^{-1}(Q,\tau)))\ =\ \begin{tikzcd}[arrows={-stealth}, sep=small, cramped, sep=60]
v_3  \ar[r, "-u_{13}"] \ar[rr, swap, bend right=15, "-u_{23}"]  &    
v_1 \ar[r, "-u_{12}"] \ar[rr, swap, bend right=15, "-u_{14}+u_{13}u_{34}"]  &    
v_2 \ar[r, "-u_{24}+u_{23}u_{34}"] &    
v_4 \ar[lll, swap, bend left=30, "-u_{34}"]
\end{tikzcd}\,. 
\vspace{-5pt}
\end{equation*}
Directly computing $\mu_{v_3}(Q,\tau)$, we arrive at the same result:
\begin{equation*}
\mu_{v_3}(Q,\tau) = \rho_1(\sigma_1^{-1}(\sigma_2^{-1}(Q,\tau))). 
\end{equation*}
This agrees with~\eqref{eq:U'-proper-mut}, for $k=3$. 
\end{example}

\hide{
Applying the reversal $\rho_4$ to the linearly ordered quiver $\sigma_3(\sigma_2(Q,\tau))$
given by~\eqref{eq:sigma3sigma2-LOQ},
we obtain the linearly ordered quiver 
\begin{equation}
\label{eq:rho4sigma3sigma2-LOQ}
\rho_4(\sigma_3(\sigma_2(Q,\tau)))\ =\ 
\begin{tikzcd}[arrows={-stealth}, sep=small, cramped, sep=60]
v_1 \ar[r, "-u_{13}+u_{12}u_{23}"] \ar[rr, swap, bend right=15, "-u_{14}+u_{12}u_{24}"] &    
v_3  \ar[r, "-u_{34}"]  \ar[rr, swap, bend right=15, "-u_{23}"] &    
v_4 \ar[r, "-u_{24}"] & 
v_2 \ar[lll, swap, bend left=30, "-u_{12}"] 
\end{tikzcd}\ . 
\vspace{-5pt}
\end{equation}
We note that the quiver in~\eqref{eq:rho4sigma3sigma2-LOQ} 
is nothing but $\mu_2(Q)$ (if we ignore the linear ordering). 
} 

Theorem~\ref{th:braid-vs-COQ-mut} asserts that 
the proper mutation class of a COQ is contained in its $\mathbf{B}_n^\pm$-orbit.
The opposite inclusion is usually false.  
On the other hand, the following statement holds: 

\begin{corollary}
\label{cor:braid-orbit-via-mut}
The $\mathbf{B}_n^\pm$-orbit of a COQ~$Q$
consists of all COQs that can be obtained from $Q$ by a sequence of proper mutations,  
reversals, and wiggles. 
\end{corollary}

The proof of Corollary~\ref{cor:braid-orbit-via-mut} will rely on some preliminary work. 

\begin{definition}
\label{def:proper-mu-LOQ}
Let $(Q,\tau)$ and $(Q',\tau')$ be linearly ordered quivers. 
We say that these quivers are obtained from each other by a \emph{proper mutation} 
if the corresponding cyclically ordered quivers $(Q,\sigma)$ and $(Q',\sigma')$ are related to each other by 
a proper mutation of COQs. 
\end{definition}

\begin{proposition}
\label{pr:bondal is mutation reversal}
Let $(Q,\tau)\in\LOQ(n)$ and $(Q',\tau')=\sigma^{\pm1}_k(Q,\tau)$. 
Then $(Q',\tau')$ can be obtained from~$(Q,\tau)$ by applying proper mutations, reversals, cyclic shifts, and/or wiggles. 
\end{proposition}

\begin{proof}
We first note that if $(Q',\tau')$ can be obtained from~$(Q,\tau)$ via 
proper mutations, reversals, cyclic shifts, and/or wiggles,
then the same is true with $(Q,\tau)$ and $(Q',\tau')$ interchanged. 
It therefore suffices to treat the case $(Q',\tau')=\sigma^{-1}_k(Q,\tau)$. 

Let $\tau=(v_1<\dots<v_n)$ and $V=\{v_1,\dots,v_n\}$. 
If the vertices $v_k$ and $v_{k+1}$ are not adjacent in~$Q$ 
(i.e., there are no arrows $v_k\to v_{k+1}$ or ${v_k\leftarrow v_{k+1})}$, 
then $(Q',\tau')$ is obtained from~$(Q,\tau)$ by the wiggle $(v_k, v_{k+1})$, and we are done. 

Let $v_k$ and $v_{k+1}$ be adjacent in~$Q$. 
Applying cyclic shifts if necessary and relabeling the vertices accordingly, 
we may assume without loss of generality that $k=1$.
Pick $S\subset\{1,\dots,n\}$ so that $v_2$ is a proper vertex in $\rho_S(Q,\tau)$. 
(Thus $\In_{\rho_S(Q)}(v_2) = \{v_{1}\}$.) 
We can now invoke~\eqref{eq:U'-proper-mut}, obtaining  
\begin{equation}
\label{eq:mu-vk-rhoS}
\mu_{v_2} (\rho_S(Q,\tau)) = \rho_{1}(\sigma^{-1}_1(\rho_S(Q,\tau))).
\end{equation}
Denote $T=s_1(S)$, where $s_1$ is the adjacent transposition~$(1,2)$.
We then get 
\begin{equation*}
(Q',\tau')
=\sigma^{-1}_1(Q,\tau) 
\stackrel{\eqref{eq:rho-sigma1}-\eqref{eq:rho-sigma3}}{=\joinrel=\joinrel=\joinrel=\joinrel=\joinrel=\joinrel=} \rho_{T}(\sigma^{-1}_1(\rho_S(Q,\tau))) 
\stackrel{\ref{eq:mu-vk-rhoS}}{=\joinrel=\joinrel=} \rho_{T} \rho_{1} ( \mu_{v_2} (\rho_S(Q,\tau))),
\end{equation*}
so $(Q',\tau')$ can be obtained from $(Q,\tau)$ via reversals and a proper mutation. 
\end{proof}

\begin{proof}[Proof of Corollary~\ref{cor:braid-orbit-via-mut}]
Combine Propositions~\ref{pr:B-orbit-of-COQ} and~\ref{pr:bondal is mutation reversal}.
\end{proof}



As mentioned above, a $\mathbf{B}_n^\pm$-orbit typically consists of several proper mutation equivalence classes. 
It is natural to ask whether two \emph{proper} COQs lying in the same \hbox{$\mathbf{B}_n^\pm$-orbit}
are necessarily mutation equivalent. The answer turns out to be positive for 3-vertex quivers
but generally negative for larger quivers, see Theorem~\ref{th:signed-braid-n=3}
and Example~\ref{eg:proper-COQs-congr-but-mut-inequiv}, respectively. 
One possible explanation of this phenomenon may be related to the fact 
that every proper $3$-vertex COQ is \emph{totally proper}, see Proposition~\ref{prop:3-vert proper}.
(Indeed, there is a unique proper cyclic ordering for any complete $3$-vertex quiver.) 
The following problem (cf.\ Problem~\ref{rem:conj-class-for-tp-determines-mut-class}) remains open. 

\begin{problem}
\label{problem:braid-orbit-totally-proper}
Can a $\mathbf{B}_n^\pm$-orbit contain more than one (proper) mutation equivalence class
of totally proper quivers?
To rephrase, if two totally proper COQs lie in the same $\mathbf{B}_n^\pm$-orbit,
does it follow that they are in fact mutation equivalent? 
\end{problem}

\pagebreak[3]

\begin{example}
\label{eg:proper-COQs-congr-but-mut-inequiv}
Let $(Q,\tau),(Q',\tau')\in\LOQ(4)$ be given by
\vspace{2pt}
\begin{equation}
(Q,\tau)\ =\  \begin{tikzcd}[arrows={-stealth}, sep=small, cramped, sep=60]
v_1 \ar[r] \ar[rr, swap, bend right=15, "2"]  &    
v_2    &    
v_3  \ar[r, "2"]&    
v_4 \ar[lll, swap, bend left=30, "2"] \ar[ll, bend left=15, "2"] 
\end{tikzcd}\,, 
\vspace{-15pt}
\end{equation}
\begin{equation}
(Q',\tau')\ =\  \begin{tikzcd}[arrows={-stealth}, sep=small, cramped, sep=60]
v_2 \ar[r]  &    
v_1  \ar[rr, swap, bend right=15, "2"]   &    
v_4 \ar[ll, bend left=15, "2"]  \ar[r, "2"]&    
v_3  
\end{tikzcd}\,.
\end{equation}
Inspection shows that the COQs associated with $(Q,\tau)$ and $Q',\tau')$ are both proper. 

It is straightforward to verify that
\begin{equation}
\label{eq:Q'tau'=...}
(Q',\tau')=\mu_{v_2}(\gamma\,(\rho_1(\rho_4(\mu_{v_3}(Q,\tau))))), 
\end{equation}
where $\gamma$ denotes the cyclic shift that changes a linear order on the vertices as follows:
\begin{equation*}
(a<b<c<d) \mapsto (b<c<d<a). 
\end{equation*}
In view of Proposition~\ref{pr:B-orbit-of-COQ} and Theorem~\ref{th:braid-vs-COQ-mut}, 
formula~\eqref{eq:Q'tau'=...} implies that
$(Q,\tau)$ and $(Q',\tau')$ lie in the same $\mathbf{B}_n^\pm$-orbit. 
We can furthermore use \eqref{eq:bondal-coxeter}, \eqref{eq:U'-proper-mut}, and \eqref{eq:rho-sigma1}--\eqref{eq:rho-sigma3} to transform~\eqref{eq:Q'tau'=...} into the formula 
\begin{equation}
\label{eq:rho1rho3...}
(Q',\tau')=\rho_1(\rho_3(\sigma_1^{-1}(\sigma_3(Q,\tau)))), 
\end{equation}
which directly certifies that 
$(Q',\tau')\in\mathbf{B}_n^{\pm}(Q,\tau)$. 
At the same time, one can show that $Q$ and $Q'$ are not mutation equivalent---hence the corresponding COQs 
do not belong to the same proper~mutation equivalence class.
This can be deduced from G.~Muller's results~\cite{Muller} on reddening sequences, as follows. 
The quiver $Q'$ is acyclic, hence it has a reddening (in fact, a maximal green) sequence,
cf.~\cite[Lemma~2.20]{Brustle-Dupont-Perotin}. 
By~\cite[Corollary~19]{Muller}, any quiver mutation equivalent to~$Q'$ must also have a reddening sequence. 
We now observe that~$Q$ contains the \emph{Markov quiver} as a full subquiver on the vertices $\{v_1, v_3, v_4\}$.
Since the Markov quiver does not have a reddening sequence \cite[Remark~3.4]{Ladkani}, 
$Q$ does not have one either, by \cite[Theorem~17]{Muller},
so $Q$ cannot be mutation equivalent to~$Q'$. 

The unipotent companions~$U$ and $U'$ of $(Q,\tau)$ and $(Q',\tau')$ are given by 
\begin{equation*}
U=\begin{bmatrix}
1 & -1 & -2 & 2\\
0 & 1 & 0 & 2 \\
0 & 0 & 1 & -2\\
0 & 0 & 0 & 1
\end{bmatrix},
\qquad 
U'=\begin{bmatrix}
1 & -1 & 2 & 0\\
0 & 1 & 0 & -2\\
0 & 0 & 1 & -2\\
0 & 0 & 0 & 1
\end{bmatrix}. 
\end{equation*}
Formula \eqref{eq:rho1rho3...} implies that $U$ and $U'$ are congruent over the integers---despite 
the corresponding proper COQs being mutation-inequivalent.
To be concrete, we have $U' = G U G^T$, where
\begin{equation*}
G=\begin{bmatrix}
0 & -1 & 0 & 0\\
1 & 1 & 0 & 0\\
0 & 0 & -2 & -1\\
0 & 0 & 1 & 0
\end{bmatrix}. 
\end{equation*}

We emphasize that neither of the two COQs associated with $(Q,\tau)$ and $(Q',\tau')$ 
is totally proper: $\mu_{v_3}(Q,\tau)$ is not proper at~$v_4$, while
$\mu_{v_1}(Q',\tau')$ is not proper at~$v_2$. 
\end{example}

\begin{theorem}
\label{th:signed-braid-n=3}
Let $Q, Q'$ be proper (hence totally proper) 3-vertex COQs.
The following are equivalent:
\begin{itemize}[leftmargin=.3in]
\item[\rm(i)]
$Q$ and $Q'$ are (proper) mutation equivalent; 
\item[\rm(ii)]
$Q$ and $Q'$ lie in the same signed braid group orbit. 
\end{itemize}
\end{theorem}

The latter orbit is well defined (i.e., it does not depend on the choice of linear ordering compatible with the given cyclic ordering) by Theorem~\ref{pr:B-orbit-of-COQ}. 

The (i)$\Rightarrow$(ii) direction has already been established in Theorem~\ref{th:braid-vs-COQ-mut}.
To prove the implication (ii)$\Rightarrow$(i), we will need to show that if $(Q,\tau),(Q',\tau')\in\LOQ(3)$
lie in the same $\mathbf{B}_3^{\pm}$-orbit, then they are related to each other via proper mutations
(in the sense of Definition~\ref{def:proper-mu-LOQ}) and cyclic shifts. 

\begin{definition}
%
For a linearly ordered quiver $(Q, \tau)\in\LOQ(n)$, 
we will call the set 
\begin{equation*}
\rho(Q,\tau) =\{\rho_S(Q, \tau)\}
\end{equation*}
the \emph{reversal orbit} of~$(Q, \tau)$.
It typically contains $2^{n-1}$ elements, cf.\ Remark~\ref{rem:full-reversal}. 
%
\end{definition}

\begin{lemma}
\label{lem:sigma-rhoQ}
Applying $\sigma_k$ or $\sigma_k^{-1}$ to a reversal orbit of a linearly ordered quiver yields a reversal orbit.
\end{lemma}

\begin{proof}
This is a direct consequence of the commutation relations \eqref{eq:rho-sigma1}--\eqref{eq:rho-sigma3}.
\end{proof}

\begin{lemma}
\label{lem:rho-Q3}
Let $\mathbf{R}$ be a reversal orbit of 3-vertex linearly ordered quivers.
Then exactly one of the following three cases takes place:  
\begin{itemize}[leftmargin=.25in]
\item[\rm(a)] $\mathbf{R}$ is the set of all orientations of a weighted cyclically ordered forest, 
all of them proper (cf.\ Example~\ref{eg:trees-are-proper}) and sink/source mutation equivalent to each other; 
\item[\rm(b)] $\mathbf{R}$ contains one proper non-acyclic quiver and three non-proper acyclic quivers;
\item[\rm(c)] $\mathbf{R}$ contains 
one non-proper non-acyclic quiver and 
three proper complete acyclic quivers that are sink/source mutation equivalent to each other.
\end{itemize}
(Here we use the term ``proper'' in the sense of Definition~\ref{def:proper-mu-LOQ},
i.e., the corresponding COQ must be proper, up to wiggles.) 

Thus $\mathbf{R}$ intersects with exactly one proper mutation equivalence class. 
\end{lemma}

\begin{proof}
This can be verified by straightforward examination of all possible cases. 
\end{proof}

\begin{definition}
For a reversal orbit $\mathbf{R}\subset\LOQ(3)$,
we denote by $\mathbf{Q}(\mathbf{R})$ the unique proper mutation equivalence class that intersects~$\mathbf{R}$,
cf.\ Lemma~\ref{lem:rho-Q3}.
\end{definition}


\begin{lemma}
\label{lem:sigma-reversal-orbit}
Let $\mathbf{R}$ be a reversal orbit of 3-vertex linearly ordered quivers.
For any $k\in\{1,2\}$, we have $\mathbf{Q}(\mathbf{R})=\mathbf{Q}(\sigma_k(\mathbf{R}))$. 
In other words, the (totally) proper quivers in the reversal orbits $\mathbf{R}$ and
$\sigma_k(\mathbf{R})$ are mutation equivalent. 
\end{lemma}

\begin{proof}
We need to find proper representatives in~$\mathbf{R}$
and $\sigma_k(\mathbf{R})$
that are (proper) mutation equivalent to each other. 
Since mutation equivalence is invariant under~cyclic shifts, we may assume that $k=1$. 
We may also assume that our quivers are connected. 

Let us pick $(\bar Q, \bar\tau)\in\sigma_1(\mathbf{R})$  
so that $\bar\tau=(v_1<v_2<v_3)$ and $\bar Q$ contains the arrows $v_1\to v_2\to v_3$,
cf.\ Lemma~\ref{lem:rho-Q3}. 
(This may potentially require a wiggle in the case (a) of Lemma~\ref{lem:rho-Q3}.)
Then we can apply \eqref{eq:U'-proper-mut} with $(Q,\tau)=(\bar Q,\bar\tau)$ to obtain 
\begin{equation*}
\mu_{v_2}(\bar Q,\bar \tau) = \rho_1 \sigma_1^{-1}(\bar Q,\bar\tau) \in \mathbf{R},
\end{equation*}
as desired. 
\end{proof}

\pagebreak[3]

\begin{proof}[Proof of Theorem~\ref{th:signed-braid-n=3}]
Assume that $(Q,\tau),(Q',\tau')\in\LOQ(3)$ are (totally) proper.
Let $\mathbf{R}=\rho(Q,\tau)$ and $\mathbf{R}'=\rho(Q',\tau')$ 
be the corresponding reversal orbits. 
If $(Q,\tau)$ and $(Q',\tau')$
lie in the same $\mathbf{B}_3^{\pm}$-orbit, then we can get from $\mathbf{R}$ to $\mathbf{R}'$
by repeated applications of $\sigma_1^{\pm1}$ and/or $\sigma_2^{\pm1}$. 
Lemma~\ref{lem:sigma-reversal-orbit} implies that $\mathbf{Q}(\mathbf{R})=\mathbf{Q}(\mathbf{R}')$. 
Since $(Q,\tau)\in\mathbf{Q}(\mathbf{R})$ and $(Q',\tau')\in\mathbf{Q}(\mathbf{R'})$, 
we conclude that $(Q,\tau)$ and $(Q',\tau')$ are (proper) mutation equivalent. 
\end{proof}


\begin{corollary}
The signed braid group orbit of a proper 3-vertex COQ $(Q,\sigma)$ is the union
of reversal orbits of all COQs in the proper mutation class of $(Q,\sigma)$. 
\end{corollary}

\begin{proof}
This follows from Theorem~\ref{th:signed-braid-n=3} and the last statement in Lemma~\ref{lem:rho-Q3}. 
\end{proof}


We conclude this section by explaining how another well-known invariant of quiver mutations
fits into the framework of the $\mathbf{B}_n^\pm$ action on $\mathcal{U}(n,\ZZ)$. 

\begin{observation}[{\cite{Seven3x3}
}]
\label{obs:gcd-seven}
The greatest common divisor of the entries in a given row (or column) of the exchange matrix
is a mutation invariant of a labeled quiver. 
Allowing for relabelings, the multiset of these gcd's is a mutation invariant. 
\end{observation}

We will show that this multiset is in fact constant on $\mathbf{B}_n^\pm$-orbits. 

\begin{definition}
\label{def:d(U)}
For a matrix $U=(u_{ij})\in\mathcal{U}(n,\ZZ)$,
let $d_r(U)$ denote the greatest common divisor of all non-diagonal entries in the $r$th row and column: 
\begin{equation*}
d_r(U) = \operatorname{gcd}(u_{1r}, u_{2r}, \ldots, u_{r-1,r}, u_{r,r+1}, \ldots, u_{rn})). 
\end{equation*}
We then denote by $\mathbf{d}(U)$ the multiset $\mathbf{d}(U) = \{ d_1(U), \ldots, d_n(U)\}$.
\end{definition}

\begin{proposition}
\label{pr:gcd invariant of unipotent}
If two matrices $U,U'\in\mathcal{U}(n,\ZZ)$ lie in the same $\mathbf{B}_n^\pm$-orbit, 
then $\mathbf{d}(U) = \mathbf{d}(U')$. 
\end{proposition}

\begin{proof}
It suffices to check the claim in the following two cases: 

\noindent
\textbf{Case~1:} 
$U' = \rho_k(U)$.
Then $u'_{ij} = \pm u_{ij}$ for all $i$ and~$j$, so $\mathbf{d}(U') = \mathbf{d}(U)$.

\noindent
\textbf{Case~2:} 
$U' = \sigma_k(U)$.
We will rely on the following elementary fact. 

\begin{lemma}
\label{lem:gcd}
Let $\mathbf{M}=\{m_1,\dots,m_r\}$ and $\mathbf{M'}=\{m'_1,\dots,m'_r\}$
be two collections of integers. 
Suppose that for some $j\in\{1,\dots,r\}$, we have $m_j=\pm m_j'$
and, for any $i\ne j$, we have $m_j \mid (m_i'-m_i)$.
(In other words, $m_i'$ is obtained from $m_i$ by adding a number divisible by~$m_j$.)
Then $\operatorname{gcd}(\mathbf{M})=\operatorname{gcd}(\mathbf{M'})$. 
\end{lemma}

\enlargethispage{5pt}

To complete the proof of Proposition~\ref{pr:gcd invariant of unipotent},
it suffices to verify, using Lemma~\ref{lem:gcd}, that 
\begin{equation*}
d_i(U') = 
\begin{cases}
d_i(U) & \text{if $i\notin\{k,k+1\}$;} \\
d_{k+1}(U) & \text{if $i=k$;} \\
d_k(U) & \text{if $i=k+1$.} 
\end{cases}
\end{equation*}
We omit the details. 
\end{proof}

\begin{example}
In the case of Example~\ref{eg:U4x4} (i.e., $n=4$, $k=2$, $U'=\sigma_2(U)$), 
we get:
\begin{equation*}
\begin{array}{ll}
d_1(U)\!=\!\operatorname{gcd}(u_{12}, u_{13}, u_{14});  
	& d_1(U')\!=\!\operatorname{gcd}(-u_{12}u_{23}+u_{13}, u_{12} , u_{14} ) \!=\! d_1(U); \\[3pt]
d_2(U)\!=\!\operatorname{gcd}(u_{12}, u_{23}, u_{24}); 
	& d_2(U') \!=\!\operatorname{gcd}(-u_{12}u_{23}+u_{13}, u_{23}, -u_{23}u_{24}+u_{34})\!=\!d_3(U);\\[3pt]
d_3(U)\!=\!\operatorname{gcd}(u_{13}, u_{23}, u_{34}); 
	& d_3(U')\!=\!\operatorname{gcd}(u_{12}, -u_{23}, u_{24}) \!=\! d_2(U). 
\end{array}
\end{equation*}
\end{example}


\newpage 

\end{document}